\documentclass{amsart}
\usepackage[utf8]{inputenc}
\usepackage{mathrsfs}
\usepackage{amsmath,amssymb,amsthm}
\usepackage{subcaption}
\usepackage{tikz-cd}
\usepackage{pgf,tikz,pgfplots}
\pgfplotsset{compat=1.14}
\usetikzlibrary{arrows}
\usetikzlibrary{patterns}
\usepackage{mathtools,eucal}
\usepackage{anyfontsize}

\usepackage[margin=1in]{geometry}

\usepackage[pdfauthor={Suraj Krishna},pdftitle={Immersed cycles and the JSJ decomposition},pdfkeywords={CAT(0) square complexes, VH complexes, one-ended, graph of free groups, JSJ, cyclic splittings},pdftex]{hyperref}
\hypersetup{colorlinks,%
citecolor=black,%
filecolor=black,%
linkcolor=blue,%
urlcolor=black,%
}
\usepackage[capitalize,nameinlink,noabbrev]{cleveref}

\newtheorem{theorem}{Theorem}[section]
\newtheorem{prop}[theorem]{Proposition}
\newtheorem{lemma}[theorem]{Lemma}
\newtheorem{cor}[theorem]{Corollary}

\theoremstyle{definition}
\newtheorem{defn}[theorem]{Definition}
\usepackage{microtype}
\theoremstyle{remark}
\newtheorem{remark}[theorem]{Remark}
\newtheorem{fact}[theorem]{Fact}
\newtheorem{procedure}[theorem]{Procedure}
\newtheorem*{convention}{Convention}

\newcommand{\cat}{\mathrm{CAT(0)}}
\newcommand{\wt}{\widetilde}
\newcommand{\st}{\mathrm{star}}
\newcommand{\VH}{\mathcal{VH}}

\DeclareMathOperator{\link}{link}

\usepackage[normalem]{ulem}

\usepackage{enumitem}
\newlength\mylen
\newlist{mycases}{enumerate}{1}
\setlist[mycases,1]{label=\textit{Case~\arabic*.}, 
 labelwidth=\dimexpr-\mylen-\labelsep\relax,leftmargin=0pt,align=right}

\title{Immersed cycles and the JSJ decomposition}
\author{Suraj Krishna M S}
\address{School of Mathematics, Tata Institute of Fundamental Research, Mumbai 400005, India}

\email{suraj@math.tifr.res.in}

\begin{document}

\maketitle

\begin{abstract}
We present an algorithm to construct the JSJ decomposition of one-ended hyperbolic groups which are fundamental groups of graphs of free groups with cyclic edge groups. 
Our algorithm runs in double exponential time, and is the first algorithm on JSJ decompositions to have an explicit time bound. Our methods are combinatorial/geometric and rely on analysing properties of immersed cycles in certain CAT(0) square complexes. 
\end{abstract}

\section{Introduction}
In \cite{sela_jsj}, Sela showed the existence of a canonical decomposition of a torsion-free one-ended hyperbolic group over its infinite cyclic subgroups. This decomposition, which Sela called a \emph{JSJ decomposition} (see \cref{defn_jsj}), is a generalization to group theory of JSJ decompositions of 3-manifolds (due to Jaco-Shalen \cite{jaco_shalen_jsj} and Johannson \cite{johannson_jsj}). 

In the current article, we present an algorithm (\cref{thm_jsj_general_intro}) to construct the JSJ decomposition of the fundamental group $G$ of a graph of free groups with cyclic edge groups when $G$ is one-ended and hyperbolic. We develop this algorithm by first obtaining one such algorithm in a special case (\cref{thm_main_jsj_intro}): when $G$ is the fundamental group of a compact nonpositively curved square complex called a \emph{tubular graph of graphs} (introduced in \cite{me_grushko}).

A tubular graph of graphs (see \cref{defn_tubular_graph_graphs} for the precise definition) is a square complex obtained by attaching finitely many tubes (a \emph{tube} is a Cartesian product of the unit interval and a circle) to a finite collection of finite graphs. Tubular graphs of graphs are thus nonpositively curved $\VH$-complexes (introduced by Wise in \cite{wisephd}) whose vertical hyperplanes are circles (see \cref{section_setup}). 
A typical example of the fundamental group of a tubular graph of graphs is the amalgamated product of two free groups over cyclic subgroups. 

As an application of \cref{thm_main_jsj_intro}, we obtain an algorithm that takes a finite rank free group $F$ and a finite family of cyclic subgroups $\mathcal{H}$ such that $F$ is freely indecomposable relative to $\mathcal{H}$ as input and constructs the JSJ decomposition of $F$ relative to $H$ in double exponential time (\cref{thm_relative_jsj_free_group}). 

A consequence of our main result is a double-exponential time solution to the isomorphism problem for graphs of free groups with cyclic edge groups in the hyperbolic case (\cref{thm_isomorphism}). Recall that the isomorphism problem is the algorithmic problem of deciding whether two finite presentations of groups present isomorphic groups \cite{dehn}.

\subsection{JSJ decompositions}
We adopt Sela's terminology \cite{sela_jsj}. Let $G$ be a torsion-free hyperbolic group. A \emph{hanging surface subgroup} $G'$ of $G$ is a subgroup isomorphic to the fundamental group of a surface with boundary such that there exists a graph of groups decomposition of $G$ in which $G'$ is a vertex group whose incident edge groups are precisely the peripheral subgroups of $G'$. A \emph{maximal hanging surface subgroup} is a hanging surface subgroup that is not properly contained in any hanging surface subgroup.
A non-cyclic vertex group $G'$ of $G$ is \emph{rigid} if it is elliptic in every cyclic splitting of $G$.
A subgroup is \emph{full} (in the sense of Bowditch \cite{bowditch_jsj}) if it is not properly contained as a finite index subgroup in any subgroup of $G$. 

We can now define JSJ decompositions in the sense of Sela (\cite{sela_jsj}), modified by Bowditch \cite{bowditch_jsj} (see Theorem 0.1 and Theorem 5.28 of \cite{bowditch_jsj} for our formulation).

\begin{defn}[JSJ decomposition] \label{defn_jsj}
\textit{Let $G$ be a torsion-free hyperbolic group. A \emph{JSJ splitting} of $G$ is a finite graph of groups decomposition of $G$ where each edge group is cyclic and each vertex group is full and of one of the following three types: 
\begin{enumerate}
\item a cyclic subgroup,
\item a maximal hanging surface subgroup, or
\item a rigid subgroup.
\end{enumerate}
If a vertex $v$ of type (1) has valence one, then the incident edge group does not surject onto the vertex group $G_v$. Moreover, exactly one endpoint of any edge is of type (1) and the edge groups that connect to any vertex group of type (2) are precisely the peripheral subgroups of that group.}
\end{defn}

\begin{theorem}[\cite{sela_jsj}] 
Let $G$ be a torsion-free one-ended hyperbolic group, which is not the fundamental group of a closed surface. Then a JSJ decomposition of $G$ exists and is unique.
\end{theorem}

We are now ready to state our main result.

\begin{theorem}[\cref{thm_algorithm_gen_case}] \label{thm_jsj_general_intro}
There exists an algorithm of double exponential time complexity that takes a graph of free groups with cyclic edge groups with one-ended hyperbolic fundamental group $G$ as input and returns the JSJ decomposition of $G$.
\end{theorem}

As mentioned earlier, the algorithm is a consequence of the theorem below, proving which takes up a major part of the current article.

\begin{theorem}[\cref{thm_main_result_jsj_hyp}] \label{thm_main_jsj_intro}
There exists an algorithm of double exponential time complexity that takes a tubular graph of graphs with one-ended hyperbolic fundamental group $G$ as input and returns a tubular graph of graphs whose graph of groups structure is the JSJ decomposition of $G$.
\end{theorem}

Other authors have obtained algorithms to compute JSJ decompositions of groups under different conditions. In \cite{dahmani_guirardel_hyp_isomorphism}, Dahmani and Guirardel give an algorithm to compute JSJ decompositions of one-ended hyperbolic groups over maximal virtually cyclic subgroups with infinite centre. In \cite{dahmani_touikan_jsj_algorithm}, Dahmani and Touikan give an algorithm to compute JSJ decompositions of torsion-free hyperbolic groups over its cyclic subgroups. In \cite{barrett_jsj}, Barrett gives an algorithm to compute JSJ decompositions of one-ended hyperbolic groups over virtually cyclic subgroups, while Cashen \cite{cashen_relative_jsj}, along with Manning \cite{cashen_manning}, develops an implementable algorithm to construct the relative JSJ of a free group relative to a family of cyclic subgroups. We remark that the time complexity of these algorithms is not known. 

Our approach is combinatorial/geometric.
We will now describe this approach briefly. 

\subsection{Coarse behaviour and Brady-Meier tubular graphs of graphs}

\begin{defn}
A cube complex is \emph{Brady-Meier} if all its vertex links are connected and moreover each vertex link remains connected after removing any simplex in the link.
\end{defn}

\begin{theorem}[\cite{bradymeier}]
The fundamental group of a finite connected Brady-Meier nonpositively curved cube complex is one-ended.
\end{theorem}

\begin{theorem}[\cite{me_grushko}] \label{thm_partial_converse_brady_meier}
There exists an algorithm of polynomial time complexity that takes a tubular graph of graphs with one-ended fundamental group as input and returns a Brady-Meier tubular graph of graphs with isomorphic fundamental group.
\end{theorem}
Thanks to the above result, we work with Brady-Meier tubular graphs of graphs throughout this article. Let $X$ be a Brady-Meier tubular graph of graphs endowed with its $\VH$ structure. Each vertical hyperplane of $X$ is a circle (\cref{prop_tubular_vh}). If $\wt{X}$ denotes the $\cat$ universal cover of $X$, then the vertical hyperplanes of $\wt{X}$ are lines. 
Let $G$ denote the fundamental group of $X$. 
Adopting the terminology of Scott and Wall \cite{scottandwall}, $X$ has a structure of a graph of spaces (see \cref{section_graphs_spaces} for details), where each vertex space is itself a graph. Similarly, $\wt{X}$ has a structure of a tree of spaces, where each vertex space is a (vertical) tree.

In order to construct the JSJ decomposition of $G$, one has to first find cyclic subgroups over which $G$ splits. We address this issue using the Brady-Meier structure of $X$:

A geodesic line $L$ of $\wt{X}$ \emph{separates} $\wt{X}$ if $\wt{X} \setminus L$ is not connected. Two separating geodesic lines $L_1$ and $L_2$ of $\wt{X}$ \emph{cross} if $L_1$ meets two distinct components of $\wt{X} \setminus L_2$ and vice versa.
An \emph{axis} in $\wt{X}$ of an element $g \in G$ is a geodesic line in $\wt{X}$ that is invariant under the action of the cyclic subgroup $<g>$. Given $g \in G$, an axis $L$ of $g$ always exists in $\wt{X}$ \cite{bridsonhaefliger}. 

Suppose that $G$ splits over $<g>$ and that $L$ is contained in a vertical tree of $\wt{X}$. Then an application of a result of Papasoglu \cite{papasoglu_coarse_separation} to
the fact that $\wt{X}$ is a Brady-Meier complex implies that $L$ separates $\wt{X}$ and $L$ does not cross any of its translates (\cref{lemma_algebraic_splittings_and_separating_lines}). Conversely, if $L$ is separating and does not cross any of its translates, then $G$ splits over a subgroup of the stabiliser of $L$ (\cref{prop_cyclic_splittings_induce_separating_lines}). 

The properties of separation and crossing have local characterisations in the Brady-Meier $\wt{X}$. 
Denote by $N_R(L)$ the set of all points in $\wt{X}$ at distance at most $R$ from a point of $L$. 
\begin{lemma}[\cref{lemma_half-spaces_equal_components_reg_sphere_L}]
The line $L$ separates $\wt{X}$ if and only if it separates $N_{\frac{1}{4}}(L)$.
\end{lemma}

\begin{prop}[\cref{prop_local_crossing_lines}]
Two separating lines $L_1$ and $L_2$ cross if and only if 
\begin{enumerate}
\item $L_1 \cap L_2$ is non-empty and compact, and
\item $L_2$ meets two components of $N_{\frac{1}{4}} (L_1 \cap L_2) \setminus L_1$.
\end{enumerate}
\end{prop}

The quotient of $L$ by $<g>$ is an immersed circle $C$, which we call a \emph{cycle}, in $X$. The \emph{regular neighbourhood} of $C$ is the quotient of $N_{\frac{1}{4}}(L)$ by the action of $<g>$. The fact that $L$ separates $N_{\frac{1}{4}}(L)$, along with a condition that is satisfied since $G$ splits over $<g>$ (\cref{def_splitting_cycle}), implies the following result: 
\begin{lemma}[\cref{lemma_necessary_conditions_splitting_cycle}]
$C$ separates its regular neighbourhood and $\wt{C}$ does not cross any of its translates in $\wt{X}$.
\end{lemma}

We need another property to construct the JSJ decomposition. A cyclic subgroup over which $G$ splits is said to be \emph{universally elliptic} if it is elliptic in the Bass-Serre tree of any cyclic splitting of $G$ \cite{guirardel_levitt_jsj}. 
The edge groups of the JSJ decomposition are universally elliptic. 

Let $L_1$ (respectively $L_2$) be an axis of $g_1$ (respectively $g_2$) such that $G$ splits over $<g_1>$ and $<g_2>$. Then $<g_1>$ is elliptic in the Bass-Serre tree of the splitting over $<g_2>$ only if $L_1$ and any translate of $L_2$ don't cross (\cref{lemma_elliptic-elliptic_splittings}).

\subsection{Repetitive cycles and JSJ splittings}
In \cref{section_repetitive}, we introduce an important notion, namely repetitivity, that bounds the length of a cycle that induces a universally elliptic splitting. Let $\wt{C}$ denote a lift of a cycle $C$ in $\wt{X}$.

\begin{defn}[\cref{def_rep_cycle}, \cref{lemma_rep_cycles_lifts}] A cycle $C$ is \emph{$k$-repetitive} if $\wt{C}$ is a separating line and there exists an edge $e$ in $\wt{X}$ and elements $g_1, \cdots, g_k \in G$ such that 
\begin{enumerate}
\item each translate $g_i \wt{C}$ contains $e$,
\item the distance between $e$ and $g_ie$ is strictly less than the length of $C$, and
\item any two squares $\mathsf{s}$ and $\mathsf{s}'$ that contain $e$ are either separated by all translates $g_i\wt{C}$ or by none of them. 
\end{enumerate}
\end{defn}

There are two important reasons for introducing the notion of repetitive cycles. The first reason is that any cycle that is longer than a certain bound is $k$-repetitive (\cref{prop_long_cycles_repetitive}). Here, the bound depends only on $k$ and the number of squares of $X$. The second reason is the following: 
\begin{prop}[\cref{prop_repetitive_not_jsj}]
Let $C$ be a $k$-repetitive cycle with $k \geq 3$. Suppose that $\pi_1(C)$ is a maximal cyclic subgroup of $G$. Then there exists a separating line $L$ in $\wt{X}$ such that $L$ and $\wt{C}$ coarsely cross.
\end{prop}

This implies that $\pi_1(C)$ conjugates into a hanging surface subgroup of the JSJ splitting of $G$, by Proposition 5.30 of \cite{bowditch_jsj}. Hence, $\pi_1(C)$ is not universally elliptic and the length of a cycle which induces a universally elliptic cyclic subgroup is bounded. This leads to the following result: 

\begin{theorem}[\cref{thm_algo_G_hyp}] \label{thm_finite_list_intro}
There exists an algorithm of double exponential time complexity that takes a Brady-Meier tubular graph of graphs with hyperbolic fundamental group $G$ as input and returns a finite list of splitting cycles that contains all universally elliptic subgroups of $G$ up to commensurability.
\end{theorem}

We remark that a similar, but weaker, result is obtained by Cashen and Macura in Proposition 4.11 of \cite{cashen_macura_line_patterns} using a similar idea, where given a free group $F$ and a finite family of cyclic subgroups $\mathcal{H}$, the authors obtain a bound depending on $F$ and $\mathcal{H}$ such that if $F$ admits a cyclic splitting relative to $\mathcal{H}$, then there is a word of length less than the bound over which $F$ splits relative to $\mathcal{H}$.

\subsection{Obtaining a JSJ complex}
In \cref{section_jsj_complex}, we modify the given tubular graph of graphs $X$ to a tubular graph of graphs $X_{jsj}$ such that the fundamental group of the graph of spaces structure of $X_{jsj}$ is the JSJ decomposition of $G$. 

The first step involves a modification of the initial tubular graph of graphs $X$ by cutting along the finite list of cycles supplied by \cref{thm_finite_list_intro}. 
We do this cutting procedure using the machinery of \emph{spaces with walls} (due to Haglund and Paulin \cite{haglund_paulin_walls}). The vertex set of $\wt{X}$ is a space with walls, with walls defined by its vertical and horizontal hyperplanes. We enrich the wall set by adding lifts of cycles supplied by \cref{thm_finite_list_intro} (see \cref{subsection_construction_X'}). 
We then remove tubes which are attached to cyclic vertex graphs on both sides. In \cref{prop_T''_refinement_jsj}, we show that each edge group of the JSJ decomposition of $G$ is a conjugate of an edge group of the underlying graph of groups of the new tubular graph of graphs. Thus, an edge stabiliser of the Bass-Serre tree of the new tubular graph of graphs is either an edge stabiliser of the JSJ tree or a cyclic subgroup that conjugates into a maximal hanging surface subgroup of the JSJ splitting. 
It only remains to identify the maximal surface subgroups that appear as vertex groups in the JSJ decomposition. Once identified, removing tubes corresponding to edge stabilisers which conjugate into maximal hanging surface subgroups gives the JSJ decomposition, proving the main result (\cref{thm_main_jsj_intro}).

\subsection{Identifying surfaces}
We give a criterion to identify surfaces in the Brady-Meier setup. A vertex graph of a tubular graph of graphs is a \emph{surface graph} if the fundamental group of the graph is a surface group whose peripheral subgroups are precisely the incident edge subgroups. Then

\begin{lemma} [\cref{lemma_surface_graphs}]
A vertex graph of a Brady-Meier tubular graph of graphs is a surface graph if and only if every edge of its double is contained in exactly two squares.
\end{lemma}

We refer the reader to \cref{defn_double} for the definition of a double.

\subsection*{Acknowledgements}
I thank my advisors Fr\'{e}d\'{e}ric Haglund and Thomas Delzant for their time and for the many useful discussions. I thank Henry Wilton for the many references. I am grateful to the anonymous referee for the careful reading and helpful comments which improved the article, and especially for motivating \cref{thm_algorithm_gen_case}. Part of this work was done when I was visiting IRMA, Strasbourg. This work was supported by a French public grant for research, as part of the Investissement d’avenir project, reference ANR-11-LABX-0056-LMH, LabEx LMH.

\section{The setup} \label{section_setup}

\subsection{\texorpdfstring{$\VH$}{VH}-complexes}

The notion of $\VH$-complexes was first introduced in \cite{wisephd}.

\begin{defn} A \textit{square complex} is a two dimensional CW complex in which each 2-cell is attached to a combinatorial loop of length 4 and is isometric to the standard Euclidean unit square $I^2 = [0,1]^2$.
\end{defn} 

All our square complexes will be locally finite.

\begin{defn}[Vertex links] Let $v\in X$ be a vertex of a square complex. The link of $v$, denoted by $\link(v)$ is a graph whose vertex set is the set $\{e \mid e \mathrm{\: is \: a \: half\mbox{-}edge \: incident \: to \:} v\}$. The number of edges between two vertices $e, f$ is the number of squares of $X$ in which $e,f$ are adjacent half-edges.
\end{defn}

\begin{defn} A square complex is \emph{nonpositively curved} if the length of a closed path in the link of any of its vertices is at least four.
\end{defn}

By a result of Gromov \cite{gromov_hyperbolic}, a simply connected nonpositively curved square complex is $\cat$ in the metric sense.

\begin{defn}[\cite{sageev}]
Let $X$ be a square complex. A \emph{mid-edge} of a square $\mathsf{s}$ in $X$ is an edge (after subdivision of $\mathsf{s}$) running through the center of $\mathsf{s}$ and parallel to two of the edges of $\mathsf{s}$.
Declare two edges $e$ and $f$ to be equivalent if there exists a sequence $e = e_1, \cdots, e_n = f$ of edges such that $e_i$ and $e_{i+1}$ are opposite edges of some square of $X$.
Given an equivalence class $[e]$ of edges, the \emph{hyperplane dual to $e$}, denoted by $\mathsf{h}_e$, is the collection of mid-edges which intersect edges in $[e]$.
\end{defn}

\begin{defn}[\cite{wisephd}] A \textit{$\VH$-complex} is a square complex in which every 1-cell is labelled as either vertical or horizontal in such a way that each 2-cell is attached to a loop which alternates between horizontal and vertical 1-cells.
\end{defn}

The labelling of the edges of a $\VH$-complex as horizontal and vertical induces a labelling of the vertices in the link of any vertex as horizontal and vertical, thus making the link a bipartite graph. Similarly, the hyperplanes of a $\VH$-complex are also labelled as vertical and horizontal, with a vertical hyperplane being dual to an equivalence class of horizontal edges and a horizontal hyperplane being dual to an equivalence class of vertical edges.

\begin{remark}
Since the link of any vertex of a $\VH$-complex is bipartite, the length of a closed path is even. Thus a $\VH$-complex is nonpositively curved if there exists no bigon in any vertex link. 
\end{remark}

\subsection{Graphs of spaces} \label{section_graphs_spaces}
Graphs of groups are the basic objects of study in Bass-Serre theory \cite{serre}.
They were studied from a topological perspective in \cite{scottandwall} by looking at graphs of spaces instead of graphs of groups. We will adopt this point of view.

\begin{defn}\label{defn_graph_spaces} By a \textit{graph of spaces}, we mean the following data: $\Gamma$ is a connected graph, called the underlying graph. For each vertex $s$ (edge $a$) of $\Gamma$, $X_s$ ($X_a$) is a topological space. Further, whenever $a$ is incident to $s$, $\partial_{a,s} : X_{a} \to X_{s}$ is a $\pi_1$-injective continuous map. The \textit{geometric realisation} of the above graph of spaces is the space $X =( \bigsqcup_{s \in \Gamma^{(0)}} X_{s} \sqcup \bigsqcup_{a \in \Gamma^{(1)}} X_{a} \times [0,1]) /\sim$, where $(x,0)$ and $(x,1)$ are identified respectively with $\partial_{a,s}(x)$ and $\partial_{a,s'}(x)$. Here, $s$ and $s'$ are the two endpoints of $a$.
\end{defn}

Note that the universal cover of $X$ has the structure of a \emph{tree of spaces}, a graph of spaces whose underlying graph is the Bass-Serre tree of the associated graph of groups structure of $X$ \cite{scottandwall}.

\subsection{Tubular graphs of graphs}

\begin{defn} \label{defn_tubular_graph_graphs}
A \textit{tubular graph of graphs} is a finite graph of spaces in which each vertex space is a finite connected simplicial graph and each edge space is a simplicial graph homeomorphic to a circle. Further, the attaching maps are simplicial immersions. We will always assume that the underlying graph is connected.
\end{defn}

As a consequence of the definition, unless the underlying graph is trivial, no vertical graph is a tree. We note that asking for each vertex graph to be simplicial is not a serious restriction as every one dimensional CW complex is a simplicial graph after subdivision. 

We remark that not all graphs of free groups with cyclic edge groups can be realised as tubular graphs of graphs since we require both the images of each edge space to have identical length. This can be achieved, however, if the underlying graph is a tree.

It is easy to see (compare with Theorem 1.18 of \cite{wisephd}) that
\begin{prop} \label{prop_tubular_vh}
The geometric realisation of a tubular graph of graphs is a finite (hence compact), connected nonpositively curved $\VH$-complex whose vertical hyperplanes are circles.
\end{prop}
Indeed, the geometric realisation is $\VH$, where vertical edges are the edges of vertex graphs and horizontal edges are the edges induced by vertices of edge graphs. Note that the squares are obtained from the `tubes' (edge graph times the unit interval, for each edge graph). Then the link of any vertex of a vertex graph does not have bigons as attaching maps of edge spaces are simplicial immersions.

\begin{convention} Throughout this article, we will use the same notation for a tubular graph of graphs and the $\VH$-complex which is its geometric realisation. $X$ will always denote a Brady-Meier tubular graph of graphs with fundamental group $G$. $X_s$ will denote a vertex graph (a component of the vertical 1-skeleton) in $X$ and $\wt{X}$ the $\cat$ universal cover of $X$. Unless mentioned otherwise, we work with the $\cat$ metric in $\wt{X}$.
\end{convention}

\begin{defn} [Thickness] \label{def_thickness}
For an edge $e$ in $X$, the \textit{thickness} of $e$ is the number of squares of $X$ which contain $e$.
\end{defn}

Observe that a horizontal edge of $X$ always has thickness equal to two.
Since $X$ is Brady-Meier, we have
\begin{lemma}[\cite{me_grushko}] \label{lemma_bm_thickness_two}
Every edge of of $X$ has thickness at least two.
\end{lemma}

\begin{defn}[Paths, lines] \label{def_paths_cycles_lines}
Recall that a path in a space $Z$ is a continuous map from a closed interval to $Z$. 

\noindent A \emph{combinatorial path} (see \cite{wise_mccammond_paths} for instance) is a map of graphs $\rho : P \to \Gamma$, where $P$ is a subdivided compact interval and $\Gamma$ is a graph. Further, all our combinatorial paths will be assumed to be immersions of graphs.
$P$ is always assumed to be oriented. When there is no confusion about $\Gamma$, we will refer to $\rho : P \to \Gamma$ as the path $P$. Unless mentioned to the contrary, a path between two vertices of $X$ or $\wt{X}$ is a combinatorial path. Unless mentioned to the contrary, a path between two vertices of $X$ or $\wt{X}$ is a combinatorial path.

\noindent A \emph{segment} is an embedded combinatorial path. Note that any compact graph homeomorphic to an interval is the image of a segment. We will often call such graphs as segments.

\noindent A \emph{cycle} is an immersion of graphs $\phi: C \to \Gamma$, where $C$ is a subdivided circle. We will often denote it by $C$.
 
\noindent A \emph{line} is an isometric embedding $\mathbb{R} \hookrightarrow \wt{X}$ (with the $\cat$ metric), while a ray is an isometric embedding of $[0,\infty)$. A \emph{combinatorial line} is an isometric embedding of graphs $R \to \wt{X}^1$, where $R$ is the real line subdivided at integer intervals. 
We will only consider combinatorial lines that are also lines in the $\cat$ metric (see \cref{rmk_transversal_lines_not_considered} below).
\end{defn}

Since horizontal edges of $\wt{X}$ are of thickness two, vertical hyperplanes of $\wt{X}$ are lines. Further, 

\begin{fact}\label{fact_geodesics_strips}
The first cubical neighbourhood in $\wt{X}$ of a vertical hyperplane $\mathsf{h}$, or the set of all closed squares of $\wt{X}$ that meet $\mathsf{h}$, is convex (\cite{sageev}) and hence isometric to a Euclidean strip $[0,1] \times \mathbb{R}$ with $\mathsf{h} \cong \{\frac{1}{2}\} \times \mathbb{R}$.
Thus maximal geodesics in such a strip are of the form either $\{t_0\} \times \mathbb{R}$ or geodesics from $(0,x)$ to $(1,y)$.
\end{fact}
We next divide the set of lines in $\wt{X}$ into the following three types.
\begin{defn}
A \emph{vertical} line is a combinatorial line contained in a vertical tree. A \emph{tubular} line is one that is parallel to a vertical hyperplane in the first cubical neighbourhood of the hyperplane.
A \emph{transversal} line is a line that hits at least two vertical trees.
\end{defn}

Observe that a given line can be both vertical and tubular. 
We note that a tubular line that is not vertical is disjoint from the vertical 1-skeleton and hits any horizontal edge at most at one point, while a transversal line hits at least one vertical hyperplane (in exactly one point).
\begin{remark}\label{rmk_transversal_lines_not_considered}
As mentioned in the introduction, edge groups of the JSJ decomposition are universally elliptic. By \cref{lemma_transversal_not_universally_elliptic}, transversal lines are not stabilised by universally elliptic subgroups. Hence, only vertical and tubular lines play a role in the analysis that follows. 
\end{remark}

\section{Regular neighbourhoods and regular spheres} \label{section_regular_neighbourhoods}

Recall that a cell of a square complex is either a vertex, an edge or a square.

\begin{defn}[Cubical neighbourhoods] \label{def_cubical_neighbourhoods}
The \emph{first cubical neighbourhood} $Y^{+1}$ of a subset $Y$ of a square complex $Z$ is a subcomplex of $Z$ given by the union of all cells of $Z$ that meet the closure of $Y$. The \emph{$n^{th}$ cubical neighbourhood} $Y^{+n}$ is defined inductively as $(Y^{+(n-1)})^{+1}$. 
\end{defn}

\begin{defn}[Cubical subdivisions]
The \emph{first cubical subdivision} $Z^{(1)}$ of a square complex $Z$ is a square complex obtained by subdividing $Z$ in the following way: Each edge of $Z$ is subdivided into two edges with the midpoint of the initial edge forming a new vertex. Each square of $Z$ is subdivided into four squares of equal area by taking the centre of the square as a new vertex and taking four new edges between the centre of the square and each of the midpoints of the edges of the square.
The \emph{$n^{th}$ cubical subdivision} $Z^{(n)}$ of $Z$ is the first cubical subdivision of $Z^{(n-1)}$.
\end{defn}

We will now define an abstract neighbourhood for a combinatorial path in a square complex. The path may not embed in the square complex, but it will embed in its abstract neighbourhood.
Fix a combinatorial path $\rho$ from $P$ to the 1-skeleton of a square complex $Z$. Here the path may or may not be a cycle. We allow $P$ to be a combinatorial ray or a combinatorial line. We remind the reader that $\rho$ is an immersion of graphs.
We will consider $\rho$ as a map from $P$ to the 1-skeleton of $Z^{(2)}$, the second cubical subdivision of $Z$.

\begin{defn}
The \emph{regular neighbourhood} $N(P)$ of $P$ in a square complex $Z$ is a square complex constructed as follows. Let $\mathsf{c}$ be a cell of $Z^{(2)}$. We take one copy of $\mathsf{c}$ for each component of $\rho^{-1}(\mathsf{c})$ (see \cref{fig_regular_neighbourhood_construction}). The adjacency of cells is given by the adjacency of arcs of $P$, where each arc is a component of the pre-image of a cell of $Z^{(2)}$. 

\begin{figure}
\begin{center}
\begin{tikzpicture}

\draw [red] (-13,1.5) to (-6,1.5);
\node [above] at (-9.5,1.5) {$P$};
\draw [->] (-5,1.5) to (-3,1.5);

\path [fill=white!60!yellow] (-8.5,-1.25) rectangle (-8,-.75);
\path [fill=white!60!yellow] (-4.5,-1.25) rectangle (-5,-.75);
\draw [red] (-10,-1) to (-3,-1);
\draw (-10.2,-1.25) rectangle (-2.8,-.75);
\node [below] at (-6.5,-1.25) {$N(P)$};

\draw [->] (-10,1) to (-8,0);
\draw [->] (-4.5,0) to (-2.5,.7);

\path [fill=white!60!yellow] (-.5,0) -- (0,-.5) -- (.5,0) -- (0,.5);

\draw [red] plot [smooth] coordinates {(-.8,-.8) (0,0) (1.4,1.5) (0,2.75) (-1.4,1.5) (0,0) (.8,-.8)};

\draw plot [smooth] coordinates {(-1.25,-.75) (-0.2,0.3) (1.1,1.5) (0,2.5) (-1.1,1.5) (0.2,0.3) (1.25,-.75)};
\draw plot [smooth] coordinates {(-.75,-1.25) (0.2,-0.3) (1.7,1.5) (0,3) (-1.7,1.5) (-0.2,-0.3) (.75,-1.25)};
\draw (-1.25,-.75) -- (-.75,-1.25);
\draw (.75,-1.25) -- (1.25,-.75);

\node [below] at (0,-1.25) {$Z$};
\end{tikzpicture}
\end{center}
\caption{Two disjoint subpaths of $P$ are mapped to the yellow square} \label{fig_regular_neighbourhood_construction}
\end{figure}
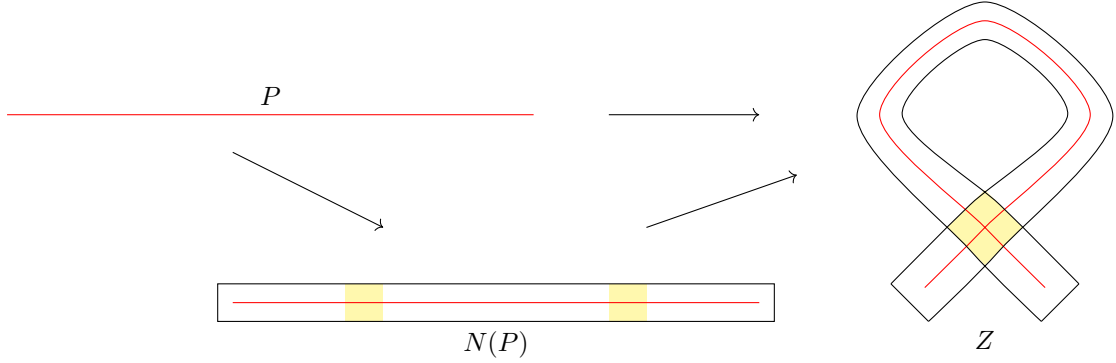
\end{defn}

Since $\rho$ restricted to each arc of the pre-image of a cell of $Z^{(2)}$ is an (isometric) embedding, we observe: 

\begin{fact}\label{fact_N(P)_embedding}
There is a natural embedding of $P$ in $N(P)$ such that $\rho$ factors through this embedding.
\[\begin{tikzcd}
P \arrow{rr}{\rho} \arrow[hookrightarrow]{dr}& &Z \\
& N(P) \arrow{ur} &
\end{tikzcd}
\]

Note that the map $N(P) \to Z$ is an immersion. In particular, if $\rho$ is an embedding, then $N(P)$ embeds in $Z$, since $\rho^{-1}(\mathsf{c})$ of any cell $\mathsf{c}$ contains a single component.
\end{fact}

The reason for choosing the second cubical subdivision instead of the first in the definition of $N(P)$ is to make it easier to define certain operations (see \cref{defn_graph_sum}).

\begin{defn}
The \emph{regular sphere} around $P$, denoted by $\partial N(P)$, is the union of all cells of $N(P)$ that are disjoint from $P$.
\end{defn}

\begin{fact}
The regular sphere around a vertex is isomorphic as graphs to the first barycentric subdivision of the vertex link.
\end{fact}
 
\subsection{The regular sphere around an edge} The goal of this subsection is to show that the regular sphere around an edge of a square complex can be built from the regular spheres around its endpoints.

Let $Y$ be a subset of a simplicial graph $\Gamma$. Recall that the \emph{star of $Y$}, denoted by $\st(Y)$ is the subgraph of $\Gamma$ consisting of all vertices and edges of $\Gamma$ that meet $Y$. The \emph{open star of $Y$}, denoted by $\mathring{\st}(Y)$, is the interior of $\st(Y)$.

\begin{defn}[\cite{manning}] \label{defn_graph_sum}
Let $\Gamma_1$ and $\Gamma_2$ be graphs. Let $ v_1 \in \Gamma_1$ and $ v_2 \in \Gamma_2$ be vertices of equal valence, say $k$. Let $\phi_i : \{1, \cdots, k\} \to \mathrm{adj}(v_i)$ be a labelling of the vertices adjacent to $v_i$, $i = 1,2$. 
Then the \emph{spliced graph} $\Gamma_1 \; {}_{(v_1,\phi_1)} \! \bigoplus_{(v_2,\phi_2)} \Gamma_2$ is defined as a quotient of $\Gamma_1 \setminus \mathring{\st}(v_1) \sqcup \Gamma_2 \setminus \mathring{\st}(v_2)$, where $\phi_1(j)$ is glued to $\phi_2(j)$, for $1 \leq j \leq k$. 
If $v_1 \neq v_2$ are vertices in $\Gamma_1$ as above, then we define the \emph{self-spliced graph} ${}_{(v_1,\phi_1)} \! \bigoplus_{(v_2,\phi_2)} \Gamma_1$ as a quotient of $\Gamma_1 \setminus (\mathring{\st}(v_1) \cup \mathring{\st}(v_2))$, where $\phi_1(j)$ is glued to $\phi_2(j)$, for $1 \leq j \leq k$.
\end{defn} 

Recall that the \emph{dipole graph of order $d$} is a multigraph consisting of two vertices and $d$ edges joining them.

Let $e$ be an edge of a nonpositively curved square complex $Z$ with endpoints $u$ and $v$. We can now state the main result of this subsection.

\begin{lemma}\label{lemma_regular_sphere_at_edge}
The regular sphere around $e$ is homeomorphic to the spliced graph of the regular spheres around $u$ and $v$ with the natural labelling induced by the squares containing $e$.
\end{lemma}

\begin{proof}
Let $m$ be the midpoint of $e$. Then $m$ is a vertex after a subdivision of the square complex. Observe that $\partial N(m)$ is homeomorphic to a dipole graph of order $d$, where $d$ is the thickness of $e$. Let $e_a$ be the initial half-edge of $e$ and $e_b$ its second-half.
Then $e_a$ and $e_b$ meet $\partial N(m)$ at distinct vertices of valence $d$, which we will also call $e_a$ and $e_b$ respectively (see \cref{fig_regular_spheres_adj_vertices}). Thus $\partial N(m) \setminus \mathring{\st}(e_a) \cup \mathring{\st}(e_b)$ is a disjoint union of $d$ segments, one for each square that contains $e$.

\begin{figure}
\begin{center}
\begin{tikzpicture} [scale = 0.5]

\fill [ultra thick, white!80!green] (-1,-1.2) rectangle (1,-0.8);
\fill [ultra thick, white!80!green] (-0.625, -1.05) to (0,-0.8) to (0,-1.2) to (-0.625, -1.45) to (-0.625,-1.05);

\draw [ultra thick] (0,-2) to (0,2);

\draw [fill] (0,-2) circle [radius=0.1];
\draw [fill] (0,2) circle [radius=0.1];

\draw (0,-2) rectangle (4,2);
\draw (0,-2) rectangle (-4,2);

\draw (0,-2) to (-2.5,-3) to (-2.5,1) to (0,2);

\draw [red] (-1,2) to (-1,1) to (1,1) to (1,2);
\draw [red] (-0.625,1.75) to (-0.625, 0.75) to (0,1);

\draw [blue] (-1,-2) to (-1,-1) to (1,-1) to (1,-2);
\draw [blue] (-0.625,-2.25) to (-0.625, -1.25) to (0,-1);

\draw [red, dotted] (-1,2) to (-1.5,2.75);
\draw [red, dotted] (1,2) to (1.5,2.75);
\draw [red, dotted] (-0.625,1.75) to (-1,2.75);
 
\draw [blue, dotted] (-1,-2) to (-1.5,-2.75);
\draw [blue, dotted] (1,-2) to (1.5,-2.75);
\draw [blue, dotted] (-0.625,-2.25) to (-.75,-2.75); 

\node [right] at (0,0) {$e$};
\node [above] at (0,2) {$v$};
\node [below] at (0,-2) {$u$};
\node [right] at (1,1.5) {$\partial N(v)$};
\node [right] at (1,-1.5) {$\partial N(u)$};
\node [left] at (-1,-1) {\small $\mathring{\st}(e_a)$};

\end{tikzpicture}
\end{center}
\caption{Regular spheres around two adjacent vertices. The star of $e_a$ is highlighted in green.} \label{fig_regular_spheres_adj_vertices}
\end{figure}
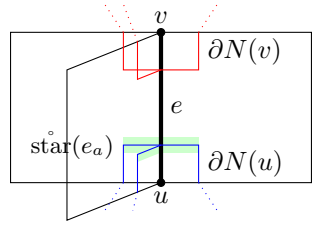
Similarly, $e_a$ ($e_b$) meets $\partial N(u)$ ($\partial N(v)$) at a vertex of valence $d$, see \cref{fig_regular_spheres_adj_vertices}. So $\partial N(u) \setminus \mathring{\st}(e_a)$ ($\partial N(v) \setminus \mathring{\st}(e_b)$) is a graph with $d$ `hanging' edges: edges with one of their endpoints having valence one. We thus see that $\partial N(e) \cong \partial N(u) \setminus \mathring{\st}(e_a) \sqcup \partial N(v) \setminus \mathring{\st}(e_b) / \sim$ with the natural gluing.
\end{proof}

\subsection{The regular sphere around a combinatorial path}

Henceforth, till the end of \cref{section_regular_neighbourhoods}, $Z$ is either $X$ or $\wt{X}$. 
Assume that $P$ is not a vertex. Let $e$ be an edge in $P$. If $P$ is not a cycle, then $P$ is a concatenation of paths $P_1$, $e$ and $P_2$. If $P$ is a cycle, we denote the connected complement of $\mathring{e}$ by $P_1$. 

\begin{lemma}\label{lemma_regular_sphere_at_combinatorial_path}
The regular sphere around $P$ is homeomorphic to a 
\begin{enumerate}
\item spliced graph of the regular spheres around $P_1$ and $P_2$ (with labelling induced by the squares containing $e$) if $P$ is not a cycle, and

\item self-spliced graph of the regular sphere around $P_1$ (with labelling induced by the squares containing $e$) if $P$ is a cycle.
\end{enumerate}
\end{lemma}
The proof is analogous to the proof of \cref{lemma_regular_sphere_at_edge}.

\subsection{Connected regular spheres}

For the rest of the section, $P$ will always be a non-cyclic path. We recall that $Z$ is either $X$ or $\wt{X}$ and $P$ is either a compact interval, a combinatorial ray or a combinatorial line.

Recall that a point $y$ of a topological space $Y$ is said to be a \emph{cut point} of $Y$ if $Y \setminus \{y\}$ is not connected.

\begin{lemma} \label{lemma_no_cuts_regular_spheres_points}
The regular sphere around any vertex or the midpoint of any edge of $Z$ is connected and has no cut points if and only if $Z$ is Brady-Meier. \qed
\end{lemma}

We now state the main result of the section.

\begin{prop} \label{prop_regular_sphere_compact_path_bradymeier}
If $P$ is compact, then the regular sphere around $P$ is connected and has no cut points.
\end{prop}

The proof requires the following lemma. It was observed in \cite{cashen_macura_line_patterns}, but they do not give a proof.

\begin{lemma} \label{lemma_connected_sum_graphs_no_cut_points}
Let $\Gamma_1$ and $\Gamma_2$ be connected graphs with no cut points. Suppose that $\Gamma$ is the spliced graph $\Gamma_1 \; {}_{(v_1,\phi_1)} \! \bigoplus_{(v_2, \phi_2)} \Gamma_2$. Then $\Gamma$ has no cut points.
\end{lemma}
\begin{proof}
First observe that $\Gamma_i \setminus \mathring{\st}(v_i)$ is connected by assumption. Let $v \in \Gamma$. We will show that $v$ is not a cut point. Assume that $v \in \Gamma_1$. The important case to consider is of a point $x \neq v \in \Gamma_1$. Since $\Gamma_1 \setminus \{v\}$ is connected, there exists a path from $x$ to $v_1$ in $\Gamma_1$ disjoint from $v$. Let $u$ be a vertex adjacent to $v_1$ on this path. Then $u$ is glued to a vertex of $\Gamma_2$ in $\Gamma$. Thus there exists a path from $x$ to $\Gamma_2$ disjoint from $v$.
\end{proof}

\begin{proof}[Proof of \cref{prop_regular_sphere_compact_path_bradymeier}]
The proof is by induction on the length of $P$. If $P$ is a vertex, then the result is obviously true. Suppose that $P$ is of length at least one. Let $e$ be an edge in $P$ and $P_1$ and $P_2$ be subpaths such that $P$ is the concatenation of $P_1$, $e$ and $P_2$. By induction, the regular sphere around $P_i$ has no cut points. 
\cref{lemma_regular_sphere_at_combinatorial_path}
and \cref{lemma_connected_sum_graphs_no_cut_points} then give the result for $P$.
\end{proof}

\begin{lemma}\label{lemma_connected_regular_spheres}
The regular sphere around a combinatorial ray $P$ of $\wt{X}$ is connected. 
\end{lemma}

\begin{proof}
Let $v \in \partial N(P)$. By \cref{lemma_regular_sphere_at_combinatorial_path}, there exists $p \in P$ such that $v \in \partial N(p)$. If $p_0$ denotes the initial point of $P$, then $P$ is a concatenation of the paths $P_1$ and $P_2$, where $P_1$ is the subpath of $P$ from $p_0$ to $p$ and $P_2$ is its complement. Since $\partial N(P_1)$ has no cut points (by \cref{prop_regular_sphere_compact_path_bradymeier}, there exists a path in $\partial N(P_1)$ from $v$ to a point $v_0$ in $\partial N(p_0)$ disjoint from $P_2$. Hence the result.
\end{proof}

\begin{cor}[Rays don't separate]
A combinatorial ray of $\wt{X}$ does not separate $\wt{X}$.
\end{cor}

The following powerful result for $\wt{X}$ will be used repeatedly in later sections.

\begin{lemma}[Path-abundance lemma] \label{lemma_path_abundance} Let $P$ be a combinatorial geodesic in $\wt{X}$ and $x \in \wt{X} \setminus P$. Then given $p \in P$, there exists a path $\alpha$ from $x$ to $p$ such that $\alpha \cap P = \{p\}$.
\end{lemma}

\begin{proof}
First note that $N(P)$ embeds in $\wt{X}$, by \cref{fact_N(P)_embedding}.
Let $\gamma$ be a path from $x$ to $p$. Let $\gamma'$ be the maximal initial subpath of $\gamma$ such that $\mathring{\gamma'} \cap P$ is empty. If $\gamma'$ ends at $p$, then declare $\gamma' = \alpha$. 

Suppose not. Let $p'$ be the endpoint of $\gamma'$. Then $P$ is a concatenation $P_1 \cdot [p',p] \cdot P_2$. By \cref{prop_regular_sphere_compact_path_bradymeier}, the regular sphere around $[p',p]$ has no cut points. In particular, $\partial N([p',p]) \setminus P_1$ is connected. We recall that we denote the point at which $P_i$ meets $\partial N([p',p])$ also as $P_i$. 

Denoting $\gamma' \cap \partial N([p,p'])$ by $\gamma'$, we note that there exists a path $\beta$ between $\gamma'$ and $P_2$ in $\partial N([p,p']) \setminus P_1$. Let $h$ be a vertex adjacent to $P_2$ such that $\beta$ meets $h$.
Note that $h \in \partial N(p) \setminus P_2$. The required path $\alpha$ is a concatenation of $\gamma'$, $\beta$ and a path in $N(p)$ from $h$ to $p$.
\end{proof}

\section{Separating and coarsely separating lines}
Recall that a subspace $Y$ of a topological space $Z$ \emph{separates two points} $z_1$ and $z_2$ in $Z$ if $z_1$ and $z_2$ lie in different components of $Z \setminus Y$. $Y$ \emph{separates $Y' \subset Z$} if $Y$ separates two points of $Y'$.

\begin{defn}[Separating lines] \label{defn_coarse_separation}
A \emph{separating line} in $\wt{X}$ is a line that separates $\wt{X}$. 
\end{defn} 

Given a subspace $Y$ of a metric space $Z$, recall that $N_R(Y)$ denotes the set of points in $Z$ at distance at most $R$ from $Y$.

\begin{defn}[Coarsely separating lines, \cite{papasoglu_coarse_separation_lamplighter}]
A line $L$ \emph{coarsely separates} $\wt{X}$ if there exists $R >0$ such that 
\begin{enumerate}
\item $N_R(L)$ separates $\wt{X}$, and
\item there exist components $Y_1 \neq Y_2$ of $\wt{X} \setminus N_R(L)$ such that for any $R'\geq R$, $Y_i \nsubseteq N_{R'}(L)$.
\end{enumerate}
\end{defn}

Since a line is an embedding in $\wt{X}$, the regular sphere around a combinatorial line embeds in $\wt{X}$, by \cref{fact_N(P)_embedding}.

Let $\mathsf{h}$ be a vertical hyperplane. Note that $\mathsf{h}$ is a combinatorial line in the first cubical subdivision of $\wt{X}$.

\begin{defn}
The regular sphere around a non-vertical tubular line $L$ at distance at most $\frac{1}{2}$ from a vertical hyperplane $\mathsf{h}$ in $\wt{X}$ is defined to be the regular sphere around $\mathsf{h}$ in the first cubical subdivision of $\wt{X}$.
\end{defn}

\begin{lemma} \label{lemma_L_separates_partial_N(P)}
Let $L$ be a combinatorial separating line in $\wt{X}$ and $P \subset L$ be a combinatorial subpath. Then $L$ separates $\partial N(P)$.
\end{lemma}
 
\begin{proof}
Suppose the lemma is not true. Then note that $N(P) \setminus L$ is connected. 
Let $x, y \in \widetilde{X} \setminus L$. Fix $p \in P$. By \cref{lemma_path_abundance}, there exist paths $\alpha$ from $x$ to $p$ and $\beta$ from $y$ to $p$ such that $\alpha \cap L = \beta \cap L = \{p\}$. Since $\partial N(P) \setminus L$ is connected, there exists a path in $\partial N(P) \setminus L$ between $\alpha \cap \partial N(P)$ and $\beta \cap \partial N(P)$. 
Thus $x$ and $y$ are not separated by $L$ for any $x,y \in \wt{X}$, a contradiction.
\end{proof} 

\subsection{Separating implies coarsely separating}
Throughout this subsection, $L$ refers to a vertical or a tubular line.

\begin{defn}[Half-spaces of a line] A \emph{half-space} of $L$ is the closure in $\wt{X}$ of a component of $\wt{X} \setminus L$.
\end{defn}

We warn the reader that there can be more than two half-spaces of a separating line in general.

An easy consequence of \cref{lemma_path_abundance} is the following result:
\begin{lemma} \label{lemma_half-spaces_contain_lines}
Let $Y$ be a half-space of $L$. Then $L \subset Y$. \qed
\end{lemma}

In fact, we can read the number of half-spaces of $L$ off its regular sphere:

\begin{lemma} \label{lemma_half-spaces_equal_components_reg_sphere_L}
There exists a natural map from the set of half-spaces of $L$ to the set of components of the regular sphere around $L$. Further, this map is bijective.
\end{lemma}

\begin{proof}
Observe that each component of $\partial N(L)$ lies in a half-space of $L$. 
Let $Y$ be a half-space of $L$, and $h_1, h_2 \in Y \cap \partial N(L)$. Then there exists a path between $h_1$ and $h_2$ in the component $Y \setminus L$. There also exists a path between $h_1$ and $h_2$ through $L$, since $h_i \in \partial N(L)$. These two paths between $h_1$ and $h_2$ bound a disk $D$, as $\wt{X}$ is simply connected, and $D \cap \partial N(L)$ gives a path between $h_1$ and $h_2$ in $\partial N(L)$. The required map is the one that sends a half-space $Y$ of $L$ to $Y \cap \partial N(L)$.
\end{proof}

\begin{cor}\label{cor_square_meets_halfspace}
Given an edge $e$ in $L$, for each component $K$ of $\partial N(L)$, there exists a square $\mathsf{s}$ containing $e$ such that $\mathsf{s} \cap \partial N(L) \subset K$.
\end{cor}
\begin{proof}
Let $Y$ be the half-space of $L$ corresponding to $K$, by \cref{lemma_half-spaces_equal_components_reg_sphere_L}.
By \cref{lemma_half-spaces_contain_lines}, $Y$ meets $e$. Let $m$ be the midpoint of $e$. By \cref{lemma_path_abundance}, there exists a path between any point in the interior of $Y$ to $m$ that does not meet $L \setminus \{m\}$. Hence $Y$ contains a square $\mathsf{s}$ that contains $e$ and is as required.
\end{proof}

\begin{fact}\label{fact_tubular_lines_separate}
It is easy to see that $L$ is a separating line whenever it is tubular. Clearly, if $L$ is not vertical, then it separates the strip that contains it. Otherwise, any strip that contains $L$ induces a component (line) of the regular sphere around $L$.
\end{fact}

\begin{lemma} \label{lemma_half-spaces_deep}
Let $Y$ be a half-space of $L$. Then for any $R >0$, $Y \nsubseteq N_R(L)$.
\end{lemma}

\begin{proof}
A hyperplane of a $\cat$ cube complex is, after subdivision, a $\cat$ subcomplex \cite{sageev}. Thus every hyperplane of $\wt{X}$ is a tree. But since each edge of $\wt{X}$ is of thickness at least 2 (\cref{lemma_bm_thickness_two}), every hyperplane is an unbounded tree.
Observe that if $L$ meets a hyperplane $\mathsf{h}$ at exactly one point, then $\mathsf{h}$ has points at arbitrarily large distances from $L$. 
It is easy to see that $Y$ contains the interior of at least one square $\mathsf{s}$. Choose $\mathsf{s}$ such that $\mathsf{s}$ meets $L$. Then the horizontal hyperplane through $\mathsf{s}$ meets $L$ at a single point.
\end{proof}

\begin{prop} \label{lemma_separating_lines_coarsely_separate}
$L$ coarsely separates $\wt{X}$. \qed
\end{prop} 

\subsection{Coarsely separating periodic lines}

\begin{defn}
The \emph{translation length} of an element $g \in G$ is the infimum of $d(x,gx)$ over all $x \in \wt{X}$.
An \emph{axis of $g \in G$} in $\wt{X}$ is a line $L$ in $\wt{X}$ such that $gL \subset L$ and $g$ moves an element of $L$ by its translation length.
A line $L$ in $\wt{X}$ is \emph{periodic} if it is an axis of some element of $G$.
\end{defn}

Given $g \in G$, an axis in $\wt{X}$ of $g$ always exists (see Theorem II.6.8 of \cite{bridsonhaefliger} for details).

\begin{lemma}\label{lemma_vertical_component_periodic_line}
Given a combinatorial periodic line $L$, either $L$ is vertical, or each vertical subpath of $L$ is compact.
\end{lemma}
\begin{proof}
Let $g \in G$ be such that $gL \subset L$.
Suppose that a vertical component of $L$ is not compact, and hence contains a ray $\gamma$. Let $e$ be an edge of $L$ adjacent to $\gamma$. Then either $g$ or $g^{-1}$ sends $e$ into $\gamma$. Since $G$ sends vertical edges to vertical edges, $e$ is vertical. Continuing this way, we conclude that $L$ is vertical. 
\end{proof}

The main result of this subsection is the following:

\begin{prop} \label{lemma_coarsely_separating_lines_separate}
A periodic coarsely separating combinatorial line $L$ of $\widetilde{X}$ separates $\wt{X}$. 
\end{prop}

The proof uses the following lemma. Recall that $\mathsf{h}^{+1}$ is the first cubical neighbourhood of a hyperplane $\mathsf{h}$.

\begin{lemma}\label{lemma_half-space_coarsely_connected_lines}
Let $Y$ be a half-space of a periodic combinatorial line $L$ such that for any vertical hyperplane $\mathsf{h}$ in $Y$, $L$ is not contained in $ \mathsf{h}^{+1}$. Then for each $k \in \mathbb{N}$, $Y \setminus L^{+k}$ is connected.
\end{lemma}

\begin{proof}[Proof of \cref{lemma_coarsely_separating_lines_separate}]

If $L$ is contained in $\mathsf{h}^{+1}$ for some vertical hyperplane $\mathsf{h}$, then $L$ is tubular and hence separating (\cref{fact_tubular_lines_separate}.
So assume that $L$ is not contained in $ \mathsf{h}^{+1}$ for any vertical hyperplane $\mathsf{h}$. Suppose that $L$ does not separate. Let $Y = \wt{X}$ be the unique half-space of $L$. By \cref{lemma_half-space_coarsely_connected_lines}, $Y \setminus L^{+k}$ is connected for all $k$, implying that $L$ does not coarsely separate. 
\end{proof}

The proof of \cref{lemma_half-space_coarsely_connected_lines} requires some work. 
For the rest of the subsection, we fix a periodic combinatorial line $L$ and a half-space $Y$ of $L$ such that $L$ is not contained in $\mathsf{h}^{+1}$ for any vertical hyperplane in $Y$. 

\begin{remark}
By Lemma 13.15 of \cite{haglund_wise_special}, $L^{+k}$ is convex for any $k$. 
\end{remark}

\begin{defn}
A hyperplane $\mathsf{h}$ is \emph{tangent to} a subcomplex $Z$ of $\wt{X}$ if $Z$ is disjoint from $\mathsf{h}$ but meets $\mathsf{h}^{+1}$.
\end{defn}

\begin{fact} \label{fact_L^k+1_from_L^k}
As $L^{+K}$ is convex, any element of $L^{+(k+1)}$ is contained either in $L^{+k}$ or in the first cubical neighbourhood of a hyperplane tangent to $L^{+k}$. 
\end{fact}

\begin{lemma}\label{lemma_hyperplane_meets_L_compact}
Given a vertical hyperplane $\mathsf{h}$ in $Y$ and $k \in \mathbb{N}$, $\mathsf{h} \cap L^{+k}$ is compact. 
\end{lemma}
\begin{proof}
Suppose there exists a vertical hyperplane $\mathsf{h}$ such that $\mathsf{h} \cap L^{+k}$ is not compact. 
Let $T$ be the underlying Bass-Serre tree of the tree of spaces structure of $\wt{X}$. By \cref{lemma_vertical_component_periodic_line}, the image of $L$ in $T$ is either a point or a line.
We first claim that $L$ is vertical. This follows from the fact that the image of $\mathsf{h}$ in $T$ is a point and a ray of $\mathsf{h}$ is at finite Hausdorff distance from a ray of $L$. 

Let $\alpha$ be the path in $T$ between the image of $\mathsf{h}$ and the image of $L$. Let $\mathsf{h}'$ be the unique vertical hyperplane tangent to $L$ such that its image in $T$ lies in $\alpha$. 
Then it is easy to see that $\mathsf{h}'^{+1} \cap L$ contains a ray.
It then follows that $L \subset \mathsf{h}'^{+1}$ as both $\mathsf{h}'$ and $L$ are periodic and $G$ acts freely on $\wt{X}$. This is a contradiction.
\end{proof}

We will denote by $\partial L^{+k}$ the set of all cells in $L^{+k}$ disjoint from $L^{+(k-1)}$.

\begin{lemma}\label{lemma_edges_in_L^k}
Let $v$ be a vertex in $\partial L^{+k}$. Then exactly one of the following holds.
\begin{enumerate}
\item One vertical and one horizontal edge incident to $v$ lie in $\partial L^{+k}$.
\item Two vertical edges (and no horizontal edge) incident to $v$ lie in $\partial L^{+k}$.
\item Finitely many horizontal edges (and no vertical edge) incident to $v$ lie in $\partial L^{+k}$.
\end{enumerate}
\end{lemma}
\begin{proof}
Since $v \notin L^{+(k-1)}$, observe that at most one edge incident to $v$ meets $L^{+(k-1)}$. The three mutually exclusive cases to then consider are that of no edge incident to $v$ meeting $L^{+(k-1)}$, a unique horizontal edge meeting $L^{+(k-1)}$ or a unique vertical edge meeting $L^{+(k-1)}$ (see \cref{fig_trichotomy_v_in_boundary_L^k}).
\end{proof}

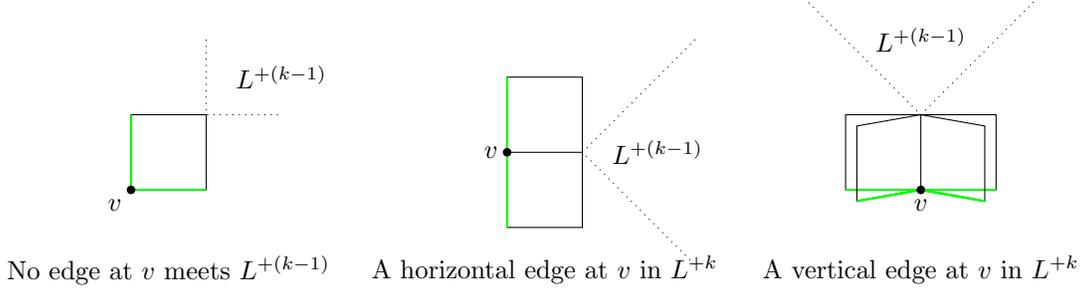
\begin{figure}
\begin{center}
\begin{tikzpicture}[scale = 0.5]
\begin{scope}[xshift = -10cm]
\draw (-1,-1) rectangle (1,1);
\draw [thick, green] (-1,1) to (-1,-1) to (1,-1);
\draw [dotted] (1,3) to (1,1) to (3,1);
\draw [fill] (-1,-1) circle [radius = 0.1];
\node [below left] at (-1,-1) {$v$};
\node at (3,2) {$L^{+(k-1)}$};
\node [below] at (0,-2.5) {No edge at $v$ meets $L^{+(k-1)}$};
\end{scope}

\begin{scope}
\draw (-1,-2) rectangle (1,2);
\draw [thick, green] (-1,2) to (-1,-2);
\draw (-1,0) to (1,0);
\draw [dotted] (4,3) to (1,0) to (4,-3);
\draw [fill] (-1,0) circle [radius = 0.1];
\node [left] at (-1,0) {$v$};
\node at (3,0) {$L^{+(k-1)}$};
\node [below] at (0,-2.5) {A horizontal edge at $v$ in $L^{+k}$};
\end{scope}

\begin{scope}[xshift = 10cm]
\draw (-2,-1) rectangle (2,1);
\draw [thick, green] (-2,-1) to (2,-1);
\draw [thick, green] (-1.7,-1.3) to (0,-1) to (1.7,-1.3);
\draw (0,-1) to (0,1);
\draw (0,-1) to (-1.7,-1.3) to (-1.7,0.7) to (0,1);
\draw (0,-1) to (1.7,-1.3) to (1.7,0.7) to (0,1);
\draw [thick, green] (-1.7,-1.3) to (0,-1) to (1.7,-1.3);
\draw [dotted] (3,4) to (0,1) to (-3,4);
\draw [fill] (0,-1) circle [radius = 0.1];
\node [below] at (0,-1) {$v$};
\node at (0,3) {$L^{+(k-1)}$};
\node [below] at (0,-2.5) {A vertical edge at $v$ in $L^{+k}$};
\end{scope}

\end{tikzpicture}
\end{center}
\caption{The edges of $\partial L^{+k}$ are in green} \label{fig_trichotomy_v_in_boundary_L^k}
\end{figure}

Before proving \cref{lemma_half-space_coarsely_connected_lines}, we will have to prove 
\begin{lemma} \label{lemma_intersecting_trace_connected_L^k+1}
Let $\mathsf{h}_1$ and $\mathsf{h}_2$ be hyperplanes in $Y$ tangent to $L^{+k}$. Suppose that $\mathsf{h}_1^{+1} \cap L^{+k}$ and $\mathsf{h}_2^{+1} \cap L^{+k}$ intersect. Then $\mathsf{h}_1^{+1} \cap \partial L^{+(k+1)}$ and $\mathsf{h}_2^{+1} \cap \partial L^{+(k+1)}$ lie in a component of $\partial L^{+(k+1)}$.
\end{lemma}

We will denote $\mathsf{h}_i^{+1} \cap L^{+k}$ by $\sigma_i$.

\begin{lemma}\label{lemma_sigma_horizontal}
Suppose that $\sigma_1$ and $\sigma_2$ are horizontal. Then $\sigma_1 \cap \sigma_2$ is a singleton.
\end{lemma}
\begin{proof}
Let $v$ be a vertex in $\sigma_1 \cap \sigma_2$. If (1) or (2) of \cref{lemma_edges_in_L^k} holds at $\{v\}$, we are done. If (3) holds, observe that any horizontal edge $f$ in $\sigma_1$ is in the first cubical neighbourhood of exactly two horizontal hyperplanes, $\mathsf{h}_1$ and $\mathsf{h}$, where $\mathsf{h}$ meets $L^{+k}$. 
\end{proof}

By \cref{lemma_hyperplane_meets_L_compact}, $\sigma_i$ is compact whenever it is vertical. Thus, $\sigma_1 \cap \sigma_2$ is always compact.

\begin{lemma}\label{lemma_terminal_vertex}
Let $v$ be a terminal vertex in $\sigma_1 \cap \sigma_2$. Then either $\sigma_1 \cap \sigma_2 = \{v\}$ or $v$ is a terminal vertex of $\sigma_1$ or $\sigma_2$. 
\end{lemma}

\begin{proof}
If $\sigma_1 \cap \sigma_2$ contains an edge, then by \cref{lemma_sigma_horizontal}, $\sigma_1$ and $\sigma_2$ are both vertical.
Then either (1) or (2) of \cref{lemma_edges_in_L^k} holds at the terminal vertex $v$. If (1) holds, then $v$ is terminal in both $\sigma_1$ and $\sigma_2$. If (2) holds (see \cref{fig_sigma_1_sigma_2}), then it is easy to see that $v$ is terminal in one of the two segments.
\end{proof}

\begin{figure}
\begin{center}
\begin{tikzpicture}[scale = 0.7]
\draw (-1,-2) rectangle (1,2);
\draw (-1,0) to (1,0);
\draw [thick, red] (-1,-2) to (-1,2); 
\draw [fill] (-1,0) circle [radius = 0.1];

\draw [thick, green] (-1,0) to (-1,-2); 
\draw [green] (-1,-2) to (-1.8,-2.8) to (-1.8,-.8) to (-1,0);
\draw [red] (-1,0) to (-1.8,.8) to (-1.8,2.8) to (-1,2);

\draw [red, dotted] (-1.8,-.7) to (-1.8,.8);
\draw [green, dotted] (-1,0) to (-2,2);
\draw [green, dotted] (-1.8,-.8) to (-2.8,1.2);

\node at (-1.4,-.45) {$e_1$};
\node at (-1.4,.45) {$e_2$};
\node [above right] at (-1,0) {$v$};

\node [right] at (-1,1) {$\sigma_2$};
\node [right] at (-1,-1) {$\sigma_1$};
\end{tikzpicture}
\end{center}
\caption{The edge $e_i$ passes through the vertical hyperplane $\mathsf{h}_i$}\label{fig_sigma_1_sigma_2}
\end{figure}
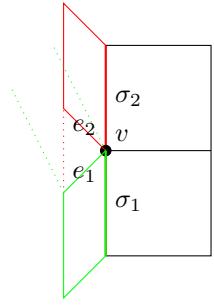

\begin{proof}[Proof of \cref{lemma_intersecting_trace_connected_L^k+1}]
Let $v$ be a terminal vertex of $\sigma_1 \cap \sigma_2$. Let $e_i$ be the edge incident to $v$ such that the hyperplane $\mathsf{h}_i$ passes through $v$.
We have three cases given by \cref{lemma_edges_in_L^k}.

\begin{mycases}
\item Only one vertical edge $f$ incident to $v$ lies in $\partial L^{+k}$. Since $\wt{X}$ is Brady-Meier, there exists a path $\beta$ in $\link(v) \setminus \{f\}$ between $e_1$ and $e_2$. The projection of $\beta$ to $\partial \{v\}^{+1}$ hits the other endpoints of $e_1$ and $e_2$, which lie in $\partial L^{+(k+1)}$. Further, $\beta$ and thus its projection are disjoint from $L^{+k}$. Hence the result.

\item Two vertical edges $f_1$ and $f_2$ incident to $v$ lie in $\partial L^{+k}$.
By \cref{lemma_terminal_vertex}, either $\sigma_1$, say, is horizontal or $v$ is terminal in $\sigma_1$. Thus one of the edges, say $f_2$, does not lie in $\sigma_1$. Let $\beta$ be a path in $\link(v) \setminus \{f_1\}$ between $e_1$ and $e_2$. Since $f_2$ does not lie in $\sigma_1$, $\beta$ and its projection to $\partial \{v\}^{+1}$ is disjoint from $\sigma_1$. If $f_2$ does not lie in $\sigma_2$ or $\beta$ is disjoint from $f_2$, then we are done as the projection of $\beta$ gives the required path in $ \partial L^{+(k+1)}$. If not, then we repeat the procedure at $v'$, the other endpoint of $f_2$. We continue until the path no longer meets the compact $\sigma_2$. Hence the result.

\item Only horizontal edges incident to $v$ lie in $\partial L^{+k}$.
Let $f$ be the vertical edge incident to $v$ and contained in $L^{+k}$. Let $\beta$ be a path between $e_1$ and $e_2$ in $\link(v) \setminus \{f\}$. Then $\beta$ is disjoint from $L^{+k}$ and so is its projection to $\partial \{v\}^{+1}$.
\end{mycases}
\end{proof}

We are now ready to prove \cref{lemma_half-space_coarsely_connected_lines}.

\begin{proof}[Proof of \cref{lemma_half-space_coarsely_connected_lines}]
The proof is by induction. Note that $Y \setminus L^{+k}$ is connected whenever $Y \cap \partial L^{+k}$ is connected.
Since $Y$ is a half-space of $L$, $Y \cap \partial N(L)$ is connected, by \cref{lemma_half-spaces_equal_components_reg_sphere_L}. Thus $Y \cap \partial L^{+1}$ is connected.

Assume that $Y \cap \partial L^{+k}$ is connected, for some $k$. We will now show that $Y \cap \partial L^{+(k+1)}$ is connected.
Indeed, $L^{+(k+1)}$ is contained in the union of $L^{+k}$ and the first cubical neighbourhoods of hyperplanes tangent to $L^{+k}$, by \cref{fact_L^k+1_from_L^k}. Thus given two vertices $u$ and $u'$ in $Y \cap \partial L^{+(k+1)}$, there exist hyperplanes $\mathsf{h}$ and $\mathsf{h}'$ tangent to $L^{+k}$ such that $u \in \mathsf{h}^{+1}$ and $u' \in \mathsf{h}'^{+1}$. Let $\sigma = \mathsf{h}^{+1} \cap L^{+k}$ and $\sigma' = \mathsf{h}'^{+1} \cap L^{+k}$. By the induction assumption, there exists a path between $\sigma$ and $\sigma'$ in $\partial L^{+k}$. This implies that there exists a finite sequence of tangent hyperplanes $\mathsf{h} = \mathsf{h}_1, \cdots, \mathsf{h}_n = \mathsf{h}'$ such that if $\sigma_i = \mathsf{h}_i^{+1} \cap L^{+k}$, then $\sigma_i \cap \sigma_{i+1}$ is nonempty. \cref{lemma_intersecting_trace_connected_L^k+1} then implies that $u$ and $u'$ lie in a component of $\partial L^{k+1}$. Hence the result.
\end{proof}

\section{A crossing criterion for lines}

\begin{defn}[Crossing of lines] \label{def_crossing_lines}
Let $L$ and $L'$ be two separating lines of $\wt{X}$. We say that $L$ \emph{crosses} $L'$ if for every half-space $Y'$ of $L'$, $L \nsubseteq Y'$. $L$ and $L'$ \emph{don't cross} if neither $L$ crosses $L'$ nor $L'$ crosses $L$.
\end{defn}

Note that two disjoint lines don't cross. Thus a vertical line and a tubular line that is not vertical never cross. We will see later that in fact, no vertical line crosses a tubular line.
Two intersecting lines may or may not cross. The main goal of this section is to obtain the following criterion for the crossing of two lines. 

\begin{prop}[Crossing criterion] \label{prop_local_crossing_lines}
Let $L$ and $L'$ be two separating combinatorial lines in $\wt{X}$. $L$ crosses $L'$ if and only if 
\begin{enumerate}
\item $L \cap L' = P$ is non-empty and compact and
\item $L' \cap \partial N(P)$ separates $L \cap \partial N(P)$.
\end{enumerate} 
\end{prop}

Throughout this section, $L$ and $L'$ are two separating combinatorial lines and $P$ denotes their intersection.
The following three lemmas are consequences of the Brady-Meier property of $\wt{X}$ and the path-abundance lemma (\cref{lemma_path_abundance}), and we omit the proofs.
\begin{lemma} \label{lemma_P_non-compact}
If $P$ is either empty or non-compact, then $L$ and $L'$ don't cross. \qed
\end{lemma}

\begin{lemma}[Crossing is symmetric] \label{lemma_L_L'_others_half-spaces}
$L$ is contained in a half-space of $L'$ if and only if $L'$ is contained in a half-space of $L$. \qed
\end{lemma}

\begin{lemma} \label{lemma_crossing_by_half-spaces}
$L$ and $L'$ don't cross if and only if for each half-space $Y$ of $L$, there exists a half-space $Y'$ of $L'$ such that either $Y \subset Y'$ or $Y' \subset Y$ and similarly for each half-space $Y'$ of $L'$, there exists a half-space $Y$ of $L$ such that either $Y \subset Y'$ or $Y' \subset Y$. \qed
\end{lemma}

Before we go to the proof of \cref{prop_local_crossing_lines}, we will collect a couple of results about graphs without cut points as $\partial N(P)$ has no cut points whenever $P$ is compact (\cref{prop_regular_sphere_compact_path_bradymeier}). 

\subsection{Graphs with no cut points}

We fix a connected graph $\Gamma$ in this subsection such that $\Gamma$ has no cut points. We further assume that $\Gamma$ contains at least one edge. A \emph{cut pair} is a pair of points that separates $\Gamma$. 

We now draw the attention of the reader to certain similarities between cut pairs in $\Gamma$ and separating lines in $\wt{X}$. If $\{a,b\}$ is a cut pair, then a \emph{half-space} of $\{a,b\}$ is the closure of a component of $\Gamma \setminus \{a,b\}$. The first result is analogous to \cref{lemma_half-spaces_contain_lines}.

\begin{lemma}
Let $Y$ be a half-space of a cut pair $\{a,b\}$. Then $\{a,b\} \subset Y$. 
\end{lemma}
\begin{proof}
Since $\Gamma$ is connected, at least one of the two, say $a$, is contained in $Y$. If $b$ is not contained in $Y$, then $a$ is a cut point as $a$ separates $Y$ from $b$.
\end{proof}
The second result is analogous to \cref{lemma_L_L'_others_half-spaces}.
\begin{lemma}[Lemma 2.3 (1), \cite{cashen_relative_jsj}] \label{lemma_crossings_cut_pairs_in_graphs}
Let $\{a,b\}$ and $\{a',b'\}$ be cut pairs in $\Gamma$. Then $\{a',b'\}$ separates $\{a,b\}$ if and only if $\{a,b\}$ separates $\{a',b'\}$.
\end{lemma}

\begin{cor}[Lemma 2.3 (3), \cite{cashen_relative_jsj}] \label{cor_cut_pairs_3_components_no_crossings}
Let $\{a,b\}$ and $\{a',b'\}$ be cut pairs in $\Gamma$. 
If there exist at least three half-spaces of $\{a,b\}$, then $\{a',b'\}$ is not separated by $\{a,b\}$.
\end{cor}

\subsection{The criterion}
\begin{proof} [Proof of \cref{prop_local_crossing_lines}]
By \cref{lemma_L_separates_partial_N(P)}, both $L$ and $L'$ separate $\partial N(P)$.
Suppose that $P$ is compact and $L' \cap \partial N(P)$ separates $L \cap \partial N(P)$. Then by \cref{cor_cut_pairs_3_components_no_crossings}, $L'$ separates $\partial N(P)$ into exactly two components. This implies that $L'$ has exactly two half-spaces as each half-space of $L'$ meets $P$, by \cref{lemma_half-spaces_contain_lines}. This in turn implies that different components of $\partial N(P) \setminus L'$ are contained in different half-spaces of $L'$. Hence $L$ crosses $L'$. 

For the converse, note that if $P$ is not compact then $L$ and $L'$ don't cross, by \cref{lemma_P_non-compact}. Similarly, if $L' \cap \partial N(P)$ does not separate $L \cap \partial N(P)$, then clearly, $L$ lies in a half-space of $L'$.
\end{proof}

\section{Cyclic splittings and separating lines} \label{section_dual_trees}

\begin{defn} \label{def_algebraic_splitting}
A group $G$ \emph{splits over a subgroup $H$} if $G$ decomposes as a nontrivial free product with amalgamation over $H$ or as an HNN extension over $H$. In the sequel, a decomposition of $G$ either as a free product with amalgamation or as an HNN extension (over $H$) will be called a \emph{basic splitting} of $G$ (over $H$). 
\end{defn}

\begin{prop}[\cite{serre}] \label{prop_splitting_tree}
$G$ splits over $H$ if and only if $G$ acts without edge inversions on an unbounded tree $T$ such that $H$ is the stabiliser of some edge of $T$ and there exists no proper $G$-invariant subtree of $T$.
\end{prop}

As a consequence of \cref{lemma_coarsely_separating_lines_separate}, we have the following application of a result of Papasoglu:
\begin{lemma}[\cite{papasoglu_coarse_separation}] \label{lemma_algebraic_splittings_and_separating_lines}
Let $H$ be a cyclic subgroup over which $G$ splits. Suppose that a vertical line $L$ is an axis of $H$. Then $L$ is a separating line that does not cross any of its translates.
\end{lemma}

\begin{lemma} \label{lemma_dual_tree}
Let $L$ be a periodic line that separates $\wt{X}$ and does not cross any of its translates. Then there exists an unbounded $G$-tree $T_L$ and a vertex in $T_L$ whose stabiliser is the stabiliser of $L$. Further, there exists no proper $G$-invariant subtree of $T_L$.
\end{lemma}

The construction of such a \emph{dual tree} when $L$ has exactly 2 half-spaces and is disjoint from all its translates is standard. In that case, the dual tree is bipartite and is constructed as follows: each component of $\wt{X} \setminus \sqcup_{g \in G} gL$ defines a black vertex while each translate of $L$ defines a white vertex. The adjacency is given by containment: a white vertex is adjacent to a black vertex if it is contained in the closure of the black vertex. 

In our case, $L$ may not be disjoint from its translates and may have more than two half-spaces. This necessitates a more careful treatment, but the underlying idea is still the same. Our construction, in fact, coincides with the above standard construction when $L$ is disjoint from its translates and has only two half-spaces.

We start with an observation that will be used in the proof.

\begin{lemma} \label{lemma_non-crossing_lines_and_half-spaces}
Let $L_1$ and $L_2$ be separating lines that don't cross. Given half-spaces $Y_1$ of $L_1$ and $Y_2$ of $L_2$ such that $(Y_1 \setminus L_1) \cap (Y_2 \setminus L_2)$ is non-empty, then either $L_1$ is contained in $Y_2$ or $L_2$ is contained in $Y_1$. 
\end{lemma}
\begin{proof}
Since $L_1$ and $L_2$ don't cross, there exist half-spaces $Y_1'$ of $L_1$ and $Y'_2$ of $L_2$ such that $L_1 \subset Y'_2$ and $L_2 \subset Y'_1$, by definition. We claim that either $Y'_1 = Y_1$ or $Y'_2 = Y_2$. Suppose not. 
Since $L_1 \subset Y_2'$, $L_1$ is disjoint from $\wt{X} \setminus Y_2' \supset Y_2 \setminus L_2$. Thus $Y_2 \setminus L_2$ is contained in a half-space $Y''_1$ of $L_1$. But the fact that $L_2$ is contained in $Y'_1$ implies that $Y_1'' = Y_1'$ and hence $Y_2 \setminus L_2$ is disjoint from $Y_1 \setminus L_1$, a contradiction. 
\end{proof}

The required tree $T_L$ will be the $\cat$ cube complex dual to a space with walls. Recall that 
\begin{defn} [\cite{haglund_paulin_walls}]
A \emph{wall} on a nonempty set $Z$ is a partition of $Z$ into two subsets. $Z$ is a \emph{space with walls} if $Z$ is endowed with a collection of walls such that any two points of $Z$ are separated by finitely many walls.
\end{defn} 

\begin{remark} The two subsets that define a wall are known as half-spaces in the literature. Note that we have already used this terminology for separating lines. Separating lines in $\wt{X}$ do define walls, as we will show below.
We will refer to a half-space associated to a wall as a half-space of the space with walls.
\end{remark}

We quickly recall some terminology of spaces with walls before going to the proof of \cref{lemma_dual_tree}. We refer the reader to \cite{nica_cubulation} for further details.

\begin{defn}
Let $Z$ be a space with walls. An \emph{ultrafilter} on $Z$ is a nonempty collection $\omega$ of half-spaces of $Z$ that satisfy the following conditions:
\begin{enumerate}
\item $A \in \omega$ and $A \subset B$ imply that $B \in \omega$ and
\item exactly one of $A$ and $A^c$ is contained in $\omega$.
\end{enumerate}
\end{defn}

\begin{lemma} \label{lemma_elements_ultrafilters_intersect}
If $\omega$ is an ultrafilter of $Z$ and $A, B \in \omega$, then $A$ and $B$ are not disjoint. \qed
\end{lemma}

For a $z \in Z$, the \emph{principal ultrafilter} $\sigma_z$ is defined to be the set of half-spaces of $Z$ that contain $z$. An ultrafilter $\omega$ of $Z$ is said to be \emph{almost principal} if for some (and therefore for any) $z \in Z$, the symmetric difference between $\omega$ and $\sigma_z$ is finite.

\begin{proof}[Proof of \cref{lemma_dual_tree}]
Let $Z_L = \wt{X} \setminus \cup_{g \in G} gL$. Then each half-space $Y$ of $gL$ defines a wall $\{Y \cap Z_L, Y^c \cap Z_L\}$, which we will denote as $\{Y, Y^c\}$. It is easy to see that $Z_L$ is a space with walls. 
By theorem 4.1 of \cite{nica_cubulation}, there exists a connected graph $T_L$ whose vertices are the principal and almost principal ultrafilters of $Z_L$. Two vertices are adjacent if the cardinality of their symmetric difference is two. $T_L$ is then the 1-skeleton of a unique $\cat$ cube complex (see section 3 of \cite{sageev}, for instance). 

We claim that $T_L$ is a tree. If not, then it is the 1-skeleton of a cube complex of dimension at least 2. Thus there exists a cycle $(w_1,w_2,w_3,w_4)$ of length 4 in $T_L$.
Since $\omega_1$ and $\omega_2$ are adjacent, there exists a wall $\{Y,Y^c\}$ of $Z_L$ such that $Y \in \omega_1$ and $Y^c \in \omega_2$. Similarly, there exists a wall $\{Y',Y'^c\}$ such that $Y' \in \omega_1$ and $Y'^c \in \omega_4$. Note that $Y' \in \omega_2$, $Y \in \omega_4$ and $Y^c, Y'^c \in \omega_3$, by definition. We will show below that this is not possible.
Assume that $Y$ and $Y'$ are half-spaces of the lines $gL$ and $g'L$.
By \cref{lemma_elements_ultrafilters_intersect}, $Y$ and $Y'$ are not disjoint. This implies that either $gL \subset Y'$ or $g'L \subset Y$, by \cref{lemma_non-crossing_lines_and_half-spaces}. Assume the former. Either $Y \subset Y'$ or not. 
If $Y \subset Y'$, then no ultrafilter can contain both $Y$ and $Y'^c$ and hence $\omega_4$ cannot exist. 
On the other hand, if $Y \nsubseteq Y'$, then $g'L$ meets $Y$ in its interior and hence $Y'^c \subset Y$. This then implies that no ultrafilter can contain both $Y^c$ and $Y'^c$ and hence $\omega_3$ cannot exist. This proves the claim.

There exists a natural action of $G$ on $T_L$. An element $g \in G$ sends an ultrafilter $\omega$ to an ultrafilter $g\omega$ where $g \omega$ is the set of half-spaces $gY$ of $Z_L$, where $Y \in \omega$. 

We claim that there exists an ultrafilter whose stabiliser is the stabiliser of $L$. Let $Y_1, \cdots, Y_n$ be the set of half-spaces of $L$, and let $\omega_L$ be the set of half-spaces of $Z_L$ consisting of $Y_1^c, \cdots, Y_n^c$ and all half-spaces (of proper translates of $L$) which contain $L$. Note that $\omega_L$ is an ultrafilter. (If $n=2$, then $\omega_L \setminus \{Y_1\}$ and $\omega_L \setminus \{Y_2\}$ are ultrafilters and hence $\omega_L$ is a vertex of $T_L$ after subdivision.) 
We claim that $\omega_L$ is almost principal. Indeed, choose $y_1 \in Y_1 \cap Z_L$. Then $\sigma_{y_1} = \{Y_1, Y_2^c, \cdots, Y_n^c\} \cup \{Y | y_1 \in Y\}$. There exist at most finitely many lines $g_1L \cdots g_kL$ that separate $y_1$ from $L$ in $\wt{X}$. Except for the half-spaces of these lines, a half-space contains $y_1$ if and only if it contains $L$. Hence $\sigma_{y_1} \triangle \omega_L$ is finite. It is straightforward to check that $\mathrm{stab}(L) = \mathrm{stab}(\omega_L)$.

\cref{lemma_half-spaces_deep} implies that $T_L$ is unbounded as a vertex at maximal distance from $\omega_L$ will contain a half-space that does not contain a translate of $L$ in its interior.

There is no proper $G$-invariant subtree of $T_L$. Now $T_L$ is spanned as a tree by the principal ultrafilters of $Z_L$, by Proposition 4.8 of \cite{nica_cubulation}. It thus suffices to prove that no subtree spanned by a proper subset of the set of principal ultrafilters is $G$-invariant. 
Choose $y_i \in Y_i \cap Z_L$ such that there exists a path $\alpha$ from $y_i$ to $L$ with $\mathring{\alpha} \subset Z_L$. Then observe that any principal ultrafilter $\sigma_y$ is a translate of $\sigma_{y_i}$, for some $i$. 
Thus, if a proper subtree is $G$-invariant, then it has to miss at least one $\sigma_{y_i}$, say $\sigma_{y_1}
$. But this is not possible as the interior of $Y_1$ contains at least one translate of $L$, by \cref{lemma_half-spaces_deep}.
\end{proof}

\begin{prop} \label{prop_cyclic_splittings_induce_separating_lines}
Let $H$ be a cyclic subgroup of $G$ and $L$ an axis of $H$ in $\wt{X}$. Suppose that $L$ is a separating line that does not cross any of its translates and that $H$ is equal to the stabiliser of a proper subset of the set of half-spaces of $L$.
Then $G$ splits over $H$.
\end{prop}
\begin{proof}
Let $T_L$ be the dual tree of $L$ obtained from \cref{lemma_dual_tree}. 
Let $Y_1, \cdots, Y_n$ be the half-spaces of $L$ and $\omega_i$ be the vertex adjacent to $\omega_L$ such that $\omega_i \triangle \omega_L = \{Y_i, Y_i^c\}$. Let $T$ be the quotient simplicial graph of $T_L$ obtained by first identifying for each $h \in H$ and $i \in \{1, \cdots, n\}$, vertices $\omega_i$ and $h\omega_i$, and then extending equivariantly. It is easy to check that $T$ is a tree that satisfies the conditions of \cref{prop_splitting_tree} for $H$.
\end{proof}

A cyclic subgroup $H$ of $G$ that satisfies the hypothesis of \cref{prop_cyclic_splittings_induce_separating_lines} is a \emph{geometric splitting subgroup}.

\begin{prop} \label{prop_algebraic_commensurable_geometric_splitting}
Let $H$ be a cyclic subgroup of $G$ over which $G$ splits. If $H$ has a vertical axis in $\wt{X}$ then there exists a geometric splitting subgroup $H'$ commensurable with $H$.
\end{prop}

\begin{proof}
Let $L$ be an axis of $H$ satisfying \cref{lemma_algebraic_splittings_and_separating_lines}. 
Observe that $H$ is contained in the stabiliser of $L$ which is a cyclic subgroup. Choose a half-space $Y$ of $L$ and let $H'$ be the largest subgroup of the stabiliser of $L$ that preserves $Y$. Then by \cref{prop_cyclic_splittings_induce_separating_lines}, $H'$ is as required.
\end{proof}

\section{Vertical cycles and cyclic splittings}
In this section, we will examine splittings induced by vertical lines in $\wt{X}$.
Recall that a cycle (\cref{def_paths_cycles_lines}) is an immersion of graphs $\phi: C \to \Gamma$, where $C$ is a subdivided circle. From now on, throughout the text, unless mentioned otherwise, $\Gamma$ will be a vertex graph $X_s$ of $X$ and so $C$ is a vertical cycle.

\begin{remark}
The map $\phi$ is $\pi_1$-injective. Indeed, $\pi_1(C)$ injects into $\pi_1(X_{s})$ \cite{stallings_topology_finite_graphs} and $\pi_1(X_{s})$ injects into the fundamental group of $X$ in the graph of groups setup \cite{serre}. 
\end{remark}

\begin{fact}
A line in $\wt{X}$ is periodic and vertical if and only if it is a lift of a vertical cycle in $X$.
\end{fact}

By abuse of notation, we will often call the lift $\tilde{\phi} : \wt{C} \to \wt{X}$ as the line $\wt{C}$. Since the projection of the regular neighbourhood of a path onto the path is a deformation retraction, we have:

\begin{lemma} \label{lemma_regular_sphere_C_quotient_wt(C)}
$\partial N(C) \cong \partial N(\wt{C}) / \pi_1(C)$. \qed
\end{lemma}

\begin{defn}
A \emph{cyclic path} is an immersed combinatorial path $\rho : P \to X_{s}$ such that the initial and terminal vertices of $P$ have the same image while the initial and terminal edges of $P$ have distinct images. 

A cyclic path $P$ induces a quotient cycle $\phi_P : C_P \to X_{s}$, where $C_P$ is the quotient of $P$ obtained by gluing the initial and terminal vertices and defining $\phi_P([x]) \coloneqq \rho(x)$.
\end{defn}

\begin{defn}[Fundamental domain of a cycle]
Let $\phi: C \to X_s$ be a cycle. A cyclic path $\rho_C : P_C \to X_s$ with induced quotient cycle $C_{P_C}$ is said to be a \emph{fundamental domain of $C$} if the following diagram commutes.
\begin{center}
\begin{tikzcd}[sep=tiny]
C_{P_C} \arrow [rr, "\cong"] \arrow[dr, "\phi_{P_C}"'] && C \arrow [dl, "\phi"] \\
&X_s&
\end{tikzcd}
\end{center}

\end{defn}
\begin{remark}
It is easy to see that for the action of $\pi_1(C)$ on $\wt{C}$, a lift $\wt{P}_C$ of $P_C$ is a fundamental domain of $\wt{C}$ in the usual sense.
\end{remark}

\begin{defn}\label{defn_orthogonal_sphere}
Let $P_C$ be a fundamental domain of a cycle $C$. Let $u$ and $v$ be the initial and terminal vertices of $P_C$ and $a$ and $b$ the initial and terminal edges. Let $b_u$ be the vertex of $\partial N(u)$ that meets $b$ and $a_v$ the vertex of $\partial N(v)$ that meets $a$. The \emph{orthogonal sphere around $P_C$}, denoted by $\partial_{orth} N(P_C)$, is defined as the closure of $\partial N(P_C) \setminus (\{b_u\}^{+2} \cup \{a_v\}^{+2})$, where $\{b_u\}^{+2}$ (respectively $\{a_v\}^{+2}$) denotes the second cubical neighbourhood in $\partial N(P_C)$ of $b_u$ ($a_v$).
\end{defn}

Let $\wt{C}$ be a lift of $C$ and $\wt{P}_C \subset \wt{C}$ of $P_C$. Then note that 

\begin{fact}\label{fact_orthogonal_sphere_and_its_lift}
The natural map from $P_C \cong \wt{P}_C \hookrightarrow \wt{C}$ induces an embedding of graphs $\partial_{orth} N(P_C) \hookrightarrow \partial N(\wt{P}_C) \setminus \wt{C} \subset \partial N(\wt{C})$ as a deformation retract. Further , $\partial_{orth} N(P_C)$ is connected if and only if $\partial N(\wt{P}_C) \setminus \wt{C}$ is connected.
\end{fact}

It thus follows from \cref{lemma_regular_sphere_C_quotient_wt(C)} that
\begin{lemma} \label{lemma_orthogonal_sphere_fundamental_domain}
The regular sphere around a cycle $C$ is isomorphic to the quotient of the orthogonal sphere around a fundamental domain $P_C$ of $C$ with the natural gluing induced by $\pi_1(C)$.
\end{lemma}

Let $\rho_C : P_C \to X_s$ and $e'$ an edge in $P_C$. By \cref{cor_square_meets_halfspace}, we have:

\begin{lemma}\label{lemma_edge_square_orthogonal_sphere}
Let $K$ be a component of $\partial_{orth} N(P_C)$. Then there exists a square $\mathsf{s}$ in $N(P_C)$ that meets $e'$ and $\mathsf{s} \cap \partial_{orth} N(P_C) \subset K$. \qed
\end{lemma} 

\begin{defn} A cycle $C$ is a \emph{UC-separating (universal cover separating) cycle} if $\wt{C}$ is a separating line. $C$ is \emph{strongly UC-separating} if $\partial N(C)$ is not connected.
\end{defn}

By \cref{lemma_L_separates_partial_N(P)}, we have 
\begin{lemma}\label{lemma_orth_sphere_uc_separating}
If $C$ is a UC-separating cycle, then $\partial_{orth} N(P_C)$ is not connected.
\end{lemma}

\begin{lemma} \label{lemma_regular_sphere_C_disconnected}
$C$ is strongly UC-separating if and only if the following two conditions are satisfied: 
\begin{enumerate}
\item $C$ is a UC-separating cycle and
\item $\pi_1(C)$ does not act transitively on the set of half-spaces of $\wt{C}$.
\end{enumerate}

\end{lemma}

\begin{proof}
Recall that $\partial N(C)\cong \partial N(\wt{C}) / \pi_1(C)$ (\cref{lemma_regular_sphere_C_quotient_wt(C)}).
A component of $\partial N(C)$ lifts to a component of $\partial N(\wt{C})$. So if $\partial N(C)$ is connected, then either $\partial N(\wt{C})$ is itself connected or every component of $\partial N(\wt{C})$ projects onto $\partial N(C)$. So $\pi_1(C)$ acts transitively on the components of $\partial N(\wt{C})$ and therefore on the components of $\wt{X} \setminus \wt{C}$ (\cref{lemma_half-spaces_equal_components_reg_sphere_L}). 
The converse is clear. \end{proof}

\begin{defn}
A cycle $\phi' : C' \to X_{s}$ is an \emph{$n^{th}$-power} of the cycle $\phi : C \to X_{s}$ if there exists an $n$-fold covering map $\psi: C' \to C$ such that the following diagram commutes. 
\begin{center}
\begin{tikzcd}[row sep=tiny]
C' \arrow[dd, "\psi"] \arrow[dr, "\phi'"] & \\
 & X_{s} \\
C \arrow[ur, "\phi"] &
\end{tikzcd}
\end{center}
\end{defn}

\begin{lemma}\label{lemma_power_uc_separatig_is_separating}
Let $N$ be such that the thickness of any edge of $X$ is at most $N$. Given a UC-separating cycle $C$, there exists $n\leq N$ such that the regular sphere around an $n^{\mathrm{th}}$ power of $C$ is not connected. 
\end{lemma}

\begin{proof}
By \cref{cor_square_meets_halfspace}, the number of half-spaces of $\wt{C}$ is at most $N$. Thus there exists a subgroup $H$ of index at most $N$ of $\pi_1(C)$ that does not act transitively on the set of half-spaces of $\wt{C}$. The required cycle $C'$ is the quotient of $\wt{C}$ by $H$.
\end{proof}

\begin{defn}
Let $\rho_C : P_C \to X_s$ be a fundamental domain of a cycle $C$. 
A \emph{subcycle} of $C$ is the quotient cycle of a cyclic path $\rho_C|_P : P \to X_s$ with $P \subset P_C$.
\end{defn}

Observe that if $C'$ is an $n^{th}$-power of $C$, then $C$ is a subcycle of $C'$. We will often use this fact.

\begin{defn}
A UC-separating cycle $C$ has a \emph{self-crossing} if $\wt{C}$ and a translate cross.
\end{defn}

\begin{defn} \label{def_splitting_cycle}
A cycle $C$ is a \emph{splitting cycle} if the following conditions are satisfied: \begin{enumerate}
\item $C$ is strongly UC-separating,
\item $\pi_1(C)$ is equal to the stabiliser of a proper subset of the set of half-spaces of $\wt{C}$, and
\item $C$ has no self-crossings.
\end{enumerate}
\end{defn}

\begin{remark}
By \cref{prop_cyclic_splittings_induce_separating_lines}, $G$ splits over $\pi_1(C)$ whenever $C$ is a splitting cycle. 
\end{remark}

We will now examine when $C$ can have self-crossings. We start with the following.

\begin{lemma} \label{lemma_length_lcm}
Let $L_1$ and $L_2$ be vertical lines of $\wt{X}$ stabilised by $H_1$ and $H_2$ respectively. Let $n_i$ be the translation length of a generator of $H_i$. If the length of $P = L_1 \cap L_2$ is at least LCM($n_1,n_2$), then $P = L_1$.
\end{lemma}

Recall that $L_i$ is an axis of $H_i$. Hence a generator of $H_i$ translates every point of $L_i$ by $n_i$ (see Theorem II.6.8(i) of \cite{bridsonhaefliger}). Since any element of $H_i$ takes vertices to vertices, $n_i$ is indeed an integer. 

\begin{proof}
Suppose that $P$ contains a segment of length LCM$(n_1,n_2) = k$. Let $v$ be a terminal point of $P$. Choose generators $h_1 \in H_1$ and $h_2 \in H_2$ such that $h_i(v) \in P$. 
Since the length of $P$ is at least $k$, $h_i^{k/n_i}(v) \in P$ and hence $h_1^{k/n_1}(v) = h_2^{k/n_2}(v)$. As $G$ acts freely on $\wt{X}$, $h_1^{k/n_1} = h_2^{k/n_2}$ and hence $L_1 = L_2$.
\end{proof}

\begin{cor} \label{cor_P_embeds}
If $L_2$ is a translate of a periodic vertical line $L_1$, then either $L_2 = L_1$ or $P$ embeds in $L_1 / H_1$, where $H_1$ is the stabiliser of $L_1$.
\end{cor}
In particular, for a cycle $C$, if $g\cdot \widetilde{C} \neq \wt{C}$, then $P = g\cdot \widetilde{C} \cap \wt{C}$ embeds in $C$.

\begin{defn} \label{defn_self_intersection}
A segment $P \subsetneq C$ is said to be a \emph{component of self-intersection} of $C$ if there exists a translate $g \wt{C} \neq \wt{C}$ such that the projection to $C$ of $\wt{C} \cap g\wt{C}$ is equal to $P$. We say that there is a \emph{self-crossing of $C$ at $P$} if there exists a $g \in G$ such that $\wt{C} \cap g \wt{C} = P$ and $\wt{C}$ and $g \wt{C}$ cross. 
\end{defn}

\begin{fact} \label{fact_sphere_P_embeds_sphere_C}
Let $P \subset C$ be a segment so that a lift of $P$ in $\wt{C}$ is isomorphic to $P$ (and hence also denoted by $P$). Then $\partial N(P) \cap \partial N(C) \simeq \partial N(P) \setminus \wt{C}$. In other words, there is a self-crossing at $P$ only if $g \wt{C}$ meets different components of $\partial N(P) \cap \partial N(C)$, by \cref{prop_local_crossing_lines}.
\end{fact}

\begin{lemma} \label{lemma_necessary_conditions_splitting_cycle}
A splitting cycle is strongly UC-separating and has no self-crossing at any component of self-intersection. \qed
\end{lemma} 

Also, splitting cycles capture all `vertical' splittings up to commensurability:

\begin{lemma} \label{lemma_alg_splittings_commensurable_geom_splittings}
Let $H$ be a cyclic subgroup over which $G$ splits with a vertical axis. Then there exists a splitting cycle $C$ such that $\pi_1(C)$ is commensurable with a conjugate of $H$. 
\end{lemma}
\begin{proof}
By \cref{prop_algebraic_commensurable_geometric_splitting}, $H$ is commensurable with a geometric splitting subgroup $H'$. Let $L$ be a vertical axis for $H$.
The required splitting cycle is obtained by taking the quotient of $L$ by $H'$. 
\end{proof}

As an immediate consequence, we have
\begin{lemma} \label{lemma_splitting_cycle_commensurable}
Given a UC-separating cycle $C$ with no self-crossings, there exists a splitting cycle $C'$ such that $\pi_1(C)$ and $\pi_1(C')$ are commensurable. \qed
\end{lemma}

\section{Universally elliptic splittings}

Recall that a subgroup $H$ of $G$ is \emph{elliptic} in a $G$-tree $T$ if $H$ fixes a point in $T$.

\begin{lemma}[Elliptic splittings] \label{lemma_elliptic-elliptic_splittings}
Let $H_1$ and $H_2$ be cyclic subgroups over which $G$ splits. Let $L_i$ be an axis of $H_i$ in $\wt{X}$. $H_1$ is elliptic in the Bass-Serre tree of the basic splitting over $H_2$ if and only if given any translate $gL_2$ of $L_2$, $L_1$ and $gL_2$ don't cross.
\end{lemma}
\begin{proof}
Note that if $L_1$ and $gL_2$ don't cross for any $g$, then $L_1$ is contained in a half-space of $gL_2$ for each $g$. Thus for $x \in L_1 \setminus L_2$, the stabiliser of $\sigma_x$ in the dual tree $T_{L_2}$ (\cref{lemma_dual_tree}) of $L_2$ contains $H_1$ and hence $H_1$ is elliptic in $T_{L_2}$. The Bass-Serre tree $T_2$ of the basic splitting over $H_2$ is obtained from $T_{L_2}$ by a sequence of $G$-equivariant gluings of edges of $T_{L_2}$. Thus $H_1$ remains elliptic in $T_2$.
Conversely, if there exists $g$ such that $L_1$ and $gL_2$ cross, then $g^{-1}H_1g$ is hyperbolic in the dual tree $T_{L_2}$ and hence in $T_2$. 
\end{proof}

\begin{remark}
Since $G$ is one-ended, $H_1$ is elliptic in the Bass-Serre tree of the basic splitting over $H_2$ if and only if $H_2$ is elliptic in the Bass-Serre tree of the basic splitting over $H_1$, see Theorem 2.1 of \cite{rips_sela_jsj}.
\end{remark}

\begin{defn}[\cite{guirardel_levitt_jsj}] \label{def_universally_elliptic_splittings}
A cyclic splitting of $G$ over the subgroup $H$ is \emph{universally elliptic} if $H$ is elliptic in the Bass-Serre tree of any cyclic splitting of $G$. We then say that $H$ is a universally elliptic subgroup.
Analogously, a splitting cycle $C$ is universally elliptic if $\pi_1(C)$ is universally elliptic.
\end{defn}

A splitting induced by a transversal line can never be universally elliptic:

\begin{lemma}\label{lemma_transversal_not_universally_elliptic}
Let $H$ be a cyclic subgroup over which $G$ splits. Suppose that an axis of $H$ is transversal in $\wt{X}$. Then $H$ is not universally elliptic. 
\end{lemma}
\begin{proof}
Let $L$ be a transversal axis of $H$. By definition, there exists a vertical hyperplane $\mathsf{h}$ such that $L \cap \mathsf{h}$ is a singleton. Since $\mathsf{h}$ separates $\wt{X}$ and is either equal to or disjoint from its translates, it induces a splitting of $G$.
Let $T$ be the Bass-Serre tree of the splitting. Let $e$ be the edge stabilised by the stabiliser of $\mathsf{h}$. Note that the image of $e$ under $H$ then spans a line of $T$. Hence, $H$ is not elliptic in $T$. 
\end{proof}

Splittings induced by vertical lines need more careful treatment. They may or may not cross other vertical or transversal lines which induce splittings.
We present below one sufficient condition for a splitting induced by a vertical line (cycle) to be universally elliptic. 

\begin{prop} \label{prop_C_3_half-spaces_universal_elliptic}
Let $L$ be a line that separates $\wt{X}$ into at least three half-spaces. Then a subgroup of the stabiliser of $L$ is universally elliptic.
\end{prop}

\begin{proof}

Let $L'$ be a separating line such that $L$ and $L'$ meet at a compact segment $P$. Since $\partial N(P) \setminus L$ has at least 3 components, by \cref{cor_cut_pairs_3_components_no_crossings}, $L' \cap \partial N(P)$ lies in a component of $\partial N(P) \setminus L$. Hence, $L$ and $L'$ don't cross. In particular, $L$ does not cross any of its translates.
Let $H$ be a maximal subgroup of the stabiliser of $L$ that preserves a half-space of $L$. Then by \cref{prop_cyclic_splittings_induce_separating_lines}, $G$ splits over $H$ and $H$ is universally elliptic.
\end{proof}

\section{Repetitive cycles} \label{section_repetitive}

\begin{defn}\label{defn_pre-image_square_reg_sphere}
Let $\rho : P \to X$ be a combinatorial path. Let $e$ be an edge in $X$ and $e'$ an edge in $\rho^{-1}(e)$. Denote also by $e'$ the image of $e'$ in $N(P)$.
Given a square $\mathsf{s}$ in $X$ containing $e$, denote by $\mathsf{s}'$ (\cref{fig_preimage_square_reg_sphere}) the union of all squares meeting $e'$ in $N(P)$ whose image in $X$ is contained in $\mathsf{s}$.
Then the \emph{pre-image of $\mathsf{s}$ around $e'$} in $\partial N(P)$ is defined as the segment $\mathsf{s}' \cap \partial N(P)$.
\end{defn}

\begin{figure}
\begin{center}
\begin{tikzpicture}
\draw [dotted] (0,0.5) to (-4,0.5);
\draw [ultra thick] (-1,0.5) to (-3,0.5);
\node [below] at (-2,.5) {$e'$};
\draw [->] (-1.25,-.5) to (2,-1.25);
\draw [->] (2.5,.5) to (4.5,.5);
\node [above] at (3.5,.5) {$\rho$};
\node [above] at (-2.5,.5) {$P$};

\draw [step = 0.5] (3.5,-2.5) grid (4,-.5);
\draw [green] (4,-2.5) to (4,-.5);
\draw [ultra thick] (3.5,-2.5) to (3.5,-.5);
\node [left] at (3.5,-1.5) {$e'$};
\node [right] at (4,-2.5) {$\mathsf{s}' \cap \partial N(P)$};
\draw [->] (4.5,-1.5) to (6,-0.5);

\draw (7,-1) rectangle (9,1);
\draw [ultra thick] (7,-1) to (7,1);
\node [left] at (7,0) {$e$};
\node at (8,0) {$\mathsf{s}$};

\end{tikzpicture}
\end{center}
\caption{A pre-image of a square in the regular sphere}\label{fig_preimage_square_reg_sphere}
\end{figure}
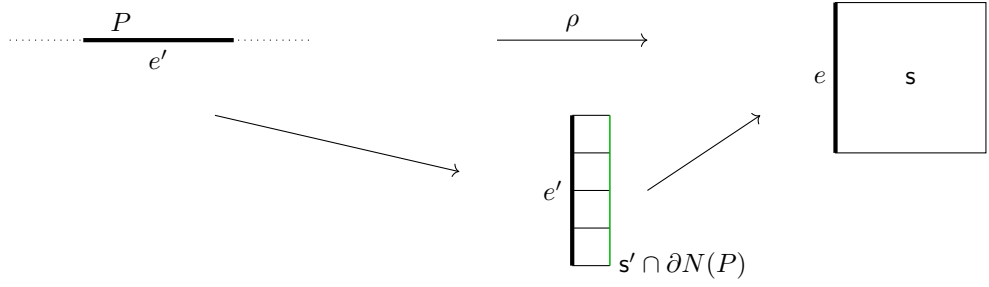

Recall that the orthogonal sphere of any fundamental domain $P_C$ of a UC-separating cycle $C$ contains at least two components (\cref{lemma_orth_sphere_uc_separating}). 
By \cref{lemma_edge_square_orthogonal_sphere}, for each component $K$ of $\partial_{orth} N(P_C)$ and each edge $e'$ in $P_C$ with image $e$ in $X$, there exists a square $\mathsf{s}$ containing $e$ such that the pre-image of $\mathsf{s}$ around $e'$ lies in $K$.

\begin{defn}[Repetitive cycles] \label{def_rep_cycle} Let $C$ be a UC-separating vertical cycle. $C$ is a \emph{$k$-repetitive cycle} if there exists a vertical edge $e$ in $X$ and a fundamental domain $P_C$ of $C$ such that 

\begin{enumerate}
\item \label{condition_1_repetitive} at least $k$ distinct edges $e_1, \cdots, e_k$ of $P_C$ are mapped to $e$, and
\item \label{condition_2_repetitive} for each square $\mathsf{s}$ containing $e$, there exists a component $K$ of $\partial_{orth} N(P_C)$ such that for each $i \in \{1, \cdots, k\}$, the pre-image of $\mathsf{s}$ around $e_i$ in $\partial N(P_C)$ lies in $K$.
\end{enumerate}
\end{defn}

Intuitively, if $C$ is $k$-repetitive, then the squares at $e$ do not `mix' in the components of $\partial_{orth} N(P_C)$. In other words, the notion of repetitiveness requires the cycle to not only `repeat' itself along some edges (Condition \ref{condition_1_repetitive}), but also to ensure that the partitions induced by the cycle on the set of squares containing $e_i$ coincide. 
Note that the definition depends on the choice of a fundamental domain, as illustrated in \cref{fig_domain_dependence_rep_cycle}.

\begin{fact}
A $k$-repetitive cycle is $k'$-repetitive for $1 \leq k' \leq k$.
\end{fact}

\begin{figure}
\begin{center}
\begin{tikzpicture} [scale = 0.6]
\begin{scope}[xshift = -10cm]
\draw [ultra thick] (0,0) circle [radius = 4];
\node at (0,0) {\fontsize{20}{22.4}$C$};

\draw [brown] (4.2,0.4) -- (3.8, 0.4);
\draw [brown] (4.2,-0.4) -- (3.8,-0.4);

\node [left] at (4.25,0) {$e_1$}; 
\node [right] at (4.5,0) {$\mathsf{s}'_1$};
\node [left] at (3.4,0) {$\mathsf{s}_1$};

\draw [brown] (0.4,4.2) -- (0.4,3.8);
\draw [brown] (-0.4,4.2) -- (-0.4,3.8);

\node at (0,4.25) {$e_2$}; 
\node [above] at (0,4.5) {$\mathsf{s}'_2$};
\node [below] at (0,3.4) {$\mathsf{s}_2$};

\draw [brown] (-4.2,0.4) -- (-3.8, 0.4);
\draw [brown] (-4.2,-0.4) -- (-3.8,-0.4);

\node [left] at (-3.8,0) {$e_3$};
\node [left] at (-4.6,0) {$\mathsf{s}'_3$};
\node [right] at (-3.4,0) {$\mathsf{s}_3$};

\draw [red] (0,-4.5) arc (-90:260:4.5);
\draw [blue] (0,-3.5) arc (-90:260:3.5);

\draw [red] (0,-4.5) to (-0.6,-3.45);
\draw [blue] (-0.8,-4.45) to (0,-3.5);

\node at (0,-6) {\parbox{0.3\linewidth}{\subcaption{The red and blue curves represent $\partial N(C)$}}};
\end{scope}

\begin{scope}[xshift = -2.5cm][scale = 0.7]
\draw (0,-1) to (0,1) to (-1.75,0.75) to (-1.75,-1.25) to (0,-1) to (1.75,-1.25) to (1.75,0.75) to (0,1);
\node [left] at (0,-0.7) {$e$};
\node at (1,0) {$\mathsf{s}'$};
\node at (-1,0) {$\mathsf{s}$};

\node at (0,-2) {\parbox{0.3\linewidth}{\subcaption{$e_i \mapsto e$ in $X$}}};
\node at (0,2) {$X$};
\end{scope}

\begin{scope}[xshift = 1.5cm]
\draw [ultra thick] (0,-4.5) to (0,4.5);
\draw [blue] (-.5,-4.5) to (-.5,3.5);
\draw [red] (-.5,3.5) to (-.5,4.5);
\draw [red] (.5,-4.5) to (.5,3.5);
\draw [blue] (.5,3.5) to (.5,4.5);

\draw [brown] (-0.1,-3.4) to (0.1,-3.4);
\draw [brown] (-0.1,-2.6) to (0.1,-2.6);

\draw [brown] (-0.1,-.4) to (0.1,-.4);
\draw [brown] (-0.1,.4) to (0.1,.4);

\draw [brown] (-0.1,3.4) to (0.1,3.4);
\draw [brown] (-0.1,2.6) to (0.1,2.6);

\node [left] at (0.2,-3) {$e_1$};
\node [left] at (0.2,0) {$e_2$};
\node [left] at (0.2,3) {$e_3$};

\node [left] at (-0.5,-3) {$\mathsf{s}_1$};
\node [right] at (0.5,-3) {$\mathsf{s}'_1$};

\node [left] at (-0.5,0) {$\mathsf{s}_2$};
\node [right] at (0.5,0) {$\mathsf{s}'_2$};

\node [left] at (-0.5,3) {$\mathsf{s}_3$};
\node [right] at (0.5,3) {$\mathsf{s}'_3$};

\node at (-0.5,-6) {\parbox{0.3\linewidth}{\subcaption{$\partial_{orth} N(P_C)$}}};
\end{scope}

\begin{scope}[xshift = 4.5cm]
\draw [ultra thick] (0,-4.5) to (0,4.5);
\draw [blue] (-.5,-4.5) to (-.5,1.5);
\draw [red] (-.5,1.5) to (-.5,4.5);
\draw [red] (.5,-4.5) to (.5,1.5);
\draw [blue] (.5,1.5) to (.5,4.5);

\draw [brown] (-0.1,-3.4) to (0.1,-3.4);
\draw [brown] (-0.1,-2.6) to (0.1,-2.6);

\draw [brown] (-0.1,-.4) to (0.1,-.4);
\draw [brown] (-0.1,.4) to (0.1,.4);

\draw [brown] (-0.1,3.4) to (0.1,3.4);
\draw [brown] (-0.1,2.6) to (0.1,2.6);

\node [left] at (0.2,-3) {$e_2$};
\node [left] at (0.2,0) {$e_3$};
\node [left] at (0.2,3) {$e_1$};

\node [left] at (-0.5,-3) {$\mathsf{s}_2$};
\node [right] at (0.5,-3) {$\mathsf{s}'_2$};

\node [left] at (-0.5,0) {$\mathsf{s}_3$};
\node [right] at (0.5,0) {$\mathsf{s}'_3$};

\node [left] at (-0.5,3) {$\mathsf{s'}_1$};
\node [right] at (0.5,3) {$\mathsf{s}_1$};
\node at (-0.5,-5) {\parbox{0.3\linewidth}{\subcaption{$\partial_{orth} N(P'_C)$}}};
\end{scope}

\end{tikzpicture}
\end{center}
\caption{$C$ is $3$-repetitive with the fundamental domain $P_C$ but not with the fundamental domain $P'_C$} \label{fig_domain_dependence_rep_cycle}
\end{figure}
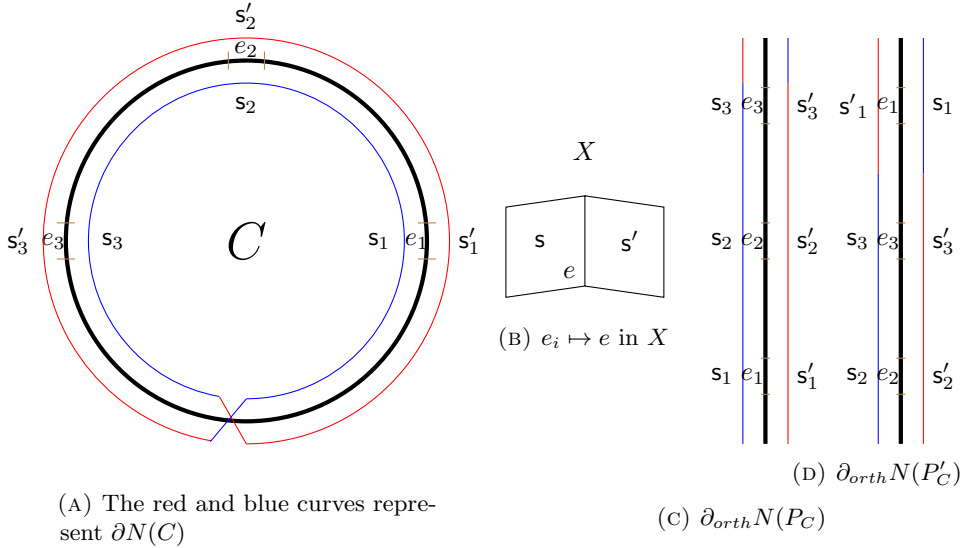

The following property of lifts of repetitive cycles will be crucial for the rest of the article. In fact, this is the only property of repetitive cycles that we will use. 

\begin{lemma}\label{lemma_rep_cycles_lifts}
Let $C$ be a $k$-repetitive cycle. Then there exists an edge $\tilde{e}$ in $\wt{X}$ and distinct elements $g_1, \cdots, g_k \in G$ such that \begin{enumerate}
\item for each $i \in \{1, \cdots, k\}$, $g_i \wt{C}$ contains $\tilde{e}$,
\item for each $i \in \{1, \cdots, k\}$, the translation length of $g_i$ is strictly less than the length of $C$, and
\item any two squares $\tilde{\mathsf{s}}$ and $\tilde{\mathsf{s}}'$ that contain $\tilde{e}$ are either separated by all translates $g_i \wt{C}$ for $i \in \{1, \cdots, k\}$, or by none of them. 
\end{enumerate}
\end{lemma}
\begin{proof}
Let $C$ be $k$-repetitive with fundamental domain $P_C$ so that there exist edges $e_1, \cdots, e_k$ in $P_C$ that satisfy the conditions of \cref{def_rep_cycle}. Let $\wt{P}_C \subset \wt{C}$ denote a lift of $P_C$ in $\wt{X}$. 
Denote $\tilde{e}_1$ in $\wt{P}_C$ by $\tilde{e}$. 
Since the edges $\tilde{e}_i$ in $\wt{P}_C$ all have the same image $e$ in $X$, there exist $1 = g_1, \cdots, g_k$ such that $g_i \tilde{e}_i = \tilde{e}$. Then clearly, for each $i \in \{1, \cdots, k\}$, the translation length of $g_i$ is strictly less than the length of $C$ and $g_i \wt{C}$ contains $\tilde{e}$.

Let $\tilde{\mathsf{s}}$ and $\tilde{\mathsf{s}}'$ be two squares that contain $\tilde{e}$. Then the squares $g_i^{-1}(\tilde{\mathsf{s}})$ and $g_i^{-1}(\tilde{\mathsf{s}}')$ in $\wt{P}_C$ are lifts of squares $\mathsf{s}_i$ and $\mathsf{s}'_i$ in $P_C$. Let $D$ and $D'$ be components of $\partial_{orth} N(P_C)$ such that the pre-image $s_1$ of $\mathsf{s}$ around $e_1$ meets $D$ and the pre-image $s'_1$ of $\mathsf{s}'$ around $e_1$ meets $D'$. By definition, the corresponding pre-image $s_i$ meets $D$ and $s'_i$ meets $D'$ for all $i$.
Now $\tilde{\mathsf{s}}= \tilde{\mathsf{s}}_1$ and $\tilde{\mathsf{s}}' = \tilde{\mathsf{s}}'_1$ lie in different half-spaces of $g_1\wt{C} = \wt{C}$ if and only if $D, D' \subset \partial N(\wt{C})$ (since $\partial_{orth} N(P_C) \hookrightarrow \partial N(\wt{C})$, by \cref{fact_orthogonal_sphere_and_its_lift}) meet different half-spaces of $\wt{C}$. Also, $\tilde{\mathsf{s}}$ and $\tilde{\mathsf{s}}'$ lie in different components of $g_i \wt{C}$ if and only if $\tilde{\mathsf{s}}_i = g_i^{-1}\tilde{\mathsf{s}}$ and $\tilde{\mathsf{s}}'_i = g_i^{-1} \tilde{\mathsf{s}}'$ lie in different half-spaces of $\wt{C}$ if and only if $D$ and $D'$ induce different half-spaces of $\wt{C}$. 
\end{proof}

As a consequence, we have the following useful result when at least two of the translates in the above lemma are not equal.
Let $k \geq 2$ and assume that $C$ is a $k$-repetitive cycle. Let $g_1, \cdots, g_k \in G$ be as in \cref{lemma_rep_cycles_lifts}.

\begin{lemma}\label{lemma_rep_separates_two_halfspaces}
Suppose that at least two translates $g_i\wt{C}$ and $g_j\wt{C}$ are distinct. Then $\wt{C}$ separates $\wt{X}$ into exactly two half-spaces.
\end{lemma}

We will need the following result on graphs with no cut points to prove the lemma.

\begin{lemma} \label{lemma_cut_pair_3_components_unique}
Let $\Gamma$ be a graph with no cut points. Let $\{a,b\}$ and $\{a',b'\}$ be cut pairs. Suppose there exist points $h_1, h_2, h_3 \in \Gamma \setminus \{a,b,a',b'\}$ such that they are pairwise separated by $\{a,b\}$ and also pairwise separated by $\{a',b'\}$. Then $\{a,b\} = \{a',b'\}$.
\end{lemma}

Compare with Lemma 3.8 of \cite{bowditch_jsj}.
\begin{proof}
By \cref{cor_cut_pairs_3_components_no_crossings}, $\{a,b\}$ lies in a half-space $Y'$ of $\{a',b'\}$ and \cref{lemma_crossings_cut_pairs_in_graphs} implies that $\{a',b'\}$ lies in a half-space $Y$ of $\{a,b\}$.
Let $Y', Y'_1, \cdots, Y'_n$ be the half-spaces of $\{a',b'\}$. If $\{a,b\} \neq \{a',b'\}$, then $Y'_1 \cup \cdots \cup Y'_n$ lies in the half-space $Y$ of $\{a,b\}$ that contains $\{a',b'\}$. This is a contradiction as at most one $h_i$ lies in $Y'$. 
\end{proof}

\begin{proof}[Proof of \cref{lemma_rep_separates_two_halfspaces}]
After a re-ordering if necessary, we assume that $\wt{C} = g_1\wt{C}$ and $g_2\wt{C}$ are distinct.
By \cref{lemma_half-spaces_equal_components_reg_sphere_L}, it is enough to show that $\partial N(C)$ has exactly two components.

By \cref{cor_P_embeds}, the segment $S = \wt{C} \cap g_2 \wt{C}$ is compact.
Suppose that $\wt{C}$ has at least three half-spaces. Then both $\partial N(S) \setminus \wt{C}$ and $\partial N(S) \setminus g_2\wt{C}$ have at least three components, by \cref{lemma_path_abundance}. This means that there exist three squares $\tilde{\mathsf{s}}$, $\tilde{\mathsf{s}}'$ and $\tilde{\mathsf{s}}''$ containing $\tilde{e}$ that meet different components of both $\partial N(S) \setminus \wt{C}$ and $\partial N(S) \setminus g_2\wt{C}$, by \cref{cor_square_meets_halfspace} (as any pair of squares at $\tilde{e}$ are either separated by both lines or by none, by \cref{lemma_rep_cycles_lifts}). \cref{lemma_cut_pair_3_components_unique} then implies that $\partial N(S) \cap \wt{C} = \partial N(S) \cap g_2 \wt{C}$, which is a contradiction. 
\end{proof}

\begin{defn}
A vertical cycle $\phi: C \to X_{s}$ is \emph{maximal} if it is not a non-trivial power of any cycle. 
\end{defn}

Note that any cycle which induces a maximal cyclic subgroup has to be maximal, justifying the name.

\begin{cor}\label{cor_maximal_rep_two_halfspaces}
A lift in $\wt{X}$ of a maximal $k$-repetitive cycle separates $\wt{X}$ into exactly two half-spaces whenever $k \geq 2$.
\end{cor}
\begin{proof}
When $C$ is maximal, the element $g_2$ that moves $\tilde{e}_2$ to $\tilde{e}_1$ in $\wt{P}_C$ does not preserve $\wt{C}$. This is because the translation length of $g_2$ is strictly less than the translation length of the generator of $\pi_1(C)$. If $g_2 \in $ stab($\wt{C}$), then $<g_2>$ and $\pi_1(C)$ are contained in a common cyclic subgroup and hence $\pi_1(C) \lneq$ stab($\wt{C}$), contradicting the fact that $C$ is maximal. Hence $g_2 \wt{C} \neq \wt{C}$. \cref{lemma_rep_separates_two_halfspaces} then gives the result.
\end{proof}

We will end the section with a crucial result that bounds the length of non-repetitive UC-separating cycles. 
Let $E$ denote the number of vertical edges of $X$ and $F$ denote the number of squares of $X$.

\begin{prop}[Long cycles are repetitive]\label{prop_long_cycles_repetitive}
Let $C$ be a vertical UC-separating cycle with length at least $E(k-1)2^{F(F+1)/2} +1$. Then $C$ is $k$-repetitive.
\end{prop}

\begin{proof}
The key ingredient in the proof is the pigeonhole principle. We apply it twice to show that $C$ satisfies both the conditions of \cref{def_rep_cycle}. We give the details below.

Fix a fundamental domain $P_C$ of $C$. Since there are $E$ vertical edges in $X$, by the pigeonhole principle, there exists an edge $e$ in $X$ such that $n$ oriented edges $e_1, \dots, e_n$ of $C$ are mapped to $e$, with $n \geq (k-1)2^{F(F+1)/2} +1$. 

Let $\lambda \leq F$ be the thickness of $e$. Note that the number of components $\mu$ of $\partial_{orth} N(P_C)$ is at most $\lambda$, by \cref{lemma_edge_square_orthogonal_sphere}.
We would like to show that there exist $k$ edges out of $e_1, \cdots, e_n$ for which the conditions of \cref{def_rep_cycle} are satisfied. Denote by $A(\lambda,\mu)$ the number of ways in which the squares $\mathsf{s}_1, \cdots, \mathsf{s}_{\lambda}$ containing $e$ can be partitioned into exactly $\mu$ nonempty subsets. If $n > (k-1) A(\lambda,\mu)$, then by the pigeonhole principle, $k$ edges which satisfy the conditions of \cref{def_rep_cycle} exist. Since $n \geq (k-1)2^{F(F+1)/2} +1$ and $\lambda \leq F$, it is enough to show that $A(\lambda,\mu) \leq 2^{\lambda(\lambda+1)/2}$.
Note that if $\mu = 1$, then $A(\lambda,\mu) = 1$ for any $\lambda$.
Also, since no subset of the partition of the squares can be empty, any subset can contain at most $\lambda - \mu +1$ squares. Hence we have 
\begin{align*}
A(\lambda,\mu) = \sum\limits_{r=1}^{\lambda-\mu+1} \binom{\lambda}{r} A(\lambda-r,\mu-1) 
\end{align*}

Observe that $A(\lambda_2,\mu) \geq A(\lambda_1,\mu)$ whenever $\lambda_2 \geq \lambda_1 \geq \mu$. Hence $A(\lambda,\mu) \leq 2^{\lambda(\lambda+1)/2}$. 
\end{proof}

\section{3-repetitive cycles and crossings}

The main result of the section is the following.

\begin{prop}\label{prop_repetitive_not_jsj}
Let $C$ be a maximal UC-separating vertical cycle that is 3-repetitive. Then there exists a periodic separating line $L'$ in $\wt{X}$ such that $L'$ and $\wt{C}$ cross. 
\end{prop}

If $C$ has self-crossings, then by definition, $\wt{C}$ and a translate cross, and there is nothing to show. The nontrivial part is to show that one such line exists even when $C$ has no self-crossings. Henceforth, till the end of this section, $C$ refers to a maximal 3-repetitive cycle $C$ with no self-crossings.

The key idea behind the proof is the following. By \cref{cor_maximal_rep_two_halfspaces}, $\wt{C}$ separates $\wt{X}$ into exactly two half-spaces. Further, by \cref{lemma_rep_cycles_lifts}, since $C$ is 3-repetitive, there exists an edge in $\wt{X}$ along which three translates of $\wt{C}$ meet. We will show that one of these translates separates the other two. The periodic line $L'$ will then be constructed by ensuring that it meets both the separated translates outside the central translate. This implies that $L'$ crosses the central translate of $\wt{C}$. We give details below.

\subsection{Lifts of \texorpdfstring{$C$}{C}} Fix a fundamental domain $P_C$ of $C$ and edges $e_1$, $e_2$, $e_3$ with image $e$ in $X_{s}$ satisfying the conditions of \cref{def_rep_cycle}. Let $d$ be the thickness of $e$.
Denote also by $e$ a lift of the edge $e$ in a vertex graph $\wt{X}_s$ of $\wt{X}$ such that for $1 \leq i \leq 3$, there exist $L_i = g_i \wt{C}$ that satisfy the conclusions of \cref{lemma_rep_cycles_lifts}. Thus $L_1$, $L_2$ and $L_3$ contain $e$ and if $\mathscr{S} = \{\mathsf{s}_1, \cdots, \mathsf{s}_d\}$ is the set of squares containing $e$, then

\begin{lemma}\label{lemma_partition_squares_e}
There exists a partition $A \sqcup B$ of $\mathscr{S}$ such that two squares $\mathsf{s}$ and $\mathsf{s}'$ in $\mathscr{S}$ lie in $A$ (or $B$) if and only if for each $i \in \{1,2,3\}$, they lie in a single half-space of $L_i$. \qed
\end{lemma}

By \cref{cor_maximal_rep_two_halfspaces}, we have 

\begin{lemma}\label{lemma_L_i_two_halfspaces}
Each $L_i$ separates $\wt{X}$ into exactly two half-spaces. \qed
\end{lemma}

Fix a square $\mathsf{s} \in A \subset \mathscr{S}$. For $1\leq i \leq 3$, let $Y_i$ be the half-space of $L_i$ that contains $\mathsf{s}$. Let $\overline{Y}_i$ be its complementary half-space.
A set-theoretic consequence of \cref{lemma_partition_squares_e} is that the half-spaces of $L_1$, $L_2$ and $L_3$ are nested. Namely, 
\begin{lemma}
For $i,j \in \{1,2,3\}$, either $Y_i \subset Y_j$ or $Y_j \subset Y_i$.
\end{lemma}
\begin{proof}
We will prove the lemma for $Y_1$ and $Y_2$. 
Observe that neither $\overline{Y}_1 \subset Y_2$ nor $Y_2 \subset \overline{Y}_1$, as otherwise a square in $B$ lies in $Y_2$ or a square in $A$ lies in $\overline{Y}_1$. Thus, by \cref{lemma_crossing_by_half-spaces}, either $Y_1 \subset Y_2$ or $Y_2 \subset Y_1$.
\end{proof}

After a re-ordering if necessary, assume that $Y_1 \subset Y_2 \subset Y_3$. We then have:
\begin{lemma}\label{lemma_L_2_separates_L_1_L_3}
$L_1$ and $L_3$ lie in complementary half-spaces of $L_2$.
\end{lemma}

\begin{proof}
First, $L_1 \subset Y_2$ as $Y_1 \subset Y_2$.
Similarly, $\overline{Y}_3 \subset \overline{Y}_2$ (as $Y_2 \subset Y_3$) implies $L_3 \subset \overline{Y}_2$.
\end{proof}

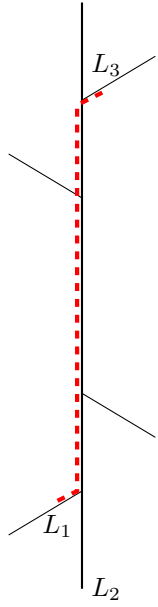
\begin{figure}
\begin{center}
\begin{tikzpicture} [scale = 0.65]
\draw [thick] (-2,-6) -- (-2,6);
\draw (-2,2) -- (-3.5,2.9);
\draw (-2,-2) -- (-.5,-2.9);
\draw (-2,4) -- (-0.5,4.9);
\draw (-3.5,-4.9) -- (-2,-4);

\node [below] at (-2.5, -4.3) {$L_1$};

\node [above] at (-1.5, 4.3) {$ L_3$};

\draw [ultra thick, dashed, red] (-2.5, -4.2) -- (-2.1,-4) to (-2.1,3.9) to (-1.5,4.2);

\node [right] at (-2,-6) {$L_2$};
 
\end{tikzpicture}
\end{center}
\caption{$L'$ crosses $L_2$ if it meets $L_1$ and $L_3$ outside $L_2$} \label{fig_crossing_line_L'}
\end{figure} 

\subsection{The main result} 

Fix orientations on $L_1, L_2, L_3$ such that they agree on $e$. As $L_1 \cap L_2$ is bounded, we can choose an element $h_1$ in the stabiliser of $L_1$ such that \begin{enumerate}
\item $h_1(Y_1) = Y_1$ (and thus $h_1(\overline{Y}_1) = \overline{Y}_1$), and 
\item $h_1(e)$ lies in $L_1 \setminus L_2$ before $e$ in the orientation of $L_1$.
\end{enumerate}
Recall that $h_1$ acts by translation on $L_1$.
Similarly, choose an element $h_3$ in the stabiliser of $L_3$ such that 
\begin{enumerate}
\item $h_3(Y_3) = Y_3$, and
\item $h_3(e)$ lies in $L_3 \setminus L_1$ after $e$ in the orientation of $L_3$.
\end{enumerate}

Let $L'$ be the axis of $h' = h_3 \cdot h_1^{-1}$ in the vertical tree $\wt{X}_s$ that contains $e$. $L'$ is periodic by definition. Observe that 

\begin{lemma}\label{lemma_L'_both_sides_L_2}
$L'$ contains $h_1(e)$ and, therefore, $h_3(e) = h'(h_1(e))$.
\end{lemma}
\begin{proof}
Suppose that $h_1(e)$ does not belong to $L'$. Let $h_1(u)$ be the initial vertex of $h_1(e)$ and $h_1(v)$ the final vertex. Let $\alpha$ be the geodesic from $h_1(u)$ to $L'$. Then $d(h_1(u), h'\cdot h_1(u)) = 2 \ell(\alpha) + \mid h'\mid$, where $\ell(\alpha)$ is the length of $\alpha$ and $\mid h'\mid$ is the translation length of $h'$ (see Proposition 24 in $\S$I.6.4 of \cite{serre}).
We also have $d(h_1(v), h'\cdot h_1(v)) = d(h_1(u), h'\cdot h_1(u)) - 2$.
But since $h' \cdot h_1(u) = h_3(u)$ and $h' \cdot h_1(v) = h_1(v)$ are adjacent, $d(h_1(v), h_3(v)) = d(h_1(u), h_3(u))$, which is a contradiction. Hence $h_1(e)$ lies in $L'$. 
\end{proof}

\begin{lemma}\label{lemma_L'_separating}
$L'$ is a separating line.
\end{lemma}
\begin{proof}
Let $C'$ be the cycle obtained by taking the quotient of $L'$ by the action of $<h'>$. We will show that $C'$ is strongly UC-separating. This will prove that $L'$ is separating (see \cref{lemma_regular_sphere_C_disconnected}).

Let $m$ be the midpoint of $e$. Subdivide $\wt{X}$ so that $m$, $h_1(m)$ and $h_3(m)$ are vertices of $L'$. Let $\sigma$ be the geodesic segment from $h_1(m)$ to $h_3(m)$. Since $h'\cdot h_1(m) = h_3(m)$ and $h'$ sends every element in the interior of $\sigma$ outside $\sigma$, $\sigma$ is a fundamental domain for $h'$ acting on $L'$ and hence a fundamental domain of $C'$. 
We will first show that $L'$ separates $\partial N(\sigma)$. 
Note that 
\begin{enumerate}
\item $\sigma = \sigma_1 \cdot e \cdot \sigma_3$, where $\sigma_1$ is the segment (see \cref{fig_L'_and_L_2}) in $L_1$ from $h_1(m)$ to the initial vertex of $e$ and $\sigma_3$ is the segment in $L_3$ from the final vertex of $e$ to $h_3(m)$.

\item By \cref{lemma_regular_sphere_at_combinatorial_path}, $\partial N(\sigma) \cong \partial N(\sigma_1) \bigoplus \partial N(\sigma_2)$, with labelling induced by the squares containing $e$.

\item $\partial N(\sigma_i) \setminus L_i = \partial N(\sigma_i) \setminus L'$ as both $L_i$ and $L'$ meet $\partial N(\sigma_i)$ at $e$ and $h_i(e)$.
\end{enumerate}

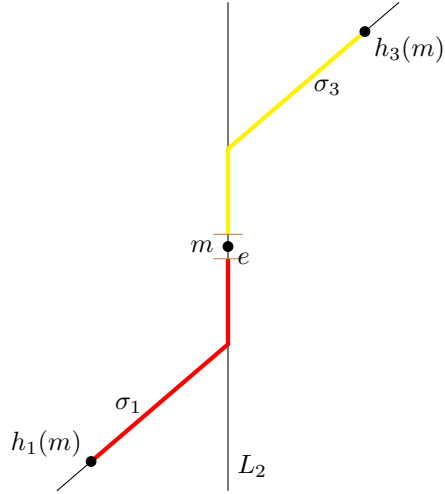
\begin{figure}
\begin{center}
\begin{tikzpicture} [scale = 0.65]
\draw (-2,-5) -- (-2,5);
\draw (-2,2) -- (1.5,5);
\draw (-2,-2) -- (-5.5,-5);

\draw [ultra thick, red] (-2,-.25) to (-2,-2) to (-4.8,-4.4);
\draw [ultra thick, yellow] (-2,.25) to (-2,2) to (0.8,4.4);

\draw [brown] (-1.7,-.25) to (-2.3,-.25);
\draw [brown] (-1.7,.25) to (-2.3,.25);

\node [right] at (-2,-4.5) {$L_2$};
\node [right] at (-2,-.25) {$e$};
\node [left] at (-2.1,0) {$m$};
\node [above right] at (-4.5,-3.6) {$\sigma_1$};
\node [below left] at (.5,3.6) {$\sigma_3$};
\node [above left] at (-4.8,-4.5) {$h_1(m)$};
\node [below right] at (.8,4.5) {$h_3(m)$};
 
\draw [fill] (-4.8,-4.4) circle [radius=0.1];
\draw [fill] (0.8,4.4) circle [radius=0.1];
\draw [fill] (-2,0) circle [radius=0.1];
\end{tikzpicture}
\end{center}
\caption{The segments $\sigma_1$ and $\sigma_3$} \label{fig_L'_and_L_2}
\end{figure}

Recall that $\partial N(\sigma_i) \setminus L_i$ is not connected (\cref{lemma_L_separates_partial_N(P)}) and $L_i$ has exactly two half-spaces (\cref{lemma_L_i_two_halfspaces}).
Thus $L_i$ induces a partition $A_i \sqcup B_i$ on the set of components of $\partial N(\sigma_i) \setminus L_i$ such that the components in $A_i$ meet one half-space of $L_i$ and the components in $B_i$ meet the other half-space of $L_i$.
Further, by \cref{lemma_partition_squares_e}, for each square $\mathsf{s} \in \mathscr{S}$, $\mathsf{s} \cap \partial N(\sigma_1)$ meets $A_1$ if and only if $\mathsf{s} \cap \partial N(\sigma_3)$ meets $A_3$. 
Therefore, there exists no path between a point in $A_1$ and a point in $B_3$ in the spliced graph $\partial N(\sigma) \setminus L'$. Hence $\partial N(\sigma)$ is separated by $L'$. Thus $\partial_{orth} N(\sigma)$ is not connected (\cref{fact_orthogonal_sphere_and_its_lift}).

As $h_1^{-1}$ preserves half-spaces of $L_1$, $h_1^{-1}$ sends a square containing $h_1(e)$ in $A_1$ ($B_1$) to a square containing $e$ in $A$ ($B$). Similarly, $h_3$ sends a square in $A$ ($B$) to $A_3$ ($B_3$).
In other words, there is no path between a point in $A_1$ and a point in $B_3$ in the quotient of $\partial_{orth} N(\sigma)$ by the action of $h'$. By \cref{lemma_orthogonal_sphere_fundamental_domain}, $\partial N(C')$ is not connected. 
\end{proof}

\begin{proof}[Proof of \cref{prop_repetitive_not_jsj}]

By \cref{lemma_L'_separating}, $L'$ is a separating line. By \cref{lemma_L'_both_sides_L_2}, $L'$ crosses $L_2$. 
\end{proof}

\section{An algorithm of double exponential time}
The main result of this section is the following theorem.

\begin{theorem}\label{thm_algo_G_hyp}
There exists an algorithm of double exponential time complexity that takes a Brady-Meier tubular graph of graphs $X$ with hyperbolic fundamental group $G$ as input and returns a finite list of splitting cycles that contains all universally elliptic cycles up to commensurability. 
\end{theorem}

For the rest of the article, we will also assume that $G$ is $\delta$-hyperbolic. Denote by $\partial \wt{X}$ the \emph{Gromov boundary} of $\wt{X}$. 
We refer the reader to \cite{bridsonhaefliger} for background on hyperbolic groups and the Gromov boundary.

As $G$ is hyperbolic and one-ended, Lemma 5.21 of \cite{bowditch_jsj} and \cref{prop_repetitive_not_jsj} imply that

\begin{prop} \label{prop_crossing_cycle_not_JSJ_hyp}
If $C$ is a maximal 3-repetitive UC-separating cycle, then $\pi_1(C)$ is not universally elliptic.
\end{prop}

\begin{proof}[Proof of \cref{thm_algo_G_hyp}]
Let $G$ and $X$ be given as in the statement. 
First note that by \cref{lemma_transversal_not_universally_elliptic}, every universally elliptic subgroup $H$ has a vertical axis in $\wt{X}$. This implies that there exists a splitting cycle $C$ in $X$ such that $\pi_1(C)$ and $H$ are commensurable (\cref{lemma_alg_splittings_commensurable_geom_splittings}).

Let $F$ be the number of squares of $X$ and $E$ the number of edges. By \cref{prop_long_cycles_repetitive}, any UC-separating cycle of length greater than $M = 2E(2^{F(F+1)/2}) \leq 8F(2^{F(F+1)/2})$ is 3-repetitive.
By \cref{prop_crossing_cycle_not_JSJ_hyp}, a universally elliptic cycle $C$ is either of length at most $M$, or it is a power of a maximal UC-separating cycle $C'$ of length at most $M$. If it is the latter, then $C'$ is not universally elliptic because it is not strongly UC-separating. By \cref{lemma_power_uc_separatig_is_separating}, $C$ is an $n^{th}$ power of $C'$, where $n$ is bounded by the maximal thickness of an edge of $C'$. Thus the length of $C$ is at most $F.M$.

There exist finitely many cycles of length at most $F.M$ in $X$. Thus our algorithm takes each cycle from this finite list as input and returns whether this cycle is strongly UC-separating with no self-crossings or not. By \cref{lemma_splitting_cycle_commensurable}, we thus have a list of all universally elliptic cycles up to commensurability.

The time taken by this algorithm is calculated as follows: 
\begin{enumerate}
\item The number of cycles of length at most $F.M$ is bounded by a number which is exponential in $F.M$ (see \cite{counting_cycles} for instance). This is of the order of a double exponential in $F$ as $M$ is itself exponential in $F$. 

\item The regular sphere around a cycle $C$ of length $k$ is a spliced graph of the regular spheres around its $k$ vertices (\cref{lemma_regular_sphere_at_combinatorial_path}), and the number of vertices and edges in this regular sphere is bounded by a constant times the number $F$ of squares of $X$. Finding whether this sphere is connected is linear in $F$, by \cite{hopcroft_connected_graph}.

\item A cycle $C$ has a self-crossing if there exists self-crossing at a component of self-intersection $P \subset C$ (\cref{defn_self_intersection}). There is a self-crossing at $P$ only if a subpath of $C$ meets $\partial N(P) \cap \partial N(C)$ in different components (\cref{fact_sphere_P_embeds_sphere_C}). This information is available when the regular sphere around $C$ is computed and does not cost any additional time.
\end{enumerate}

The algorithm thus takes double exponential time in the number of squares of $X$.
\end{proof}

\section{Constructing a JSJ complex} \label{section_jsj_complex}

The goal of this section is to construct from $X$ a tubular graph of graphs $X_{jsj}$ whose graph of groups structure gives the JSJ decomposition of $G$.

\subsection{Splitting cycles as hyperplanes} 

Let $\phi: C \to X_s \subset X$ be a splitting vertical cycle. We will show how to modify $X$ to a tubular graph of graphs $X_C$ such that $\pi_1(X) \cong \pi_1(X_C)$ and $\pi_1(C)$ is commensurable with the cyclic group generated by a vertical hyperplane of $X_C$.

We perform this construction at the level of universal covers using the machinery of spaces with walls \cite{haglund_paulin_walls} (utilised earlier in \cref{section_dual_trees}).
We refer the reader to \cite{nica_cubulation} and \cite{chatterji_niblo_cubulation} for details on constructing $\cat$ cube complexes from spaces with walls.

Recall that $\wt{X}$ is the cube complex dual to the space with walls $(\wt{X}^0, \mathcal{H})$ \cite{roller_thesis}, where $\wt{X}^0$ denotes the 0-skeleton of $\wt{X}$ and $\mathcal{H}$ the set of hyperplanes of $\wt{X}$. For our purposes, we slightly modify the space with walls as follows.
First we attach a strip $S_L = \mathbb{R} \times [0,1]$ isomorphically along $\mathbb{R} \times \{0\}$ to each translate $L$ of $\wt{C}$. Note that there is a natural square structure on $S_L$ so that every horizontal hyperplane of $\wt{X}$ that meets $L$ naturally extends to $S_L$. 
Let $Z$ be the set of open horizontal half-edges of the union of $\wt{X}$ and the attached strips. Then the vertical and horizontal hyperplanes of $\wt{X}$ induce a space with walls $(Z,W)$. 

Note that we do not add the vertical hyperplanes through the strips $S_L$ to the collection $W$. Thus the dual cube complex of $(Z,W)$ is nothing but $\wt{X}$ (Theorem 10.3 of \cite{roller_thesis}).
We now enrich $W$ to $W_C$. The walls in $W_C$ are determined by the following:

\begin{enumerate}
\item[(\textit{i})] the horizontal hyperplanes of $\wt{X}$,
\item[(\textit{ii})] the vertical hyperplanes of $\wt{X}$, and
\item[(\textit{iii})] the $G$-translates of $\wt{C}$.
\end{enumerate}

Note that the elements of type (\textit{i}) and (\textit{ii}) induce $W$, where each half-space $Y$ in $\wt{X}$ of an element of type (\textit{i}) or (\textit{ii}) defines a wall $\{Y \cap Z, Y^c \cap Z\}$. 
Given a translate $L$ of $\wt{C}$ in $\wt{X}$, each half-space $Y$ in $\wt{X}$ of $L$ defines a wall $\{Y \cap Z, Y^c \cap Z\}$ of type (\textit{iii}). Thus $L$ induces exactly $K$ walls in $Z$ if it has $K$ half-spaces in $\wt{X}$.

\begin{lemma}
$(Z,W_C)$ is a space with walls. \qed
\end{lemma}

Denote by $\wt{X}_C$ the $\cat$ cube complex dual to $Z$.

\begin{lemma} \label{lemma_X_C_VH}
$\wt{X}_C$ is a $\VH$-complex.
\end{lemma}
The proof uses the following observation, which is a consequence of \cref{lemma_crossing_by_half-spaces}.

\begin{lemma} \label{lemma_noncrossing_lines_wall_intersection}
Let $L$ and $L'$ be two non-crossing lines of $\wt{X}$. Then given half-spaces $Y$ of $L$ and $Y'$ of $L'$, at least one of the following four intersections is empty: $\mathring{Y} \cap \mathring{Y'}$, $Y^c \cap \mathring{Y'}$, $\mathring{Y} \cap Y'^c$ and $Y^c \cap Y'^c$.
\end{lemma}

Two walls $\{Y, Y^c\}$ and $\{Y', Y'^c\}$ in a space with walls \emph{cross} (\cite{chatterji_niblo_cubulation}) if all four intersections $Y \cap Y'$, $Y^c \cap Y'$, $Y \cap Y'^c$ and $Y^c \cap Y'^c$ are non-empty.

\begin{proof}[Proof of \cref{lemma_X_C_VH}]
Two walls of type (\textit{i}) don't cross as two horizontal hyperplanes of $\wt{X}$ are either equal or disjoint. Similarly, two walls of type (\textit{ii}) don't cross. By \cref{lemma_noncrossing_lines_wall_intersection}, two walls of type (\textit{iii}) don't cross either. Further, a wall of type (\textit{ii}) and a wall of type (\textit{iii}) don't cross since a vertical line is disjoint from any vertical hyperplane.
By Proposition 4.6 of \cite{nica_cubulation}, there exists a bijective correspondence between the hyperplanes of $\wt{X}_C$ and the walls of $(Z,W_C)$. Further, two hyperplanes in $\wt{X}_C$ intersect if and only if the corresponding walls cross.
Declare an edge $e$ of $\wt{X}_C$ to be vertical if and only if the hyperplane through $e$ corresponds to a wall of type (\textit{i}). Otherwise, declare the edge to be horizontal. No square contains two adjacent edges of the same type as otherwise two hyperplanes of the same type or two hyperplanes of type (\textit{ii}) and (\textit{iii}) intersect. 
\end{proof}

Observe that there exists a natural $G$-equivariant map $\hat{\eta}_C: \wt{X}_C \to \wt{X}$ such that the following diagram commutes: \begin{tikzcd}
(Z,W_C) \arrow[r, "id"] \arrow[d]
& (Z,W) \arrow[d] \\
\wt{X}_C \arrow[r, "\hat{\eta}_C"]
& \wt{X}
\end{tikzcd}

Since $W \subset W_C$, $\hat{\eta}_C$ takes any vertex (ultrafilter) $\sigma'$ of $\wt{X}_C$ to a vertex $\sigma' \cap W$ of $\wt{X}$. But every vertex of $\wt{X}$ is a principal ultrafilter, and hence $\hat{\eta}_C(\sigma') = \sigma_z$ for some $z$. By the way the set of walls $W_C$ was defined on $Z$, we have the following result:
\begin{lemma} \label{lemma_eta_C_properties}
The map $\hat{\eta}_C$ has the following properties: \begin{enumerate}
\item Let $\mathsf{c}$ be a cell of $\wt{X}$ that does not meet any translate of $\wt{C}$. Then $\hat{\eta}_C$ restricted to $\hat{\eta}_C^{-1}(\mathsf{c})$ is injective.

\item It sends vertical edges to vertical edges and horizontal edges to either horizontal edges or vertices.

\item A horizontal edge is mapped to a vertex if and only if the vertical hyperplane through this edge is induced by a wall of type \textit{(iii)}. \qed
\end{enumerate} \end{lemma}

\begin{lemma}\label{lemma_opening__cofinite}
For any $z \in Z$, $\hat{\eta}_C^{-1}(\sigma_z)$ is a finite horizontal tree. Further, the edges in the pre-image of $\sigma_z$ are dual to vertical hyperplanes induced by translates of $\wt{C}$ that meet $\sigma_z$ in $\wt{X}$.
\end{lemma}

\begin{proof}
Let $\sigma'_1$ and $\sigma'_2$ be two vertices of $\hat{\eta}_C^{-1}(\sigma_z)$. Let $\{Y,Y^c\}$ be a wall such that $Y \in \sigma'_1$ and $Y^c \in \sigma'_2$. Then clearly, $\{Y,Y^c\}$ is a wall of type (\textit{iii}). Let $L$ be the line that defines $\{Y,Y^c\}$. 
We claim that $L$ passes through the vertex $\sigma_z$ in $\wt{X}$. If not, then let $\mathsf{h}$ be a hyperplane of $\wt{X}$ that separates $L$ from $\sigma_z$. Let $Y_{\mathsf{h}}$ be a half-space of $\mathsf{h}$ that contains the vertex $\sigma_z$. Then $Y_{\mathsf{h}} \in \sigma_z$, the ultrafilter. Clearly, this implies that $Y_{\mathsf{h}} \in \sigma'_1$ and $Y_{\mathsf{h}} \in \sigma'_2$. Since $L$ and $\mathsf{h}$ are disjoint, either $Y_{\mathsf{h}} \subset Y$ or $Y_{\mathsf{h}} \subset Y^c$. Thus either $Y \in \sigma'_1$ and $Y \in \sigma'_2$ or $Y^c \in \sigma'_1$ and $Y^c \in \sigma'_2$, a contradiction. 
So $L$ has to pass through $\sigma_z$. There are only finitely many translates of $\wt{C}$ that meet at any given point of $\wt{X}$. This proves the result.
\end{proof}

Since $\hat{\eta}_C$ is a finite-to-one $G$-equivariant map, we conclude that
\begin{lemma}
$G$ acts geometrically on $\wt{X}_C$. \qed
\end{lemma}

\begin{lemma}\label{lemma_X_C_tubular}
Every vertical hyperplane of $\wt{X}_C$ is a line.
\end{lemma}
\begin{proof}
The stabiliser of a vertical hyperplane is the stabiliser of a wall of either type (\textit{ii}) or type (\textit{iii}), and hence is a cyclic subgroup. Thus every vertical hyperplane is a line.
\end{proof}

The complex $\wt{X}_C$ consists of two types of subcomplexes:
\begin{itemize}
\item Denote by $\wt{Z}_C$ a connected component of the subcomplex of $\wt{X}_C$ consisting of the union of the first cubical neighbourhood of all hyperplanes corresponding to walls of type \textit{(iii)}. In other words, $\wt{Z}_C$ is a connected component of the closed strips in $\wt{X}_C$ induced by half-spaces of translates of $\wt{C}$.

\item The second type of subcomplex, denoted by $\wt{Y}_C$, is the closure of the complement in $\wt{X}_C$ of the $G$-translates of $\wt{Z}_C$.
\end{itemize}

\begin{lemma}
The subcomplex $\wt{Z}_C$ is a tree of finite trees whose underlying tree $\hat{\eta}_C(\wt{Z}_C)$ is a copy of $\wt{\phi(C)}$.
\end{lemma}
\begin{proof}
By \cref{lemma_opening__cofinite}, $\hat{\eta}_C^{-1}(\sigma_z)$ is a finite tree for every vertex of $\wt{X}$, and thus after subdivision, for the midpoint of every edge of $\wt{X}$. The horizontal edges of $\wt{Z}_C$ are all dual to hyperplanes of type \textit{(iii)}. Note that $\hat{\eta}_C$ sends such horizontal edges to vertices and vertical edges to vertical edges (\cref{lemma_eta_C_properties}). Thus $\hat{\eta}_C(\wt{Z}_C)$ is a tree and $\wt{Z}_C$ is a tree of finite trees.

We now claim that $\hat{\eta}_C(\wt{Z}_C)$ is a copy of the universal cover of $\phi(C)$. Note that $\wt{\phi(C)}$ is a connected union of lines which are translates of $\wt{C}$. Since $\hat{\eta}_C(\wt{Z}_C)$ is also a union of translates of $\wt{C}$ with image $\phi(C)$ in $X$, $\hat{\eta}_C(\wt{Z}_C) \subset \wt{\phi(C)}$. Conversely, if a vertex $v$ of a translate $L$ of $\wt{C}$ is contained in $\hat{\eta}_C(\wt{Z}_C)$, then $\hat{\eta}^{-1}_C(v)$ meets the strips induced by half-spaces of $L$ and thus these strips are contained in $\wt{Z}_C$. The image of any such strip under $\hat{\eta}_C$ is $L$ and thus $L \subset \hat{\eta}_C(\wt{Z}_C)$. 
\end{proof}

Define $X_C \coloneqq \wt{X}_C / G$. By \cref{lemma_X_C_tubular}, $X_C$ is a tubular graph of graphs. The space $X_C$ is called the \emph{opened-up space of $X$ along $C$}. The $G$-equivariant map $\hat{\eta}_C : \wt{X}_C \to \wt{X}$ induces a map $\eta_C : X_C \to X$.
Let $Y_C$ and $Z_C$ denote the respective images of $\wt{Y}_C$ and $\wt{Z}_C$ in $X_C$. We have proved that

\begin{lemma}\label{lemma_Z_C_same_Z'_C}
$Z_C$ is a graph of finite trees with underlying graph $\phi(C)$ and with the following property:
If $u \in \phi(C)$ is a vertex (or a midpoint of an edge), then the vertex (edge) tree $T(u)$ is the tree dual to $Z$ with the walls induced by translates of $\wt{C}$ passing through a lift $\tilde{u}$ of $u$ in $\wt{X}$.
\end{lemma}

We conclude with the following observation:
\begin{lemma}\label{lemma_X_C_union_Y_C_Z_C}
The opened-up space $X_C$ is a union of the subcomplexes $Y_C$ and $Z_C$ with $Y_C \cap Z_C$ consisting of those cells of $Y_C$ that are mapped by $\eta_C$ to $\phi(C)$. \qed
\end{lemma}

\subsection{Algorithmic construction of \texorpdfstring{$X_C$}{Xc}}
The main result of this subsection is the following.

\begin{theorem}\label{thm_algorithm_X_C}
There exists an algorithm of exponential time complexity that takes a Brady-Meier tubular graph of graphs $X$ and a splitting cycle $\phi : C \to X$ as input and returns the opened-up space $X_C$ along $C$ as output.
\end{theorem}

Define a complex $Y'_C$ as the square complex obtained from $X \setminus \phi(C)$ by ``completing the missing cells'' as follows: for each vertex or edge $x$ of $\phi(C)$, take as many copies of $x$ as the number of squares of $X$ that contain $x$ and add them to the semi-open squares of $X \setminus \phi(C)$ to obtain closed squares. Observe that
\begin{lemma} \label{lemma_Y'_C_same_Y_C}
$Y'_C$ is isomorphic to $Y_C$. \qed
\end{lemma}
Therefore in order to construct $X_C$ algorithmically, it only remains to construct $Z_C$.
The first result we will need is the following. Fix a lift $\wt{C}$ of the splitting cycle $\phi : C \to X_s \subset X$. Let $K$ be the number of half-spaces of $\wt{C}$ in $\wt{X}$. 

\begin{prop} \label{lemma_ball_K_halfspaces}
There exists $D' \in \mathbb{N}$ such that for any vertex or (midpoint of an edge) $v \in \wt{C}$ and $\forall D \geq D'$, the $D^{th}$ cubical neighbourhood $\{v\}^{+D}$ of $v$ has the following properties:
\begin{enumerate}
\item For each translate $g\wt{C}$ such that $v \in g\wt{C}$, $g\wt{C}$ separates $\{v\}^{+D}$ into exactly $K$ components.
\item For every $g, g' \in G$ such that $g\wt{C} \neq g'\wt{C}$ and $v \in g\wt{C} \cap g'\wt{C}$, $g\wt{C} \cap \{v\}^{+D} \neq g'\wt{C} \cap \{v\}^{+D}$.
\end{enumerate}
\end{prop}

The main ingredient for proving \cref{lemma_ball_K_halfspaces} is the following result.
Let $N$ be such that the thickness of any edge of $X$ is at most $N$. 
\begin{lemma} \label{lemma_exponential_power_fund_domain_half_spaces}
Let $C_N$ be a $2^{N^{th}}$ power of $C$ and $P_N$ a fundamental domain of $C_N$. Then there exists a natural bijection between the set of half-spaces of $\wt{C}$ and the set of components of $\partial_{orth} N(P_N)$.
\end{lemma}

We first prove a preliminary result on the number of connected components of graphs. Let $\Gamma$ be a graph with no cut points and $\{a,b\}$ a cut pair. Assume that the valence $n$ of $a$ is equal to the valence of $b$. We will construct a spliced graph (\cref{defn_graph_sum}) of finitely many copies of $\Gamma$. Let $\phi_a : \{1, \cdots, n\} \to \mathrm{adj}(a)$ and $\phi_b : \{1, \cdots, n\} \to \mathrm{adj}(b) $ denote labellings of vertices adjacent to $a$ and $b$.
For each $i \in \mathbb{N}$, let $\Gamma_i$ be a copy of $\Gamma$ with the corresponding cut pair $\{a_i,b_i\}$. We will denote the labellings on the adjacent vertices by $\phi_{a_i}, \phi_{b_i}$.
Let $\Gamma'_i := \Gamma_1 \; {}_{(b_1,\phi_{b_1})} \! \bigoplus_{(a_2, \phi_{a_2})} \Gamma_2 \; {}_{(b_2,\phi_{b_2})} \! \bigoplus \cdots \bigoplus_{(a_i, \phi_{a_i})} \Gamma_i$.

\begin{lemma}\label{lemma_graph_sum_same_components}
Suppose that the number of components of $\Gamma \setminus \{a,b\}$ is equal to the number of components of $\Gamma'_2 \setminus \{a_1,b_2\}$. Then for each $i$, the number of components of $\Gamma'_i \setminus \{a_1,b_i\}$ is equal to the number of components of $\Gamma'_2 \setminus \{a_1, b_2\}$.
\end{lemma}

\begin{proof}
Let $k$ be the number of components of $\Gamma \setminus \{a,b\}$.
$\Gamma'_2 \setminus \{a_1,b_2\}$ has the same number of components as $\Gamma \setminus \{a,b\}$ if and only if there is a partition into $k$ subsets of $\{1, \cdots, n\}$ such that the corresponding partition induced by $\phi_a : \{1, \cdots, n\} \to \mathrm{adj}(a)$ and $\phi_b : \{1, \cdots, n\} \to \mathrm{adj}(b) $ on the vertices adjacent to $a$ and $b$ coincides with the partition induced by the $k$ components of $\Gamma \setminus \{a,b\}$. Continuing iteratively, we obtain the result.
\end{proof}

\begin{proof}[Proof of \cref{lemma_exponential_power_fund_domain_half_spaces}]
Denote by $P_k$ a fundamental domain of a $2^{k^{th}}$ power $C_k$ of $C$. Assume that for each $k$, $P_k$ has been chosen such that $P_k$ is a concatenation of two copies of $P_{k-1}$. The result then follows from \cref{lemma_graph_sum_same_components} and \cref{lemma_regular_sphere_at_combinatorial_path}.
\end{proof}

Let $L$ be a line in $\wt{X}$ and $v \in L$ a vertex. Let $D \in \mathbb{N}$. Note that $\{v\}^{+D}$ is a $\cat$ cube (sub)complex \cite{haglund_wise_special}. We will assume that $\partial N(L) \subset \wt{X}$.
By \cref{lemma_half-spaces_equal_components_reg_sphere_L}, we have

\begin{lemma}\label{lemma_components_reg_sphere_subcomplex}
There exists a bijection between the half-spaces of $L \cap \{v\}^{+D}$ in $\{v\}^{+D}$ and the components of $\partial N(L) \cap \{v\}^{+D}$. \qed
\end{lemma}

\begin{proof}[Proof of \cref{lemma_ball_K_halfspaces}]
Let $D = \ell(C) 2^N$, where $\ell(C)$ denotes the length of $C$. Then $\{v\}^{+D}$ contains a lift of $P_N$, a fundamental domain of a $2^{N^{th}}$ power of $C$. By \cref{lemma_exponential_power_fund_domain_half_spaces}, $\wt{C}$ separates $\partial N(P_N) \subset \{v\}^{+D}$ into exactly $K$ components. 
\cref{lemma_components_reg_sphere_subcomplex} then implies conclusion (1). Conclusion (2) follows from \cref{lemma_length_lcm}.
\end{proof}

\subsubsection{Construction of \texorpdfstring{$X_C$}{XC}}

Choose a basepoint $v \in \phi(C)$ with lift $\tilde{v}$ in $\wt{C}$. Let $\mathcal{B} := \{\tilde{v}\}^{+D}$ in $\wt{X}$.
Note that

\begin{lemma}\label{lemma_construction_finite_ball}
There exists an algorithm that takes $X$, $C$ and $v$ as input and returns $\mathcal{B}$ in exponential time. \qed
\end{lemma}

For each translate $L$ meeting $\mathcal{B}$, we attach a finite strip $S_L \cong L \cap \mathcal{B} \times [0,1]$. Let $Z = Z_{\mathcal{B}}$ be the set of open horizontal half-edges in the union of $\mathcal{B}$ and the collection of finite strips. For each vertex $u$ (edge $e$) in $\phi(C)$, we define a set of walls $W_u$ ($W_e$) on $Z$ as follows. First fix a fundamental domain $P$ of $C$ in $\wt{X}$ such that $\tilde{v} \in P$. Choose a lift of $u$ ($e$) in $P$. 
Let $\wt{C} = L_1, \cdots L_n$ be the translates of $\wt{C}$ that pass through $\tilde{u}$ ($\tilde{e}$). By \cref{lemma_ball_K_halfspaces}, each line $L_i$ separates $\mathcal{B}$ into exactly $K$ components. Thus each line $L_i$ induces $K$ walls of $W_u$ ($W_e$) on $Z$, where each half-space $Y$ in $\wt{X}$ of $L_i$ defines a wall $\{Y \cap Z, Y^c \cap Z\}$. 

Since $C$ has no self-crossing, no two walls of $W_u$ ($W_e$) cross (\cref{lemma_noncrossing_lines_wall_intersection}). Thus the dual cube complex of $(Z,W_u)$ (respectively $(Z,W_e)$) is a tree, denoted by $T(u)$ ($T(e)$).
Note that the definition of $W_u$ (and $W_e$) is independent of the choice of $\tilde{v}$ or of the choice of $\tilde{u}$ ($\tilde{e}$) in $P\subset \mathcal{B}$.

Suppose that an edge $e$ is incident to a vertex $u$ in $\phi(C)$. Let $\tilde{e}$ with incident vertex $\tilde{u}$ be the corresponding lifts in $P \subset B$. Since every translate of $\wt{C}$ that passes through $\tilde{e}$ also passes through $\tilde{u}$, there exists a natural inclusion $W_e \subset W_u$. Further, given translates $L_1$, $L_2$ that contain $\tilde{e}$ with half-spaces $Y_1$, $Y_2$ such that $Y_1 \subset Y_2$, suppose there exists a translate $L'$ that meets $\tilde{u}$ with half-space $Y'$ such that $Y_1 \subset Y' \subset Y_2$. Then it is easy to see that $L'$ contains $\tilde{e}$. Thus we have

\begin{lemma}
Given an edge $e$ in $\phi(C)$ incident to a vertex $u$, there exists a natural inclusion $T(e) \hookrightarrow T(u)$.
\end{lemma}

Let $Z'_C$ denote the geometric realisation of the graph of trees ($\phi(C)$, $\{T(u)\}$, $\{T(e)\})$.

\begin{prop} \label{prop_X_C_iso_X'_C}
There exists a natural isomorphism between the square complexes $Z_C$ and $Z'_C$.
\end{prop}
\begin{proof}
By \cref{lemma_Z_C_same_Z'_C}, $Z_C$ is a graph of finite trees with underlying graph $\phi(C)$. So is $Z'_C$. Further, the wall structures that define vertex and edge trees of $Z_C$ and $Z'_C$ are isomorphic: Indeed, the walls that define $T(u)$ for $u \in \phi(C)$ in $Z_C$ are induced by half-spaces in $\wt{X}$ of translates of $\wt{C}$ that pass through a lift $\tilde{u}$ of $u$. In $Z'_C$, the tree is defined by walls induced by half-spaces of translates of $\wt{C}$ in a finite ball $\mathcal{B}$ of $\wt{X}$ containing $\tilde{u}$. Since there exists a bijection between the half-spaces of $\wt{C}$ in $\mathcal{B}$ and the half-spaces of $\wt{C}$ in $\wt{X}$ (\cref{lemma_components_reg_sphere_subcomplex}), $Z_C$ is isomorphic to $Z'_C$. 
\end{proof}

\begin{proof}[Proof of \cref{thm_algorithm_X_C}]
The compact space $\mathcal{B}$ can be constructed in exponential time from $X$ (\cref{lemma_construction_finite_ball}). It costs exponential time to calculate the number of half-spaces of any translate of $\wt{C}$ (\cref{lemma_exponential_power_fund_domain_half_spaces}). 

Since the number of translates of $\wt{C}$ meeting at any point of $\wt{X}$ is bounded by the length of $C$ (by \cref{lemma_length_lcm}), the dual trees $T(u)$ ($T(e)$) of all vertices $u$ (edges $e$) in $\phi(C)$ can be constructed in polynomial time. Thus $Z'_C$ is constructed in exponential time. $Y'_C$ is constructed in linear time in $X$ and $X_C$ is obtained in linear time from $Y_C \cong Y'_C$ and $Z_C \cong Z'_C$. Hence the result.
\end{proof}

\subsubsection{Structure of \texorpdfstring{$X_C$}{Xc}}

\begin{lemma}
The tree $T(u)$ (or $T(e)$) is a bipartite tree with black vertices having valence exactly $K$.
\end{lemma}

\begin{proof}
We will first show that there exist vertices of valence $K$ in $T(u)$ ($T(e)$) and then show that the tree is bipartite.
Let $L$ be a translate of $\wt{C}$ passing through $\tilde{u}$ ($\tilde{e}$) in $\mathcal{B}$. Let $Y_1, \cdots, Y_K$ be the half-spaces of $L$. Let $z$ be an open horizontal half-edge in the strip $S_L$. Denote by $\sigma_L$ the ultrafilter $\sigma_z$ in $T(u)$ ($T(e)$).
Thus $\sigma_L$ contains $\{Y_1^c, \cdots, Y_K^c\}$ and exactly those half-spaces of translates of $\wt{C}$ passing through $\tilde{u}$ that contain $L$. 
Observe that the valence of $\sigma_L$ is at least $K$: switching each half-space $Y_i$ of $L$ gives an edge incident to $\sigma_L$. It is easy to see that it is exactly $K$ as no edge of the form $ \{Y', Y'^c\}$ can be incident to $\sigma_L$ when $Y' \neq Y_i$.

Let $\sigma'$ be a vertex at distance two from $\sigma_L$. Let $\sigma' \triangle \sigma_L = \{Y_1,Y_1^c, Y_1', Y_1^{'c}\}$, where $Y'_1$ is a half-space of a translate $L'$ of $\wt{C}$. We will show that $\sigma' = \sigma_{L'}$.
Assume that $Y'_1 \in \sigma_L$. This implies that for each half-space $Y'_i$ of $L'$ with $i \neq 1$, $Y^{'c}_i \in \sigma_L$ and hence in $\sigma'$. Further, $Y_1^{'c} \in \sigma'$ by assumption. If $\sigma' \neq \sigma_{L'}$, then any path from $\sigma'$ to $\sigma_{L'}$ involves a change of half-spaces of the type $\{Y'_i, Y_i^{'c}\}$. Hence we conclude that $\sigma' = \sigma_{L'}$. The case when $Y'_1 \notin \sigma_L$ is similar and we leave it as an exercise.
This proves that a vertex in $T(u)$ ($T(e)$) is of the form $\sigma_{L'}$ if and only if it is at even distance from $\sigma_L$. Thus the tree is bipartite.
\end{proof}

Let $u$ ($e$) be in $\phi(C)$. Let $v_1$ and $v_2$ be black vertices in $T(u)$ ($T(e)$). 

\begin{lemma}
Given an edge $e_1$ incident to $v_1$ in $T(u)$ ($T(e)$), there exists an edge $e_2$ incident to $v_2$ in $T(u)$ ($T(e)$) such that the hyperplane in $Z'_C$ dual to $e_1$ is equal to the hyperplane dual to $e_2$.
\end{lemma} 
\begin{proof}
Let $v_i$ be the ultrafilter $\sigma_{L_i}$, where $L_1$ and $L_2$ are translates of $\wt{C}$ passing through $\tilde{u}$ ($\tilde{e}$) in $\mathcal{B}$. Let $e_1$ correspond to the wall $\{Y_1, Y_1^c\}$, where $Y_1$ is a half-space of $L_1$. Since $L_2$ is a translate of $L_1$, there exists a fundamental domain $P'$ of $C'$ in $L_1$ containing $\tilde{u}$ ($\tilde{e}$) such that there exists $g \in G$ and $\tilde{u}'$ ($\tilde{e}'$ in $P'$) with $g\tilde{u}' = \tilde{u}$ ($g \tilde{e}' = \tilde{e}$) and $gP' \subset L_2$. The segment from $\tilde{u}$ to $\tilde{u}'$ ($\tilde{e}$ to $\tilde{e}'$) projects to $\phi(C)$ as a subgraph. Since $L_1$ passes through every vertex and edge in this segment, there is an edge corresponding to $\{Y_1,Y_1^c\}$ in the dual tree of the image in $\phi(C)$ of each vertex and edge of this segment. By the way $Z'_C$ was defined, this defines a unique hyperplane in $Z'_C$. Since $g\tilde{u}' = \tilde{u}$ ($g\tilde{e}' = \tilde{e}$) and $gL_1 = L_2$, the required edge incident to $\sigma_{L_2}$ is $\{gY_1, gY_1^c\}$.
\end{proof}

Let $\phi': C' \to X_s$ be a maximal cycle such that $C$ is a power of $C'$. 

\begin{lemma}\label{lemma_embeddind_C'_Z'_C}
There exists a natural embedding of $C'$ in $X_C$ such that the vertical graph that contains $C'$ is isomorphic to $C'$. Further, for a vertex $u$ (edge $e$) in $\phi(C)$, the embedded copy of $C'$ meets every black vertex of $T(u)$ ($T(e)$) exactly once. \qed
\end{lemma}

\subsection{The tubular graph of graphs \texorpdfstring{$X'$}{X'}}\label{subsection_construction_X'}

Let $\mathscr{C} = \{C_1, \cdots, C_n\}$ be the set of splitting cycles of $X$ furnished by \cref{thm_algo_G_hyp}.

\begin{remark} \label{rmk_peripheral_cycles}
It is easy to see that a vertical cycle induced by the attaching map of a tube is a splitting cycle that is not 3-repetitive. Hence each such cycle is included in $\mathscr{C}$.
\end{remark}

\begin{procedure}[Construction of $X'$] \label{procedure_construction_X'}
The tubular graph of graphs $X'$ is constructed from $X$ using the cycles in $\mathscr{C}$ as follows:

\begin{itemize}
\item Start with $X = X_0$.

\item For $1 \leq i \leq n$, check if $\phi_i : C_i \to X$ factors through a vertical cycle $\psi_i : C_i \to X_{i-1}$. If it doesn't, then declare $X_i = X_{i-1}$. Else, define $X_i$ to be the opened-up space of $X_{i-1}$ along the cycle $\psi_i : C_i \to X_{i-1}$.

\item Declare $X' = X_n$. 
\end{itemize}
\end{procedure}

\begin{lemma}
The cycle $C_i$ in $\mathscr{C}$ factors through a vertical cycle in $X_{i-1}$ if and only if for $1 \leq j \leq i$, lifts of $C_j$ and $C_i$ don't cross in $\wt{X}$. \qed
\end{lemma}

By \cref{thm_partial_converse_brady_meier}, we will assume that $X'$ is a Brady-Meier tubular graph of graphs.

\begin{theorem}\label{thm_alg_constructing_X'}
There exists an algorithm of double exponential time complexity that takes a Brady-Meier tubular graph of graphs $X$ with hyperbolic fundamental group $G$ as input and returns a homotopy equivalent Brady-Meier tubular graph of graphs whose vertical hyperplanes generate all universally elliptic subgroups of $G$ up to commensurability. 
\end{theorem}
\begin{proof}
The algorithm of \cref{thm_algo_G_hyp} takes $X$ as input and returns a set of cycles $\mathcal{C}$ that contains all universally elliptic cycles up to commensurability. Given $\mathcal{C}$, $X'$ is constructed using \cref{procedure_construction_X'}. This procedure consists of applying the algorithm of \cref{thm_algorithm_X_C} repeatedly.
The algorithm of \cref{thm_algo_G_hyp} takes double exponential time to return a finite set of cycles. The number of cycles in this set is bounded by a number of double exponential magnitude. Given this data, the algorithm of \cref{thm_algorithm_X_C} operates by taking exponential time for each cycle, and hence obtaining $X'$ costs double exponential time in the input data.
\end{proof}

\subsubsection{Structure of \texorpdfstring{$X'$}{X'}}

Note that the maps $\hat{\eta}_i : \wt{X}_i \to \wt{X}_{i-1}$ induce maps $\hat{\eta} : \wt{X}' \to \wt{X}$ and $\eta : X' \to X$.
Denote by $\Gamma'$ the underlying graph of the graph of spaces $X'$ and by $T'$ the underlying tree of the tree of spaces $\wt{X}'$. 
Let $L$ be a lift of an element $C_i$ of $\mathscr{C}$ such that $C_i$ factors through a vertical cycle in $X_{i-1}$. From \cref{lemma_embeddind_C'_Z'_C}, it follows (see \cref{fig_L'_opening_of_L}) that

\begin{lemma} \label{lemma_C_vertical_tree_stabilisers_X'}
There exists a vertical tree in $\wt{X}'$ whose stabiliser is equal to the stabiliser of $L$.
\end{lemma}

\begin{figure}
\begin{center}
\begin{tikzpicture}[scale = 0.75]
\fill [lightgray] (0,-4) rectangle (2,4);
\fill [gray] (0,-4) to (0,4) to (-1.8,4.4) to (-1.8,-3.8);
\fill [lightgray] (0,-4) to (0,4) to (-1.6,3.6) to (-1.6,-4.4);
\draw [ultra thick] (0,-4) to (0,4);
\draw (2,-4) to (2,4);
\draw (-1.6,-4.4) to (-1.6,3.6);
\draw (-1.8,-3.8) to (-1.8,4.4);
\node [below] at (0,-4) {$L'$};

\end{tikzpicture}
\end{center}
\caption{$L'$ in $\wt{X}'$ when $L$ has three half-spaces}\label{fig_L'_opening_of_L}
\end{figure}
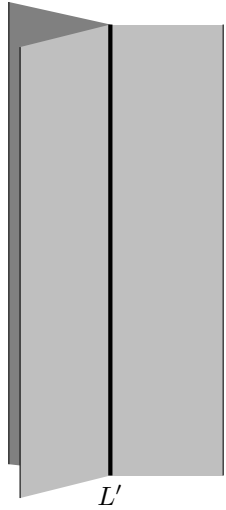

Let $\mathsf{h}$ be a vertical hyperplane in $\wt{X}'$. Let $L_1$ and $L_2$ be the two \emph{boundary lines} of $\mathsf{h}$, that is, the two vertical lines on either side of $\mathsf{h}$ at distance $\frac{1}{2}$ from $\mathsf{h}$, and parallel to $\mathsf{h}$. As $G$ is hyperbolic, it follows from the Flat Plane Theorem (see Theorem $\Gamma$.3.1 of \cite{bridsonhaefliger}) that

\begin{lemma}\label{lemma_strips_crushed_X'}
The stabiliser of $\mathsf{h}$ is equal to either stab($L_1$) or stab($L_2$). \qed
\end{lemma}

Observe that $\wt{X}'$ is the cube complex dual to a space with walls, where walls are defined on the set $Z$ of open horizontal half-edges of $\wt{X}$ along with open horizontal half-edges of strips attached to translates of $\wt{C}_i$ whenever $C_i$ factors through a vertical cycle in $X_{i-1}$. The set of walls $W'$ in $Z$ are thus of three types:
\begin{itemize}
\item \emph{Walls of type \textit{(i)}}, induced by horizontal hyperplanes of $\wt{X}$.
\item \emph{Walls of type \textit{(ii)}}, induced by vertical hyperplanes of $\wt{X}$.
\item \emph{Walls of type \textit{(iii)}}, induced by translates of $\wt{C}_i$, where $C_i \in \mathscr{C}$ factors through a vertical cycle in $X_{i-1}$.
\end{itemize} 

A \emph{vertical half-space} of $\wt{X}$ (or $\wt{X}'$) is a half-space of a wall of type \textit{(ii)} or \textit{(iii)}.
Let $\mathsf{h}$ be a vertical hyperplane in $\wt{X}'$ and $L_1$, $L_2$ its boundary lines. Then
\begin{lemma}
Either the vertical tree containing $L_1$ or the vertical tree containing $L_2$ is a line.
\end{lemma}
\begin{proof}

Let $\{Y,Y^c\}$ be the wall of type (\textit{ii}) or type (\textit{iii}) in $Z$ corresponding to $\mathsf{h}$. 
Let $L$ be the (boundary) line in $\wt{X}$ (of the vertical hyperplane) that defines $\{Y,Y^c\}$.
By \cref{lemma_C_vertical_tree_stabilisers_X'}, there exists a linear vertical tree $L'$ in $\wt{X}'$ such that stab($L$) $=$ stab($L')$. 
Note that any vertical half-space of $\wt{X}$ contained in a vertex (ultrafilter) of $L'$ contains either $Y$ or $Y^c$, by \cref{lemma_non-crossing_lines_and_half-spaces}. Thus any vertical half-space of $\wt{X}'$ (except perhaps a half-space of $L'$) that contains $L'$ contains $\mathsf{h}$. So no vertical tree separates $L'$ and $\mathsf{h}$. Hence the result.
\end{proof}

\begin{defn}
A vertex of a $G$-tree is a \emph{cyclic vertex} if its stabiliser is a cyclic subgroup of $G$.
\end{defn}

Thus in the underlying tree $T'$ of $\wt{X}'$, at least one of the two vertices of any edge is a cyclic vertex.

\begin{lemma}\label{lemma_type_1_lines_in_X'}
Let $L'$ be a line in $\wt{X}'$. Suppose that $\partial N(L')$ contains at least three components. Then the vertical tree containing $L'$ is equal to $L'$.
\end{lemma}

We need two observations to prove the lemma.
Let $L$ be a line in $\wt{X}$ that defines a wall of type (\textit{iii}). Suppose that the number of half-spaces of $L$ is $K$. Let $L'$ be a vertical tree in $\wt{X}'$ such that $\hat{\eta}(L') = L$ (\cref{lemma_C_vertical_tree_stabilisers_X'}).

\begin{lemma}\label{lemma_L'_multipod}
Exactly $K$ vertical strips are attached to $L'$ in $\wt{X}'$. Further, if $\hat{\eta}(L'') = L$ for any vertical line $L''$, then $L''$ is contained in one of these $K$ strips. 
\end{lemma}
\begin{proof}
The fact that exactly $K$ strips are attached to $L'$ follows from \cref{lemma_embeddind_C'_Z'_C}.
Further, each of the $K$ strips above are contained in $\hat{\eta}^{-1}(L)$, by \cref{lemma_opening__cofinite}.

Let $L''$ be a vertical line such that $\hat{\eta}(L'') = L$. 
Denote by $\sigma'_{L''}$ the set of vertical half-spaces contained in any vertex (ultrafilter) of $L''$. Note that $\hat{\eta}(L'') = L$ implies that the vertical half-spaces of type (\textit{ii}) in $\sigma'_{L''}$ consists of the half-spaces of vertical hyperplanes in $\wt{X}$ that contain $L$. 

If $\sigma'_{L''} = \sigma'_{L'}$, we have nothing to prove as the vertical tree that contains $L'$ is equal to $L'$.
Let $\sigma'_i$ be the set of vertical half-spaces such that $\sigma'_i \triangle \sigma'_{L'} = \{Y_i, Y_i^c\}$, where $Y_i$ is a half-space of $L$. 
Either $\sigma'_{L''} = \sigma'_i$ for some $i$ or $\sigma'_{L''} \neq \sigma'_i$ for any $i$. 

First assume the latter. Then there exists a half-space $Y_0$ of a line $L_0$ such that $\{Y_0, Y_0^c, Y_i, Y_i^c\} \subset \sigma'_{L''} \triangle \sigma'_{L'}$, with $Y_i \subset \sigma'_{L''}$. 
Let $\sigma$ be a vertex in $L \setminus L_0$. Then there exists $\sigma'' \in L''$ such that $\hat{\eta}(\sigma'') = \sigma$. This implies that the vertex $\sigma$ contains $Y_i$, which is not possible.
Assume now that $\sigma'_{L''} = \sigma'_i$, for some $i$. Let $L'_i \neq L'$ be a boundary line of the strip that separates $L''$ from $L'$. The result follows from the following observation.
Let $\gamma$ denote a geodesic between $L''$ and $L'_i$. Since $\gamma$ consists of vertical edges, $\hat{\eta}(\gamma)$ has the same length as $\gamma$ (by \cref{lemma_eta_C_properties}) and is a geodesic between $\hat{\eta}(L'')$ and $\hat{\eta}(L'_i)$.
\end{proof}

Let $L'_1, \cdots, L'_K$ be the boundary lines of strips attached to $L'$ such that $L'_i \neq L'$. Then

\begin{lemma}\label{lemma_boundary_lines_multipod}
$\partial N(L'_i)$ has exactly two components.
\end{lemma}
\begin{proof}
Note that $L'_i$ is a separating line (\cref{fact_tubular_lines_separate}).
Let $Y'_i$ be a half-space of $L'_i$ that does not contain $L'$. Then $\hat{\eta}(\mathring{Y}'_i)$ does not contain $L$ (\cref{lemma_L'_multipod}) and is connected. Thus $\hat{\eta}(Y'_i)$ lies in the half-space $Y_i$ of $L$. Further, if there exist two half-spaces of $L'_i$ that do not contain $L'$, then $\hat{\eta}^{-1}(Y_i \setminus L)$ contains these half-spaces. By \cref{lemma_opening__cofinite}, one of these half-spaces is at finite distance from $L'_i$, contradicting \cref{lemma_half-spaces_deep}.
\end{proof}

\begin{proof}[Proof of \cref{lemma_type_1_lines_in_X'}]
If $\partial N(L')$ contains three or more components, then a subgroup $H$ of stab($L'$) is universally elliptic, by \cref{prop_C_3_half-spaces_universal_elliptic}.
Let $L = \hat{\eta}(L')$. Then $H$ stabilises $L$ as $\hat{\eta}$ is $G$-equivariant, and $L$ defines a wall of type (\textit{iii}). \cref{lemma_L'_multipod} and \cref{lemma_boundary_lines_multipod} then give the result.
\end{proof}

\subsection{Modification of \texorpdfstring{$X'$}{X'}}\label{subsection_construction_X''}
The next step in the modification of $X$ to $X_{jsj}$ is the construction of an intermediate
tubular graph of graphs $X''$ from $X'$. 

\subsubsection{Construction of \texorpdfstring{$X''$}{X''}}
Remove an open tube of $X'$ if both the vertex graphs bounding the tube are circles, and then identify the vertex graphs. This is possible as, by \cref{lemma_strips_crushed_X'}, one of the attaching maps of such a tube is an isomorphism of graphs. 
Successively remove all such tubes of $X'$. Let $X''$ be the tubular graph of graphs obtained after removing all such tubes. 

\subsubsection{Structure of \texorpdfstring{$X''$}{X''}}

Let $T''$ be the underlying tree of $\wt{X}''$ and let $T_{jsj}$ denote the Bass-Serre tree of the canonical JSJ decomposition of $G$. 

\begin{lemma}\label{lemma_cyclic_vertices_of_jsj}
For each cyclic vertex $u$ of $T_{jsj}$ there exists a cyclic vertex $v$ of $T''$ such that stab($u$) $=$ stab($v$).
\end{lemma}
\begin{proof}
Fix an axis $L$ in $\wt{X}$ of stab($u$). Note that stab($u$) $=$ stab($L$) as stab($u$) is maximal cyclic.
We can assume $L$ is vertical since stab($u$) is commensurable with a universally elliptic subgroup (\cref{lemma_transversal_not_universally_elliptic}).
By \cref{lemma_alg_splittings_commensurable_geom_splittings}, there exists a splitting cycle $C$ in $X$ such that $\pi_1(C)$ is commensurable with a conjugate of stab($u$). Hence $C \in \mathscr{C}$ and there exists a vertical tree in $\wt{X}'$ whose stabiliser is stab($\wt{C})$. 
Hence the result. 
\end{proof}

Let $v$ be a cyclic vertex of $T''$ and $L''$ the corresponding vertical tree (line) in $\wt{X}''$.
Denote by $\Lambda H$ the limit set in $\partial G$ of a subgroup $H$ of $G$.

\begin{lemma}\label{lemma_edges_incident_cyclic_vertex}
The number of components of $\partial G \setminus \Lambda$stab($v$) is equal to the number of edges incident to $v$.
\end{lemma}
\begin{proof}
The number of edges incident to $v$ is equal to the number of strips attached to $L''$, which is equal to the number of components of $\partial N(L'')$. The number of components of $\partial G \setminus \Lambda$stab($v$) is equal to the supremum of the number of components of $\wt{X}'' \setminus L''^{+k}$, where $k \in \mathbb{N}$.
Let $K$ be the number of strips attached to $L''$. Let $L''_i$ be the boundary line of the $i^{th}$ strip such that $L'' \neq L''_i$. Note that the vertical tree containing $L''_i$ is not equal to $L''_i$ as otherwise the corresponding strip would have been removed to obtain $\wt{X}''$. By construction, $L''_i$ has exactly two half-spaces. Let $Y''_i$ be the half-space of $L''_i$ that does not contain $L''$. Note that no strip of $Y''_i$ contains $L''_i$.
By \cref{lemma_half-space_coarsely_connected_lines}, $Y''_i \setminus L''^{+k}_i$ and hence $Y''_i \setminus L''^{+(k+1)}$ is connected, for every $k \in \mathbb{N}$. 
\end{proof}

\begin{prop}\label{prop_T''_refinement_jsj}
For each edge stabiliser $H$ of $T_{jsj}$, there exists an edge of $T''$ whose stabiliser is $H$.
\end{prop}

\begin{proof}
Since each edge of $T_{jsj}$ is incident to a cyclic vertex, let stab($u$) be the cyclic vertex group of $T_{jsj}$ that contains $H$. Then $H$ is the stabiliser of a component of $\partial G \setminus \Lambda $stab($u$). Let $v$ be a vertex of $T''$ such that stab($v$) $=$ stab($u$) (\cref{lemma_cyclic_vertices_of_jsj}).
Then the number of edges incident to $v$ is equal to the number of components of $\partial G \setminus \Lambda $stab($u$), by \cref{lemma_edges_incident_cyclic_vertex}. Further, $H$ is the stabiliser of an edge incident to $v$ as each edge incident to $v$ induces a unique component of $\partial G \setminus \Lambda $stab($u$).
\end{proof}

\begin{defn}[\cite{guirardel_levitt_jsj}]\label{def_refinement}
A $G$-tree $\hat{T}$ is a \emph{refinement} of a $G$-tree $T$ if there exists a $G$-equivariant map $p : \hat{T} \to T$ such that $p$ sends any segment $[x,y]$ in $\hat{T}$ onto the segment $[p(x),p(y)]$. In other words, $\hat{T}$ is obtained by blowing up vertices of $T$.
\end{defn}

By \cref{prop_T''_refinement_jsj} and the properties of the JSJ decomposition (\cref{defn_jsj}), we have
\begin{cor} \label{cor_T''_refinement}
$T''$ is a refinement of $T_{jsj}$. \qed
\end{cor}

\begin{lemma}\label{lemma_surface_edges_T''}
The stabiliser $H$ of an edge $e$ of $T''$ is not an edge stabiliser of $T_{jsj}$ if and only if the cyclic vertex $u$ incident to $e$ in $T''$ is of valence two and both the non-cyclic vertices adjacent to $u$ are stabilised by hanging surface groups.
\end{lemma}

\begin{proof}
No edge of $T_{jsj}$ is such that the cyclic vertex incident to this edge is adjacent to exactly two hanging surface group vertices, by definition.
The converse follows from \cref{lemma_edges_incident_cyclic_vertex} and \cref{cor_T''_refinement}.
\end{proof}

So we can modify $X''$ to $X_{jsj}$ by removing tubes which connect hanging surface groups. This requires an identification of such groups, which is done in the next subsection.

\subsection{Surface graphs}

\begin{defn}
A vertex graph of a tubular graph of graphs is a \emph{surface graph} if the graph is not a circle and the fundamental group of the graph is a surface group whose peripheral subgroups are precisely the subgroups induced by the incident edge graphs.
\end{defn}

Thus a vertex graph is a surface graph if its fundamental group is a hanging surface group. Recall that 

\begin{defn}\label{defn_double}
The \emph{double} of a graph $\Gamma$ with a finite family of immersed cycles $\{C_1, \cdots, C_n\}$ is a tubular graph of graphs whose underlying graph consists of two vertices with $n$ edges between them, each vertex space is a copy of $\Gamma$ and the $i^{th}$ tube attaches as $C_i$ on both sides.
\end{defn}

\begin{lemma} \label{lemma_surface_graphs}
A vertex graph of a Brady-Meier tubular graph of graphs is a surface graph if and only if every edge of its double is of thickness two.
\end{lemma}

\begin{proof}
Let $D_s$ be the double of the vertex graph $X_s$ with incident edge cycles. 
It is a standard fact that $D_s$ is homeomorphic to a surface if and only if $X_s$ is a surface graph.
Note that $D_s$ is Brady-Meier as every vertex of $X_s$ satisfies the Brady-Meier conditions. Thus every edge of $D_s$ is of thickness at least two. If each edge is of thickness two, then the fact that every vertex link is connected implies that every vertex link is a circle. This implies that $D_s$ is homeomorphic to a closed surface and we are done.

Conversely, suppose that there exists an edge $e$ of thickness at least three in $D_s$. Let $\tilde{e}$ be a lift of $e$ in $\wt{D}_S$ and $\mathsf{h}$ the horizontal hyperplane through $\tilde{e}$. Note that $\mathsf{h}$ is a tree. Let $L$ be a line in $\mathsf{h}$ passing through the midpoint $m$ of $\tilde{e}$. Note that $L$ does not separate $\partial N(m)$ as $\tilde{e}$ is of thickness at least three. By \cref{lemma_L_separates_partial_N(P)}, $L$ does not separate $\wt{D}_s$. But this implies that $D_s$ is not homeomorphic to a closed surface.
\end{proof}

\subsection{Construction of \texorpdfstring{$X_{jsj}$}{Xjsj}}\label{subsection_construction_X_jsj}
We are now ready to construct $X_{jsj}$. 
Remove from $X''$ every cyclic vertex graph and the (open) tubes attached to it whenever exactly two tubes are attached to the vertex graph and both the tubes are attached to surface graphs on the other side. Call the resulting complex as $X'''$.
Denote by $T'''$ the underlying tree of $\wt{X}'''$. By \cref{lemma_surface_edges_T''}, we have

\begin{prop}\label{prop_T'''_is_jsj}
$T'''$ is isomorphic to $T_{jsj}$ as $G$-trees. \qed 
\end{prop}

The proposition proves that $X'''$ is the required $X_{jsj}$.
We now have the main result of the article:

\begin{theorem}\label{thm_main_result_jsj_hyp}
There exists an algorithm of double exponential time complexity that takes a Brady-Meier tubular graph of graphs $X$ with hyperbolic fundamental group $G$ as input and returns a Brady-Meier tubular graph of graphs whose underlying graph of groups structure is the JSJ decomposition of $G$.
\end{theorem}
\begin{proof}
Using \cref{thm_alg_constructing_X'}, we obtain the tubular graph of graphs $X'$ in double exponential time. Constructing $X''$ from $X'$ involves identifying which tubes are attached to only cyclic vertex graphs and takes at most polynomial time in the number of square of $X$. 
The construction of $X_{jsj}$ from $X''$ involves removing pairs of tubes adjacent to surface graphs. Detecting surface graphs involves constructing doubles of vertex graphs (\cref{lemma_surface_graphs}) and also takes at most polynomial time in the number of squares of $X$.
\end{proof}

\section{Relative JSJ decompositions} \label{section_relative_jsj}

Let $F$ be a finite rank free group and $\mathcal{H}$ be a finite family of maximal cyclic subgroups in $F$. Recall that $F$ \emph{splits relative to $\mathcal{H}$} if there exists a nontrivial splitting of $F$ in which each element of $\mathcal{H}$ is elliptic.
Similarly $F$ is \emph{freely indecomposable relative to $\mathcal{H}$} if $F$ does not split freely relative to $\mathcal{H}$.

\begin{defn}
A \emph{relative JSJ decomposition} of $(F,\mathcal{H})$ is a graph of groups splitting of $F$ relative to $\mathcal{H}$ which satisfies the conditions of a JSJ decomposition (\cref{defn_jsj}), with the constraints that vertex groups of type (3) are rigid relative to $\mathcal{H}$ and elements of $\mathcal{H}$ can intersect vertex groups of type (2) only in their peripheral subgroups.
\end{defn}

Recall that a subgroup $F'$ of $F$ is \emph{rigid relative to $\mathcal{H}$} if $F'$ is elliptic in every splitting of $F$ relative to $\mathcal{H}$.

\begin{theorem}[Theorem 4.25,\cite{cashen_relative_jsj}]
Given a finite rank free group $F$ and a finite family $\mathcal{H}$ of maximal cyclic subgroups of $F$ such that $F$ is freely indecomposable relative to $\mathcal{H}$, a relative JSJ decomposition of $(F,\mathcal{H})$ exists and is unique.
\end{theorem}

In \cite{cashen_manning}, the authors implement an algorithm that returns the JSJ decomposition of $F$ relative to $\mathcal{H}$, though they do not give an estimate of its time-complexity. The main result of this section is the following.

\begin{theorem}\label{thm_relative_jsj_free_group}
There exists an algorithm of double exponential time complexity that takes a finite rank free group $F$ and a finite family of maximal cyclic subgroups $\mathcal{H}$ such that $F$ is freely indecomposable relative to $\mathcal{H}$ as input and returns the relative JSJ decomposition of $F$ relative to $\mathcal{H}$.
\end{theorem}

We will construct a suitable tubular graph of graphs $X_{F,\mathcal{H}}$ to prove \cref{thm_relative_jsj_free_group}.
There exists a \emph{central vertex graph} $X_{s_c}$ in $X_{F,\mathcal{H}}$ such that $\pi_1(X_{s_c}) = F$. If $\mathcal{H} = \{H_1, \cdots, H_n\}$, then for each $H_i$ there exists an immersed cycle $\phi_i : C_i \to X_{s_c}$ such that $C_i$ induces a conjugate of the group $H_i$ in $\pi_1(X_{s_c}) = F$. Note that the word generated by $C_i$ is cyclically reduced in $F$ as $\phi_i$ is an immersion of graphs.
There exist exactly $n$ tubes in $X_{F,\mathcal{H}}$ that are attached to $X_{s_c}$ in the following way. The edge graph of the $i^{th}$ tube is isomorphic to $C_i$ and the attaching map is given by $\phi_i$. We subdivide $X_{s_c}$ and the $n$ edge graphs sufficiently to make all graphs simplicial.
The other end of the $i^{th}$ tube is attached by an isomorphism to a circular vertex graph $X_i$. There are exactly two other tubes attached to $X_i$, with both attaching maps being isomorphisms. Each of these two tubes connects $X_i$ to a copy of a surface graph whose fundamental group is the fundamental group of the oriented surface of genus two with exactly one boundary component. 
Thus the underlying graph of $X_{F,\mathcal{H}}$ is a tree with one `central' vertex $s_c$ of valence $n$, $n$ cyclic vertices adjacent to $s_c$, of valence three each, and $2n$ surface vertices of valence one each.

Let $G$ be the fundamental group of $X_{F,\mathcal{H}}$. Since each vertex group is freely indecomposable relative to its incident edge groups, $G$ is one-ended, by Theorem 18 of \cite{wilton_one-ended_groups}. Hence, we can assume that $X_{F,\mathcal{H}}$ is Brady-Meier, by \cref{thm_partial_converse_brady_meier}.

As a consequence of the Bestvina-Feighn Combination Theorem \cite{bestvina_feighn_combination}, we have
\begin{lemma}
$G$ is $\delta$-hyperbolic. \qed
\end{lemma}

Let $G_s$ be a vertex group of the graph of groups structure of $G$ induced by $X_{F,\mathcal{H}}$. 

\begin{lemma}
Either $G_s = G_{s_c}$ or $G_s$ is a conjugate of either a cyclic vertex group of the JSJ decomposition of $G$ or a maximal hanging surface group.
\end{lemma}
\begin{proof}
Let $G_s$ be a cyclic vertex group adjacent to $G_{s_c}$, with $X_s$ the corresponding vertex graph. Note that $\wt{X}_s$ is a line and $\partial N(\wt{X}_s)$ contains three components as there are three tubes attached to $X_s$. By \cref{prop_C_3_half-spaces_universal_elliptic}, a subgroup of $G_s$ is universally elliptic. 
The result then follows as $G_s$ is maximal cyclic.

Now suppose that $G_s$ is a hanging surface group corresponding to the surface graph $X_s$. Note that no separating transversal line can meet $\wt{X}_s$ as such a line will have to first cross a vertical tree (line) adjacent to $\wt{X}_s$, which is not possible, as seen above. Hence any line which crosses a line contained in $\wt{X}_s$ is itself contained in $\wt{X}_s$. Thus $G_s$ is maximal hanging.
\end{proof}

\begin{cor}
If $T_{jsj}$ is the JSJ tree of $G$ and $T$ the underlying tree of $\wt{X}_{F,\mathcal{H}}$, then $T_{jsj}$ is a refinement of $T$ obtained by blowing up lifts of the central vertex $s_c$. \qed
\end{cor}

\begin{proof}[Proof of \cref{thm_relative_jsj_free_group}]
Given $(F,\mathcal{H})$, the tubular graph of graphs $X_{F,\mathcal{H}}$ can be constructed algorithmically in polynomial time in the rank of $F$ and the lengths of $\mathcal{H}$. 
Let $X_{jsj}$ be the tubular graph of graphs obtained from $X_{F,\mathcal{H}}$ in double exponential time by \cref{thm_main_result_jsj_hyp}. 
Let $\Gamma_{jsj}$ and $\Gamma$ be the underlying graphs of $X_{jsj}$ and $X_{F,\mathcal{H}}$ respectively. 
Note that since vertex graphs other than $X_{s_c}$ induce vertex groups of the JSJ, $\Gamma_{jsj}$ is obtained from $\Gamma$ by a blow-up of the vertex $s_c$. 
Let $Y$ be the subgraph of groups of $X_{jsj}$ with underlying graph $f^{-1}(s_c)$. Then it is straightforward to check that $Y$ is the relative JSJ of $(F,\mathcal{H})$. 
\end{proof}

\section{The case of graphs of free groups with cyclic edge groups}
We can now extend our result to general hyperbolic one-ended graphs of free groups with cyclic edge groups:

\begin{theorem}\label{thm_algorithm_gen_case}
There exists an algorithm of double exponential time complexity that takes a graph of free groups with cyclic edge groups with hyperbolic one-ended fundamental group as input and returns its JSJ decomposition.
\end{theorem}

We recall that a graph of free groups with cyclic edge groups is the group-theoretic analogue of a graph of spaces (see \cref{defn_graph_spaces}) whose vertex and edge spaces are finite connected one dimensional CW complexes, with the additional restriction that all the edge spaces are circles. We will further assume that the fundamental group of the geometric realisation is one-ended and hyperbolic. Again, we will use the same notation $X$ to denote both the graph of groups under consideration and its geometric realisation. We caution the reader that in this section $X$ need not be a tubular graph of graphs.

\begin{remark}
Note that each vertex space with its incident edge spaces gives rise to a free group with a `marked' family of cyclic subgroups at the level of fundamental groups. This marked family may contain non-maximal cyclic subgroups, but this can be rectified by a normalisation process as done in Section 2.4 of \cite{cashen_relative_jsj}. The normalisation process introduces a new vertex group $H$ corresponding to each non-maximal cyclic subgroup $H'$ in the marked family such that $H$ is the maximal cyclic subgroup containing $H$. 
So without loss of generality, we will assume that we have a graph of free groups with cyclic edge groups such that whenever a vertex group is not cyclic, then the incident edge groups are all maximal cyclic. Since every cyclic subgroup of a (torsion-free) hyperbolic group is contained in a unique maximal cyclic subgroup (see Proposition 7.1 of \cite[chapitre 10]{coornaert_delzant_papadopoulos}), we note that any edge group injects into at most one of its vertex groups in a non-maximal cyclic subgroup. Thus, the normalisation procedure does not produce two new adjacent cyclic vertices.
\end{remark}

\begin{proof}[Proof of \cref{thm_algorithm_gen_case}] Since the fundamental group $G$ of the input graph of groups $X$ is one-ended, each non-cyclic vertex group with its incident edge groups gives a free group $F$ with a finite family of maximal cyclic subgroups $\mathcal{H}$ such that $F$ is freely indecomposable relative to $\mathcal{H}$. We now apply the algorithm of \cref{thm_relative_jsj_free_group} to obtain the relative JSJ decomposition of $(F,\mathcal{H})$. Replace the vertex group $F$ in $X$ with the graph of groups corresponding to the relative JSJ decomposition of $(F,\mathcal{H})$. Repeat the procedure at each non-cyclic vertex group of $X$ to obtain a new graph of groups decomposition $X'$, where each vertex group has been replaced by its relative JSJ decomposition. Observe that $X'$ is still a graph of free groups with cyclic edge groups. We now modify $X'$ to an intermediate graph of groups $X''$ by removing edges when the corresponding edge groups are attached to cyclic vertex groups on both sides, analogous to the procedure in \cref{subsection_construction_X''}. What remains is to identify surface graphs in $X''$ using \cref{lemma_surface_graphs} and gluing them together to obtain $X_{jsj}$, as done in \cref{subsection_construction_X_jsj}. The resulting graph of groups is the JSJ decomposition of $X$.
\end{proof}

\section{The isomorphism problem}

An important consequence of \cref{thm_algorithm_gen_case} is that the isomorphism problem for hyperbolic fundamental groups of graphs of free groups with cyclic edge groups is reduced to the Whitehead problem (\cite{whitehead}) and can be solved in double exponential time.

\begin{theorem}\label{thm_isomorphism}
There exists an algorithm of double exponential time complexity that takes two graphs of free groups with cyclic edge groups and hyperbolic fundamental group as input and decides whether they are isomorphic.
\end{theorem}

We refer the reader to Section 4 of \cite{dahmani_touikan_jsj_algorithm} for the relevant definitions. The main result that we invoke from their paper is Proposition 4.4 which states that given two graphs of groups $X_1$ and $X_2$ on the same underlying graph, and a collection of group isomorphisms between the corresponding vertex groups, $X_1$ is isomorphic to $X_2$ as graphs of groups if and only if there exists an `extension adjustment' \cite[Definition 4.3]{dahmani_touikan_jsj_algorithm}.

\begin{proof}[Proof of \cref{thm_isomorphism}]
Let $X_1$ and $X_2$ be the input graphs of free groups with cyclic edge groups. We will denote their (hyperbolic) fundamental groups by $G_1$ and $G_2$ respectively.

By Theorem 18 of \cite{wilton_one-ended_groups}, $G_i$ is one-ended if and only if each vertex group of $X_i$ is freely indecomposable relative to its incident edge groups. Corollary E of \cite{me_grushko} gives a polynomial time algorithm to detect whether a given vertex group is freely indecomposable relative to its incident edge groups. One then obtains the Grushko decomposition of $G_i$. Since the Grushko decomposition of a group is unique, $G_1$ and $G_2$ are isomorphic if and only if there is a one-to-one correspondence between the factors of their Grushko decompositions such that the corresponding factors are isomorphic as groups.

So assume that $G_1$ and $G_2$ are one-ended. Using \cref{thm_algorithm_gen_case}, we can construct the JSJ decompositions of $G_1$ and $G_2$. Then by the uniqueness of the JSJ decomposition and Proposition 4.4 of \cite{dahmani_touikan_jsj_algorithm}, $G_1$ and $G_2$ are isomorphic if and only if the following are satisfied:
\begin{enumerate}
\item their JSJ decompositions have a common underlying graph $\Gamma$,
\item for each vertex $v$ of $\Gamma$, there exists an isomorphism $\phi_v$ between the vertex group $G_{v,1}$ of $X_1$ and the vertex group $G_{v,2}$ of $X_2$, and
\item an extension adjustment.
\end{enumerate}
In our case, (3) boils down, for each vertex $v$, to $\phi_v$ taking the set of incident edge subgroups of $G_{v,1}$ to the set of incident edge subgroups of $G_{v,2}$. We refer the reader to Definition 4.3 and commutative diagram (8) of \cite{dahmani_touikan_jsj_algorithm} for a precise formulation.

Thus, the isomorphism problem is reduced to solving the Whitehead problem for each vertex group. There exist algorithms that provide solutions to the Whitehead problem in at most exponential time (see \cite{lee_whitehead}, \cite{kapovich_schupp_whitehead}, \cite{ventura_whitehead}). Combining the above gives us the required algorithm.
\end{proof}

\bibliographystyle{alpha}
\bibliography{Immersed_cycles_and_the_JSJ_decomposition}

\end{document}